%% file: arxiv_submission.tex
\documentclass[11pt]{amsart}

\input{newmacros}

\title[Generic Riemann Solutions for Radiative Euler Equations]{Long-Time Behavior towards Generic Riemann Solutions for Radiative Euler Equations}
\author[Cheng]{Jeffrey Cheng}
\address[Jeffrey Cheng]
{ Department of Mathematics, \newline
The University of Texas at Austin \\
Austin, TX 78712, USA}
\email{jeffrey.cheng@utexas.edu}

\author[Eun]{Namhyun Eun}
\address[Namhyun Eun]
{ School of Mathematics, \newline
Korea Institute for Advanced Study (KIAS) \\
Seoul 02455, Republic of Korea}
\email{namhyuneun@kias.re.kr}
\newtheorem{theorem}{Theorem}[section]
\newtheorem{lemma}{Lemma}[section]

\newtheorem{proposition}{Proposition}[section]
\theoremstyle{remark}
\newtheorem{remark}{Remark}[section]

\theoremstyle{plain}
\newtheorem{definition}{Definition}[section]

\numberwithin{equation}{section}

\begin{document}
\bibliographystyle{plain}

\date{\today}



\thanks{\textbf{Acknowledgment.}
N. Eun was supported by a KIAS Individual Grant (MG109001) at Korea Institute for Advanced Study, by Samsung Science and Technology Foundation (Project No. SSTF-BA2102-01), and by a National Research Foundation of Korea (NRF) grant funded by the Korean government (MSIT) (No. RS-2026-25477288).}

\begin{abstract}
We study the time-asymptotic stability of generic Riemann solutions for the one-dimensional radiative Euler equations, where the Riemann pattern consists of a \(1\)-rarefaction wave, a \(2\)-contact discontinuity, and a \(3\)-shock wave.
We prove that, for sufficiently small wave strengths and initial perturbations, the Cauchy problem admits a unique global solution, which converges uniformly as \(t\to \infty\) to a composite wave of the \(1\)-rarefaction wave, \(2\)-viscous contact wave, and \(3\)-radiative shock profile with a time-dependent shift. As a key ingredient, we establish a comprehensive set of quantitative estimates for weak radiative shock profiles.

Our proof is based on the method of \(a\)-contraction with shifts.
In all previous applications of the \(a\)-contraction method to the study of long-time behavior, the proof has relied essentially on the parabolic structure provided by either viscosity or heat conduction.
In this work, we extend the method to a hyperbolic-elliptic system without parabolic structure.
\end{abstract}
\maketitle \centerline{\date}

\tableofcontents

\section{Introduction}
We consider the one-dimensional radiative Euler system in Lagrangian mass coordinates:
\begin{equation}\label{main}
\begin{cases}
v_t - u_x = 0, \\
u_t + p(v,\th)_x = 0, \\
\left(e+\frac{u^2}{2}\right)_t + (p(v,\th)u)_x + q_x =0, \\
-\left(\frac{q_x}{v}\right)_x + \a v q + \b (\th^4)_x =0,
\end{cases}
\end{equation}
where \(v\) is the specific volume, \(u\) is the velocity, \(p\) is the pressure, \(e\) is the internal energy, \(\th\) is the absolute temperature, and \(q\) is the radiative heat flux.
Here, \(\a\) and \(\b\) are positive constants depending only on the type of gas.
We consider an ideal polytropic gas, for which the pressure and internal energy are given by: 
\begin{equation}\label{pressure}
p(v,\th) = \frac{R\th}{v}, \qquad
e(\th) = \frac{R}{\g-1}\th,
\end{equation}
where \(R>0\) is the gas constant and \(\g>1\) is the adiabatic constant, both of which also depend only on the fluid.
The total-energy equation in \eqref{main} can be rewritten in non-divergence form as
\[
\frac{R}{\g-1}\th_t + p(v,\th) u_x + q_x =0.
\]

\vspace{2mm}
A fundamental question for \eqref{main} concerns the asymptotic behavior of solutions as \(t\to\infty\).
In this paper, we aim to show that, for solutions connecting distinct far-field states \((v_\pm,u_\pm,\theta_\pm)\), the long-time behavior is governed by the corresponding inviscid Riemann problem:
\begin{equation}\label{riemann}
\begin{cases}
v_t - u_x = 0, \\
u_t + p(v,\th)_x = 0, \\
\left(e+\frac{u^2}{2}\right)_t + (p(v,\th)u)_x = 0,
\end{cases}
\qquad
(v,u,\th)(0,x) =
\begin{cases}
(v_-,u_-,\th_-), & x<0, \\
(v_+,u_+,\th_+), & x>0.
\end{cases}
\end{equation}
Note that the characteristic speeds of the Euler system in \eqref{riemann} are 
\[
\l_1(v,\th) = -\frac{\sqrt{\g R \th}}{v}, \qquad
\l_2(v,\th) = 0, \qquad
\l_3(v,\th) = \frac{\sqrt{\g R \th}}{v}.
\]
The Riemann problem \eqref{riemann} admits self-similar solutions consisting of elementary waves: shock waves, rarefaction waves, and contact discontinuities.
We consider initial data \((v_0,u_0,\th_0)\) connecting two distinct far-field states \((v_\pm,u_\pm,\theta_\pm)\in\RR_+\times\RR\times\RR_+\), that is,
\[
(v_0(x),u_0(x),\th_0(x)) \to (v_\pm,u_\pm,\theta_\pm)
\qquad \text{as } x\to\pm\infty,
\]
such that the associated Riemann problem \eqref{riemann} admits a generic Riemann solution.
Specifically, we consider a Riemann solution which consists of a \(1\)-rarefaction wave connecting $(v_-, u_-, \theta_-)$ and $(v_*, u_*, \theta_*)$, a \(2\)-contact discontinuity connecting $(v_*, u_*, \theta_*)$ and $(v^*, u^*, \theta^*)$, and a \(3\)-shock wave connecting $(v^*,u^*,\theta^*)$ and $(v_+, u_+, \theta_+)$.
Note that across the contact discontinuity, \(u_*=u^*\) and \(p(v_*,\th_*)=p(v^*,\th^*)\eqqcolon p^*\).

\subsection{Viscous Approximation of Riemann Solutions}
We introduce the viscous counterpart of the elementary waves of the Riemann problem \eqref{riemann}.
More precisely, we define the rarefaction wave and its smooth approximation, viscous contact wave, and radiative shock wave.

The rarefaction wave \((v^r,u^r,\th^r)\) and its smooth approximation \((v^R,u^R,\th^R)\) are defined in the same way as \cite[Section 2.1]{KVW-ARMA25}. We first set \(w_- \coloneqq \l_1(v_-,\th_-)\) and \(w_* \coloneqq \l_1(v_*,\th_*)\).
Note that \(w_- < w_* <0\) holds.
Let \(w^r(t,x)\) and \(w(t,x)\) be, respectively, the entropy solution and the smooth solution to the Burgers equation: 
\[
\begin{cases}
(w^r)_t + w^r (w^r)_x = 0,\\
w^r(0,x) = \begin{cases} w_-, & x<0, \\ w_*, & x>0, \end{cases}
\end{cases}
\qquad
\begin{cases}
w_t + ww_x =0,\\
w(0,x)=\frac{w_-+w_*}{2}+\frac{w_*-w_-}{2}\tanh x.
\end{cases}
\]
The corresponding rarefaction \((v^r,u^r,\th^r)(x/t)\) and its smooth approximation \((v^R,u^R,\th^R)(t,x)\) are defined on the \(1\)-rarefaction curve from \((v_-,u_-,\th_-)\) by 
\[
\l_1(v^r(x/t),\th^r(x/t))=w^r(t,x), \qquad
\l_1(v^R(t,x),\th^R(t,x))=w(1+t,x).
\]
Both satisfy the Euler equations:
\begin{equation}\label{euler-ur}
\begin{cases}
(v^r)_t - (u^r)_x = 0, \\
(u^r)_t + p(v^r,\th^r)_x = 0, \\
\frac{R}{\g-1}(\th^r)_t + p(v^r,\th^r)(u^r)_x = 0,
\end{cases}
\qquad
\begin{cases}
(v^R)_t - (u^R)_x = 0, \\
(u^R)_t + p(v^R,\th^R)_x = 0, \\
\frac{R}{\g-1}(\th^R)_t + p(v^R,\th^R)(u^R)_x = 0,
\end{cases}
\end{equation}
where the equations are written in non-divergence form. We further define the approximate radiative heat flux associated with the smooth rarefaction wave by
\begin{equation}\label{q-rare}
q^R \coloneqq -\frac{\b}{\a v^R}((\th^R)^4)_x.
\end{equation}
This can be seen as the Lagrangian mass coordinate counterpart of the approximate radiative flux used in \cite{FRX-CMAA23}.
For the exact rarefaction wave, the analogous flux \(q^r \coloneqq -\frac{\b}{\a v^r}((\th^r)^4)_x\) satisfies \(\norm{q^r(t,\cdot)}_{L^\infty(\RR)}=\Ocal(t^{-1})\) as \(t\to\infty\).
We thus take the rarefaction component of the asymptotic radiative heat flux to be zero instead of \(q^r\).

We turn to the definition of the viscous contact wave.
Motivated by \cite{WX-JDE11,X-DCDS12}, we first consider the nonlinear diffusion equation
\begin{equation}\label{nde}
\Theta_t = \frac{4\b(\g-1)p^*}{\a R^2 \g} (\Theta^2 \Theta_x)_x
\end{equation}
and its monotone self-similar solution \(\Theta(t,x)=\Theta_C(\frac{x}{\sqrt{1+t}})\) connecting the two states \(\th_*\) and \(\th^*\):
\[
\Theta_C(-\infty) = \th_*, \qquad
\Theta_C(+\infty) = \th^*.
\]
The viscous contact wave \((v^C,u^C,\th^C)\) is defined as follows:
\begin{equation}\label{uc}
v^C \coloneqq \frac{R\Theta}{p^*}, \qquad
\th^C \coloneqq \Theta, \qquad
u^C \coloneqq u^* + \frac{4\b(\g-1)}{\a R\g}\Theta^2 \Theta_x.
\end{equation}
Moreover, the radiative heat flux component \(q^C\) is defined as follows:
\begin{equation}\label{q-cont}
q^C \coloneqq -\frac{4\b p^*}{\a R}\Theta^2 \Theta_x.
\end{equation}
It follows from the definition that \((v^C,u^C,\th^C,q^C)\) satisfies
\begin{equation}\label{euler-uc}
\begin{cases}
(v^C)_t - (u^C)_x = 0, \\
(u^C)_t + p(v^C,\th^C)_x = (u^C)_t, \\
\frac{R}{\g-1}(\th^C)_t + p^*(u^C)_x + (q^C)_x = 0, \\
\a v^C q^C + \b ((\th^C)^4)_x = 0.
\end{cases}
\end{equation}
The pressure \(p(v^C,\th^C)\) of the viscous contact wave is constant and equal to \(p^*\).

The viscous counterpart of the shock wave, referred to as the radiative shock, is a traveling wave solution of the system \eqref{main}.
That is, the radiative shock \((v^S,u^S,\th^S,q^S)(\x\coloneqq x-\s t)\) satisfies the following traveling wave ordinary differential equations:
\begin{equation}\label{VS}
\begin{cases}
-\s (v^S)' - (u^S)' = 0, \\
-\s (u^S)' + p(v^S,\th^S)' = 0, \\
-\s \frac{R}{\g-1}(\th^S)' + p(v^S,\th^S)(u^S)' + (q^S)' =0, \\
-\left(\frac{(q^S)'}{v^S}\right)' + \a v^S q^S + \b ((\th^S)^4)' =0,
\end{cases}\qquad
\left(' = \frac{d}{d\x}\right)
\end{equation}
where the traveling wave speed is given by the Rankine--Hugoniot condition:
\[
\s = \sqrt{-\frac{p(v_+,\th_+)-p(v^*,\th^*)}{v_+-v^*}} >0.
\]
Note that for the case of a \(1\)-shock, the shock speed is as follows: 
\[
\s = -\sqrt{-\frac{p(v_l,\th_l)-p(v_r,\th_r)}{v_l-v_r}} <0,
\]
where the shock connects the two states, from \((v_l,u_l,\th_l,0)\) to \((v_r,u_r,\th_r,0)\).
The existence of radiative shocks is known from earlier works.
We reprove existence and establish quantitative estimates for these profiles, with the results and relevant literature presented in Section \ref{sec:shock}.

Then, we consider the following composite wave as our viscous ansatz:
\begin{equation}\begin{aligned}\label{eq:compositewave}
(\tilde{v}, \tilde{u}, \tilde{\theta}, \tilde{q})
&\coloneqq
\Big(v^r\Big(\frac{x}{t}\Big)+v^C\Big(\frac{x}{\sqrt{1+t}}\Big)
+v^S(x-\sigma t-X(t))-v_*-v^*, \\
&\qquad
u^r\Big(\frac{x}{t}\Big)+u^C(t,x)+u^S(x-\sigma t-X(t))-u_*-u^*, \\
&\qquad
\theta^r\Big(\frac{x}{t}\Big)+\theta^C\Big(\frac{x}{\sqrt{1+t}}\Big)
+\theta^S(x-\sigma t-X(t))-\th_*-\th^*, \\
&\qquad
q^C(t,x)+q^S(x-\sigma t-X(t))\Big),
\end{aligned}\end{equation}
which is composed of the \(1\)-rarefaction wave, the viscous \(2\)-contact wave, and the \(3\)-radiative shock profile shifted by \(X(t)\) (defined in \eqref{eq:shiftdef}), connecting \((v_-,u_-,\th_-)\) to \((v_*,u_*,\th_*)\), \((v_*,u_*,\th_*)\) to \((v^*,u^*,\th^*)\), and \((v^*,u^*,\th^*)\) to \((v_+,u_+,\th_+)\), respectively.

\subsection{Main Result}
Our main theorem is as follows.

\begin{theorem}\label{thm:main}
Let \((v_-,u_-,\th_-)\in\RR_+\times\RR\times\RR_+\) be the fixed state.
There exist positive constants \(\d_0\) and \(\e_0\) such that the following holds.
Let \((v_+,u_+,\th_+)\) be any state such that the solution to the Riemann problem \eqref{riemann} consists of a \(1\)-rarefaction wave, a \(2\)-contact discontinuity, and a \(3\)-shock wave, with intermediate states \((v_*,u_*,\th_*)\) and \((v^*,u^*,\th^*)\).
Consider the corresponding composite wave defined by \eqref{eq:compositewave}.
If the wave strengths satisfy 
\[
|v_+-v^*|+|v^*-v_*|+|v_*-v_-| \le \d_0,
\]
and the initial perturbation verifies
\begin{equation}\label{eq:initpertsmall}
\sum_{\pm} \norm{\big(v_0-v_\pm, u_0-u_\pm, \theta_0-\theta_{\pm}\big)}_{L^2(\RR_\pm)}
+\norm{\big((v_0)_x,(u_0)_x,(\th_0)_x\big)}_{H^1(\RR)} < \e_0,
\end{equation}
then the Cauchy problem \eqref{main} admits a unique global-in-time solution \((v,u,\th,q)(t,x)\).\\
Moreover, there exists an absolutely continuous shift function \(X(t)\) such that
\begin{equation}\begin{aligned}\label{eq:maincontinuity}
&(v-\tilde{v}, u-\tilde{u}, \theta-\tilde{\theta}, q-\tilde{q})(t,x) \in C(0,\infty;H^1(\RR)),\\ 
&(v_{xx}, u_{xx}, \theta_{xx}, q_{xx}) \in C(0,\infty;L^2(\RR)),\\
&q_{xxx}(t,x)-(q^S)_{xxx}(x-\sigma t-X(t)) \in L^2(0,\infty;L^2(\RR)).
\end{aligned}\end{equation}
In addition, the asymptotic behavior as \(t\to \infty\) is as follows:
\begin{equation}\label{eq:pertgoesto0main}
\sup_{x \in \RR}|(v-\vtil, u-\util, \theta-\thtil, q-\qtil)(t,x)| \to 0,
\end{equation}
and 
\begin{equation}\label{eq:shiftgoesto0main}
\lim_{t \to \infty}|\dot{X}(t)|=0.
\end{equation}
\end{theorem}
\begin{remark}
\Cref{thm:main} states that if the far-field states $(v_\pm, u_\pm, \theta_\pm)$ are connected by the superposition of a $1$-rarefaction, a $2$-contact, and a $3$-shock, then the solution of the radiative Euler equations \eqref{main} converges as $t \to \infty$ to the composite wave consisting of an inviscid self-similar rarefaction, a viscous contact wave, and a radiative shock shifted by $X(t)$.
\end{remark}
\begin{remark}
The shift $X$ satisfies \eqref{eq:shiftgoesto0main}, and therefore has sublinear growth:
\[
\lim_{t \to \infty}\frac{X(t)}{t}=0.
\]
Thus, the shifted shock has the same asymptotic propagation speed \(\s\) as the original shock wave.
\end{remark}

There is a long history of works on long-time behavior for one-dimensional dissipative compressible fluid equations, particularly the Navier--Stokes equations and related models.
We refer the reader to \cite{HLM-ARMA10,HM-CMP09,HXY-ADV08,MZ-CPAM04,MN-1985,MN-1986} and the references therein.
For the radiative Euler system, stability of basic wave patterns was shown in \cite{MR2836843} (rarefactions), \cite{MR3010836,WX-JDE11} (contacts) and \cite{LCG-CRM07,NPZ-PD10} (shocks).
Stability results are also available for composite waves consisting of a viscous contact wave and two rarefaction waves \cite{X-DCDS12}, as well as for composites of two viscous shock profiles \cite{FRX-AIHPC19}.
However, these results do not cover composite waves that include both a shock and either a rarefaction wave or a contact wave, such as the pattern considered in \Cref{thm:main}.

Let us also mention the approach of Zumbrun and co-authors, which relies on obtaining detailed pointwise Green's function bounds on the linearized operator.
This has been carried out in a large variety of different scenarios (see for example \cite{MR2787588,MR2228699,MR1947862,MR2048568,MZ-CPAM04,MR4062477,MR1665788}).
We note that such analyses rely on spectral stability assumptions (related to the zeroes of the Evans function) which are usually verifiable for small-amplitude profiles.
In principle, the methods generalize to large-amplitude profiles (as long as the spectral stability hypothesis holds), and the theory is developed with a great level of generality.
Of most relevance to this paper are the applications to scalar radiative shocks \cite{MR2579710} and hyperbolic-elliptic coupled systems \cite{NPZ-PD10}.
The radiative structure we use in the $L^2$-estimate in \Cref{sec:L2} is present already in these works.
However, the pointwise Green's function approach has only been carried out for linearization around single shock profiles.
In this paper, we are able to treat composite waves, where we must deal with interaction between the radiative shock and the other waves.

Our proof relies on the method of \(a\)-contraction with shifts, developed in the viscous setting by Kang and Vasseur \cite{KV-17,KV-JEMS21}.
This method identifies a pseudo distance that captures the contractive structure of viscous shocks by combining a suitable weight with a time-dependent shift.
In contrast to the anti-derivative method, the \(a\)-contraction method provides a compatible framework for analyzing the stability of both shocks and rarefaction waves.
Using this method, Kang, Vasseur and Wang showed the time-asymptotic stability of rarefaction-shock composite waves for the barotropic Navier--Stokes equations \cite{KVW-ADV23}, and subsequently studied the superposition of a rarefaction wave, a viscous contact wave, and a viscous shock for the Navier--Stokes--Fourier system \cite{KVW-ARMA25}. 
For further applications of the \(a\)-contraction method, see \cite{HWZ-MA25} for nonconvex fluxes, \cite{HKKKO-CMP26} for interactions with boundary layers, \cite{WY-25} for the Boltzmann equation, \cite{EEK-24,KV-Inven21} for inviscid limits, and \cite{CKV-ARMA22,CFK-ARMA26} for stability of inviscid equations.

The radiative Euler equations have a hyperbolic-elliptic structure.
An elliptic coupling is also present in the Navier--Stokes--Poisson system, for which the asymptotic stability of rarefaction-shock composite waves has been established in \cite{Shim-JDE26}.
However, viscosity provides the dissipation needed to close the estimates. Thus, despite the additional elliptic coupling, the principal dissipative mechanism remains that of the Navier--Stokes equations. Another related development is the work of Guo--Li--Wu--Xu \cite{GLWX-2026} on a heat-conductive ideal gas without viscosity.
Their system shares with the radiative Euler equations the absence of momentum viscosity, but its dissipation is provided by local Fourier heat conduction.
That is, as in viscous systems, the energy estimates are closed by exploiting local parabolic dissipation. Let us finally note that Kageura and Ueda \cite{KU-JHDE25} proved the stability of shocks for the Jin--Xin relaxation model to a scalar conservation law, which contains no parabolic terms.
However, its damped-wave reformulation yields a coercive term that is essentially the same as parabolic dissipation in the weighted Poincar\'e argument.

In contrast to these results, in the radiative Euler equations, the required dissipation must be extracted from the radiative coupling itself.
The radiative flux is governed by an elliptic equation, and eliminating this flux produces a non-local and only partially dissipative operator.
Consequently, the energy estimates must be closed through a dissipative mechanism largely different from that used in \cite{GLWX-2026}.
We explain in the following subsection how the weak damping and weak diffusion generated by this radiative coupling are combined to close the energy estimates.

\subsection{Main Difficulties and New Ideas}
Since the simultaneous propagation of rarefaction, contact, and shock creates an incompatibility between the anti-derivative method (for shocks) and the energy method (for the other waves), we employ the method of \(a\)-contraction with shifts which was used to solve this exact problem in \cite{KVW-ADV23,KVW-ARMA25} by establishing shock stability without using the anti-derivative method.

An essential difficulty arises at this stage.
The \(a\)-contraction method relies on the weighted Poincar\'e inequality in Lemma \ref{lem:poincare} to close the \(L^2\) estimate for shock stability.
All past works on systems of conservation laws have exploited the parabolic structure of the Navier--Stokes equations (or heat conduction, in the case of \cite{GLWX-2026}) to allow the use of the Poincar\'e inequality. Such a parabolic structure is completely absent in the radiative Euler equations. It is well-known that the elliptic equation is intimately linked to higher-order derivative estimates for the temperature (see, e.g., \cite{FRX-DCDS21,WX-JDE11,X-DCDS12}).
However, in those works, this only appears at the level of the higher-order estimates.

In this work, we solve this problem by identifying a structure (going back to the work \cite{NPZ-PD10}) which makes the link between the perturbation of the radiative heat flux (denoted by \(r\)) and the temperature perturbation (denoted by \(\chi\)) more apparent.
Rather than deriving energy estimates from the elliptic equation for the radiative heat flux in \eqref{main}, we solve the linearized equation directly using the Green's function in \Cref{sec:chir}.
Then, we define a new variable \(\z\) (see \eqref{eq:zetadef}), which, up to small errors, is essentially the non-local recovery of \(r\) via the elliptic equation.
The Green's function method allows us to transparently convert the temperature perturbation \(\chi\) into a linear combination of \(\z\) and \(r_x\).
We then apply the weighted Poincar\'e inequality \Cref{lem:poincare} to \(\z\) and \(\z_x\) to deplete the \(\z\) component of \(\chi\), and use the damping provided by the energy structure to deplete the \(r_x\) component. 
\par We believe that this approach is advantageous to doing energy estimates on the elliptic equation for a multitude of reasons.
Firstly, it makes the relationship between \(q\) and \(\th\) more apparent than energy estimates at the anti-derivative level.
Secondly, it clarifies the stabilizing effect of the radiative part; it contains both weak damping and weak diffusion, which combine to recover stability results which are the same as those given by usual parabolic diffusion.
Third, compared to previous analyses of shock stability based on the anti-derivative method \cite{FRX-AIHPC19,LCG-CRM07}, our approach also allows us to work directly with perturbations in Sobolev spaces in the absence of the zero-mass condition and to treat a wider range of values of the adiabatic constant \(\g\).
Finally, the adaptation of the non-local radiative structure to the $a$-contraction method lets us seamlessly handle compositions of rarefactions, contacts, and shocks.

Note that \Cref{thm:main} holds for any \(\g>1\).
For the radiative Euler equations, the restriction that the parameter \(\g\) lie between $1$ and $3$ has in the past played a fundamental role (here, we are speaking for purely mathematical reasons; the value of \(\g\) of course depends on the physics of the gas under consideration).
For example, it was shown by Mascia \cite{Mascia-PD13} that sufficiently large smooth shocks may have non-monotone temperatures when \(\g\in(1,3)\) (although in this article we treat arbitrarily small shocks for every value of \(\g>1\)).
Even in the small-shock scenario, the restriction \(\g\in(1,3)\) has played a fundamental role in the time-asymptotic stability analysis in \cite{FRX-AIHPC19}, where it was used to obtain coercivity of the dissipation provided by the radiation (the restriction \(\g\in(1,2]\) also appears in \cite{LCG-CRM07}, but this is more for simplicity rather than a critical restriction).
It is interesting that no such restriction appears in the analysis presented in this paper.
Note that the small-shock stability result for general hyperbolic-elliptic systems in \cite{NPZ-PD10}, using the pointwise semigroup approach, also has no restriction on \(\g\) other than \(\g>1\) when applied to the system \eqref{main}.

\vspace{2mm}
The rest of the paper is structured as follows.
In \Cref{sec:prelims}, we give some important properties and estimates for the smooth approximate rarefaction waves, the viscous contact waves, and the radiative shock waves.
In \Cref{sec:pfofmain}, we explain how \Cref{thm:main} follows from local existence and the \textit{a priori} estimate \Cref{prop:main}.
The rest of the paper is devoted to proving \Cref{prop:main}; \Cref{sec:L2} proves the $L^2$ estimate using the $a$-contraction method, while \Cref{sec:h2} closes the $H^2$ estimate by deriving higher-order energy estimates and space-time estimates of the density, velocity, and temperature.

\vspace{2mm}
\(\bullet\) \textbf{Notation:} Throughout the paper, \(C\) and \(c\) denote generic positive constants, independent of \(\e_0\) and \(\d_0\) (and later also the related parameters $\delta_1$ and $\e_1$, defined in \Cref{prop:main}), which may vary from line to line.
We use \(C\) for upper bounds and \(c\) for lower bounds.

\section{Preliminaries}\label{sec:prelims}
\setcounter{equation}{0}

\subsection{Approximate Rarefaction Wave}
The following lemma provides useful estimates for the smooth approximate \(1\)-rarefaction wave:

\begin{lemma}\cite[Lemma 2.1]{KVW-ARMA25}\label{lem:uR}
Let \((v^R,u^R,\th^R)(t,x)\) be the smooth approximate \(1\)-rarefaction wave constructed above, and denote its strength by 
\[
\d_R \coloneqq |v_*-v_-|
\sim |u_*-u_-|
\sim |\th_*-\th_-|.
\]
Set \(\l_{1-}\coloneqq \l_1(v_-,\th_-)\) and \(\l_{1*}\coloneqq \l_1(v_*,\th_*)\).
Then, there exist positive constants \(C\) and \(C_k\), independent of \(\d_R\), such that the following properties hold:
\begin{enumerate}
\item[(1)]
For every \(t\ge 0\) and \(x\in\RR\), 
\[
u_x^R = \frac{2v^R}{\g+1}w_x >0, \qquad
v_x^R = \frac{v^R}{\sqrt{R\g\th^R}}u_x^R >0, \qquad
\th_x^R = -\frac{(\g-1)\th^R}{v^R}v_x^R <0.
\]
\item[(2)]
For every \(t\ge 0\) and \(p\in[1,\infty]\), 
\[
\norm{\big(v_x^R,u_x^R,\th_x^R\big)(t,\cdot)}_{L^p(\RR)}
\le C\min\big\{\d_R,(\d_R)^{\frac{1}{p}}(1+t)^{-1+\frac{1}{p}}\big\}.
\]
Moreover, for \(k=2,3,4,\) 
\[
\norm{\rd_x^k \big(v^R,u^R,\th^R\big)(t,\cdot)}_{L^p(\RR)}
\le C_k \min \big\{\d_R, (1+t)^{-1}\big\},
\]
and, for every \(t\ge 0\) and \(x\in\RR\), 
\[
|\rd_x^k v^R| \le C_k v_x^R, \qquad
|\rd_x^k u^R| \le C_k u_x^R, \qquad
|\rd_x^k \th^R| \le C_k |\th_x^R|.
\]
\item[(3)]
For every \(t\ge 0\) and \(x\ge \l_{1*}(1+t)\), 
\[
\big|(v^R,u^R,\th^R)(t,x)
-(v_*,u_*,\th_*)\big|
+\sum_{k=1}^4 \big|\rd_x^k\big(v^R,u^R,\th^R\big)(t,x)\big|
\le C\d_R e^{-2|x-\l_{1*}(1+t)|}.
\]
\item[(4)]
For every \(t\ge 0\) and \(x\le \l_{1-}(1+t)\), 
\[
\big|(v^R,u^R,\th^R)(t,x)
-(v_-,u_-,\th_-)\big|
+\sum_{k=1}^4 \big|\rd_x^k\big(v^R,u^R,\th^R\big)(t,x)\big|
\le C\d_R e^{-2|x-\l_{1-}(1+t)|}.
\]
\item[(5)]
The approximate wave converges uniformly to the rarefaction wave as \(t\to\infty\): 
\[
\lim_{t\to\infty} \sup_{x\in\RR}
\left|\big(v^R,u^R,\th^R\big)(t,x)
-\big(v^r,u^r,\th^r\big)\left(\frac{x}{t}\right)\right| = 0.
\]
\end{enumerate}
\end{lemma}
\begin{proof}
The above estimates for the approximate rarefaction wave follow from the characteristic representation of the solution to the Burgers equation, as in the proof of \cite[Lemma 2.1]{MN-1986} (see also \cite[Lemma 2.1]{KVW-ARMA25}).
In particular, the higher-order estimates can be derived by repeatedly differentiating the same representation.
We therefore omit the details.
\end{proof}

We now state the following lemma, which provides quantitative estimates for \(q^R\).
\begin{lemma} \label{lem:qR}
Let \(q^R\) be defined in \eqref{q-rare}.
Under the assumptions of Lemma \ref{lem:uR}, the following estimates hold, with constants independent of \(\d_R\).
\begin{enumerate}
\item[(1)]
For every \(t\ge0\) and \(x\in\RR\),
\[
q^R
=-\frac{4\b(\th^R)^3}{\a v^R} \th_x^R
=\frac{4\b(\g-1)(\th^R)^4}{\a(v^R)^2} v_x^R
=\frac{4\b(\g-1)(\th^R)^{\frac{7}{2}}}{\a\sqrt{R\g}v^R} u_x^R>0.
\]
In particular, 
\[
C^{-1} v_x^R \le q^R \le C v_x^R.
\]
\item[(2)]
For every \(t\ge 0\) and \(p\in[1,\infty]\), 
\begin{align*}
\norm{q^R(t,\cdot)}_{L^p(\RR)}
&\le C\min\big\{\d_R, (\d_R)^{\frac{1}{p}} (1+t)^{-1+\frac{1}{p}}\big\}, \\
\sum_{j=1}^3 \norm{\rd_x^j q^R(t,\cdot)}_{L^p(\RR)}
&\le C\min\big\{\d_R, (1+t)^{-1}\big\}.
\end{align*}
Moreover, for every \(t\ge 0\) and \(x\in\RR\), 
\[
\sum_{j=1}^3 \abs{\rd_x^j q^R(t,x)}
\le C q^R(t,x).
\]
Consequently, it follows that 
\[
\norm{q^R(t,\cdot)}_{H^3(\RR)}
\le C\min\big\{\d_R, (\d_R)^{\frac{1}{2}} (1+t)^{-\frac{1}{2}}\big\}.
\]
\item[(3)]
For every \(t\ge 0\) and \(x\ge \l_{1*}(1+t)\), 
\[
\sum_{j=0}^3 \abs{\rd_x^j q^R(t,x)}
\le C\d_R e^{-2|x-\l_{1*}(1+t)|}.
\]
\item[(4)]
For every \(t\ge 0\) and \(x\le \l_{1-}(1+t)\), 
\[
\sum_{j=0}^3 \abs{\rd_x^j q^R(t,x)}
\le C\d_R e^{-2|x-\l_{1-}(1+t)|}.
\]
\item[(5)] The approximate wave converges uniformly to zero as \(t\to\infty\): 
\[
\lim_{t\to\infty} \norm{q^R(t,\cdot)}_{L^\infty(\RR)} =0.
\]
\end{enumerate}
\end{lemma}
\begin{proof}
The estimates above are obtained by straightforward calculations using the definition of \(q^R\), together with the estimates for \((v^R,u^R,\th^R)\) in Lemma \ref{lem:uR}.
We therefore omit the details.
\end{proof}

\subsection{Viscous Contact Wave}
Thanks to the self-similar structure of the nonlinear diffusion equation \eqref{nde}, the viscous contact wave defined by \eqref{uc} and \eqref{q-cont} satisfies the following quantitative estimates.
\begin{lemma}\label{lem:VC}
Let \((v^C,u^C,\th^C,q^C)(t,x)\) be the viscous contact wave defined by \eqref{uc}--\eqref{q-cont}, connecting \((v_*,u_*,\th_*,0)\) and \((v^*,u^*,\th^*,0)\), with \(v_*,\th_*,v^*,\th^*>0\).
Denote its strength by 
\[
\d_C \coloneqq |v^*-v_*|
=\frac{R}{p^*}|\th^*-\th_*|.
\]
Then there exist positive constants \(c_1\) and \(C_k\), independent of \(\d_C\), such that for all \(t\ge 0\),
\begin{align*}
\left|
\big(v^C-v_*,u^C-u_*,\th^C-\th_*,q^C\big)(t,x)
\right|
&\le C_0 \d_C \exp\left(-\frac{c_1 x^2}{1+t}\right),
&& x<0,\\
\left|
\big(v^C-v^*,u^C-u^*,\th^C-\th^*,q^C\big)(t,x)
\right|
&\le C_0 \d_C \exp\left(-\frac{c_1 x^2}{1+t}\right),
&& x>0.
\end{align*}
Moreover, for all \(t\ge 0\) and \(x\in\RR\),
\begin{align}
\left|\rd_x^k(v^C,\th^C)(t,x)\right|
&\le C_k \d_C (1+t)^{-\frac{k}{2}}
\exp\left(-\frac{c_1 x^2}{1+t}\right),
&& k=1,2,\cdots, \label{c-vth-der}\\
\left|\rd_x^k(u^C-u^*,q^C)(t,x)\right|
&\le C_k\d_C(1+t)^{-\frac{k+1}{2}}
\exp\left(-\frac{c_1 x^2}{1+t}\right),
&& k=0,1,2,\cdots. \label{c-uq-der}
\end{align}
\end{lemma}

\begin{proof}
The estimates follow from standard estimates for self-similar solutions of nonlinear diffusion equations (see, for instance, \cite{HLM-ARMA10,HXY-ADV08}).
We omit the details.
\end{proof}

\subsection{Radiative Shock Wave}\label{sec:shock}
We now turn to the radiative shock profiles.
The existence of radiative shock profiles was established in \cite{LCG-PD06} under the assumption of small shock strength and \(\g < (\sqrt{7}+1)/(\sqrt{7}-1)\).
Lattanzio--Mascia--Serre \cite{LMS-IU07} subsequently proved the existence of radiative shocks for general radiative hyperbolic-elliptic systems with linear coupling and showed that the profile regularity improves as the shock strength decreases.
They later extended their analysis to general nonlinear coupling \cite{LMS-08}, including the radiative Euler equations \eqref{main}, under the small-shock-strength assumption.
We also refer to \cite{Mascia-PD13} for the analysis of small, medium, and large shocks, which may exhibit a discontinuity at at most one point.

Although the existence and uniqueness are already known, we revisit the profile construction using geometric singular perturbation theory to obtain the comprehensive set of quantitative estimates needed for our subsequent application of the method of \(a\)-contraction with shifts.
The first lemma establishes the existence and uniqueness within this framework, while the second derives the required estimates from the same construction.
To the best of our knowledge, these estimates are the most detailed and comprehensive available (for small shocks) in the literature.

\vspace{2mm}
We begin with the following lemma, which establishes the existence and uniqueness of traveling wave solutions to the radiative Euler equations, namely radiative shock profiles satisfying \eqref{VS}.
Note that in the following two lemmas, monotonicity refers to the monotonicity of \((v,u,\th)\).
This does not extend to the \(q\)-variable, as is evident from \eqref{qS}.

\begin{lemma}\label{lem:EU}
Let \((v^*,u^*,\th^*) \in \RR_+\times\RR\times\RR_+\) be the left-end state and let \(m\ge 2\) be an integer.
Then, there exists a constant \(\d_m>0\) such that, for any right-end state \((v_+,u_+,\th_+)\in \RR_+\times\RR\times\RR_+\) satisfying the following Rankine--Hugoniot and Lax entropy conditions
\begin{align*}
&\exists \, \s \quad \text{s.t.}
\begin{cases}
-\s(v_+-v^*) -(u_+-u^*) =0, \\
-\s(u_+-u^*) +p_+-p^* =0, \\
-\s\Big(\frac{R(\th_+-\th^*)}{\g-1}+\frac{(u_+)^2-(u^*)^2}{2}\Big) +p_+u_+ -p^*u^* =0,
\end{cases}\\
&\text{and either } v^* > v_+, \, u^*>u_+, \, \th^* < \th_+,
\text{ or } v^* < v_+, \, u^*>u_+, \, \th^* > \th_+ \text{ holds,}
\end{align*}
and \(\d_S \coloneqq |v^*-v_+| < \d_m\), the system \eqref{main} admits a monotone \(\Ccal^m\) traveling wave solution \((v^S,u^S,\th^S,q^S) \colon \RR \to \RR_+ \times \RR \times \RR_+ \times \RR\).
Moreover, it is unique, up to a translation, among \(\Ccal^2\) profiles with monotonicity.
\end{lemma}

The next lemma summarizes the quantitative properties of the radiative shock profiles.
These include exponential decay properties, higher-order derivative bounds, monotonicity and comparison between profile derivatives, a precise approximation for \((v^S)'\) and exact identity and derivative estimates for the radiative heat flux.

\begin{lemma}\label{lem:VS}
Let \((v^S,u^S,\th^S,q^S)\) be the monotone \(\Ccal^4\) \(3\)-radiative shock satisfying \eqref{VS} connecting \((v^*,u^*,\th^*,0)\) and \((v_+,u_+,\th_+,0)\), which is normalized by \(v^S(0)=\frac{v^*+v_+}{2}\).
Then there exists a constant \(\d_4>0\) such that for any \(\d_S\coloneqq \abs{v^*-v_+}<\d_4\), the following holds.\\
There exist positive constants \(C\), \(C_k\) and \(c\), independent of \(\d_S\), such that the following holds:
\begin{alignat}{2}
&\begin{aligned}\label{decay}
&\big|(v^S(\x)-v^*,u^S(\x)-u^*,\th^S(\x)-\th^*)\big|
\le C\d_S e^{-c \d_S |\x|},\\
&\big|(v^S(\x)-v_+,u^S(\x)-u_+,\th^S(\x)-\th_+)\big|
\le C\d_S e^{-c \d_S |\x|},
\end{aligned}
&&\qquad\begin{aligned}
&\forall \x\le0,\\
&\forall \x\ge0,
\end{aligned}\\
&\big|(v^S)'(\x), (u^S)'(\x), (\th^S)'(\x) \big|
\le C\d_S^2 e^{-c\d_S|\x|},
&&\qquad\forall \x\in\RR. \label{derivdecay}
\end{alignat}
Moreover, for \(k=2,3,4\) and \(\x\in\RR\),
\begin{equation}\label{2bound1}
\big|\rd_\x^k (v^S(\x),u^S(\x),\th^S(\x))\big|
\le C_k \d_S^{k-1} \big|(v^S)'(\x)\big|
\le C_k \d_S^{k+1} e^{-c \d_S |\x|}.
\end{equation}
The profile satisfies
\begin{equation}\label{derivsign}
(v^S)'>0, \qquad
(u^S)'<0, \qquad
(\th^S)'<0, \qquad \forall \x\in\RR,
\end{equation}
and
\[
|(v^S)'(\x)| \sim |(u^S)'(\x)| \sim |(\th^S)'(\x)|, \qquad \forall \x\in\RR.
\]
More precisely, for all \(\x\in\RR\), 
\[
(u^S)' = -\s (v^S)', \qquad
(\th^S)' = \frac{p^*-\s^2(2v^S-v^*)}{R}(v^S)',
\]
and hence
\begin{align}
\big|(u^S)'(\x)+\s^*(v^S)'(\x)\big|
&\le C\d_S |(v^S)'(\x)|, \label{ratio-vu}\\
\left|(\th^S)'(\x) + \frac{(\g-1)p^*}{R}(v^S)'(\x)\right|
&\le C\d_S |(v^S)'(\x)|. \label{ratio-vth}
\end{align}
The shock speed satisfies
\begin{equation}\label{sm1}
|\s_\d - \s^*| \le C\d_S,
\end{equation}
where \(p^*\coloneqq p(v^*,\th^*)\), \(p_+=p(v_+,\th_+)\),
\[
\s = \s_\d
\coloneqq \sqrt{-\frac{p_+-p^*}{v_+-v^*}}, \quad \text{and} \quad
\s^* \coloneqq \sqrt{\frac{\g p^*}{v^*}} = \frac{\sqrt{\g R \th^*}}{v^*}.
\]
Furthermore, for all \(\x\in\RR\), we have
\begin{equation}\label{tail-0}
\big|(v^S)'(\x) - A^*(v_+-v^S(\x))(v^S(\x)-v^*)\big|
\le C\d_S (v_+-v^S(\x))(v^S(\x)-v^*),
\end{equation}
where the constant \(A^*\) is given by
\begin{equation}\label{eq:A^*def}
A^* \coloneqq \frac{\a R \g (\g+1)\s^*}{8\b(\g-1)^2(\th^*)^3} >0.
\end{equation}
Finally, the radiative heat flux component satisfies the exact identity
\begin{equation}\label{qS}
q^S
= \frac{(\g+1)\s^3}{2(\g-1)}(v^S-v^*)(v_+-v^S)>0,
\end{equation}
and for all \(\x\in\RR\),
\begin{equation}\label{qder}
c \big|(v^S)'(\x)\big|
\le \big| q^S(\x) \big|
\le C\big|(v^S)'(\x)\big|,
\end{equation}
\begin{equation}\label{qhder}
\big|\rd_\x^k q^S(\x)\big|
\le C_k \d_S^k \big|(v^S)'(\x)\big|
\le C_k \d_S^{k+2} e^{-c\d_S|\x|}, \qquad k=0,1,2,3.
\end{equation}
\end{lemma}

\begin{remark}
We state Lemma \ref{lem:VS} only for the \(3\)-radiative shocks, since the composite wave considered here contains a \(3\)-shock.
Nevertheless, the corresponding results for the \(1\)-shocks can be obtained through the transformation \(x \mapsto -x\), \(u \mapsto -u\), \(\s_\d \mapsto -\s_\d\), \(\s^* \mapsto -\s^*\), and \(q \mapsto -q\), with the left and right end states interchanged and relabeled accordingly.
The only differences arise from the sign changes in the derivatives of the shocks and the traveling wave speed:
\[
(v^S)'<0, \qquad
(u^S)'<0, \qquad
(\th^S)'>0, \qquad \forall \x\in\RR,
\]
and 
\[
\s = \s_\d
\coloneqq -\sqrt{-\frac{p_+-p^*}{v_+-v^*}}, \quad \text{and} \quad
\s^* \coloneqq -\sqrt{\frac{\g p^*}{v^*}} = -\frac{\sqrt{\g R \th^*}}{v^*},
\]
which lead to corresponding sign changes in \(A^*\) and \(q^S\). 
\[
A^* = \frac{\a R \g (\g+1)\s^*}{8\b(\g-1)^2(\th^*)^3} <0, \qquad
q^S
= \frac{(\g+1)\s^3}{2(\g-1)}(v^S-v^*)(v_+-v^S)<0.
\]
\end{remark}

The proofs of the above two lemmas are provided in Appendix \ref{appendix:VS}.

\subsection{Weighted Poincar\'e Inequality}
One of the key ingredients in the method of \(a\)-contraction with shifts is the following weighted Poincar\'e inequality:
\begin{lemma}\label{lem:poincare} \cite[Lemma 2.9]{KV-JEMS21}
For any \(f\colon [0,1]\to\RR\) satisfying \(\int_0^1 y(1-y)|f'|^2 dy < \infty\), 
\[
\int_0^1 \abs{f-\int_0^1 f dy}^2 dy
\le \frac{1}{2} \int_0^1 y(1-y)\abs{f'}^2 dy.
\]
\end{lemma}

\section{Proof of Theorem \ref{thm:main}}\label{sec:pfofmain}
\setcounter{equation}{0}
To begin with, we introduce the following change of variable: 
\[
(t,x) \mapsto (t, \x = x-\s t),
\]
and thus all the subsequent time derivatives are taken at fixed \(\x\).\\
Then, the radiative Euler equations \eqref{main} are given by
\begin{equation}\label{main-x}
\begin{cases}
v_t - \s v_\x - u_\x = 0, \\
u_t - \s u_\x + p(v,\th)_\x = 0, \\
\frac{R}{\g-1}(\th_t -\s \th_\x) + p(v,\th) u_\x + q_\x =0, \\
-\left(\frac{q_\x}{v}\right)_\x + \a v q + \b (\th^4)_\x =0.
\end{cases}
\end{equation}
We write \(U=(v,u,\th)\), with analogous notation for the wave profiles and constant states.\\
The composite wave is then given by
\begin{equation}\label{Ubar}
\Ubar = U^R(t,\x+\s t) + U^C(t,\x+\s t) + U^S(\x-X(t)) - U_* - U^*,
\end{equation}
and 
\begin{equation}\label{qbar}
\qbar = q^R(t,\x+\s t) + q^C(t,\x+\s t) + q^S(\x-X(t)),
\end{equation}
where \(X\) is the shift function which will be determined later.
We recall that \(U^R\) connects \((v_-,u_-,\th_-)\) and \((v_*,u_*,\th_*)\), \(U^C\) connects \((v_*,u_*,\th_*)\) and \((v^*,u^*,\th^*)\), and \(U^S\) connects \((v^*,u^*,\th^*)\) and \((v_+,u_+,\th_+)\), and all the strengths are assumed to be small.

In what follows, unless otherwise stated, \(U^R\) and \(U^C\) are evaluated at \((t,\x+\s t)\), while \(U^S\) is evaluated at \(\x-X(t)\); the same convention applies to \(q^R\), \(q^C\) and \(q^S\).
Thus, for example, \((v^S)_\x = (v^S)'(\x-X(t))\) and \(\rd_t (v^S) = -\dot{X}(v^S)_\x\).
We then introduce the following notation:
\begin{align}\label{eq:pertdefs}
\phi \coloneqq v-\vbar, \qquad
\psi \coloneqq u-\ubar, \qquad
\chi \coloneqq \th-\thbar, \qquad
r \coloneqq q-\qbar,
\end{align}
and 
\[
\pbar \coloneqq p(\vbar,\thbar), \qquad
p^R \coloneqq p(v^R,\th^R), \qquad
p^C \coloneqq p(v^C,\th^C), \qquad
p^S \coloneqq p(v^S, \th^S).
\]

The residuals are grouped as follows:
\begin{align}\label{eq:Qidefs}
Q_1 = Q_1^C + Q_1^I, \quad
Q_2 = Q_2^R + Q_2^I, \quad
Q_3 = Q_3^R + Q_3^C + Q_3^I,
\end{align}
where 
\begin{equation}\label{eq:Qrcdefs}
Q_1^C = (u^C)_t - \s (u^C)_\x, \quad
Q_2^R = (q^R)_\x, \quad
Q_3^R = -\Big(\frac{(q^R)_\x}{v^R}\Big)_\x, \quad
Q_3^C = -\Big(\frac{(q^C)_\x}{v^C}\Big)_\x,
\end{equation}
and 
\begin{align*}
Q_1^I
&= (\pbar - p^R - p^C - p^S)_\x, \\
Q_2^I
&= \pbar\ubar_\x - p^R (u^R)_\x - p^C (u^C)_\x - p^S (u^S)_\x, \\
Q_3^I
&= -\Big(
    \frac{\qbar_\x}{\vbar}
    - \frac{(q^R)_\x}{v^R}
    - \frac{(q^C)_\x}{v^C}
    - \frac{(q^S)_\x}{v^S}
    \Big)_\x \\
&\quad + \a\big(
    \vbar\qbar - v^R q^R - v^C q^C - v^S q^S
    \big) \\
&\quad + \b\big(
    \thbar^4 - (\th^R)^4 - (\th^C)^4 - (\th^S)^4
    \big)_\x.
\end{align*}
Then, \((\vbar,\ubar,\thbar,\qbar)\) satisfies the following:
\begin{equation}\label{bar-eq}
\begin{cases}
\vbar_t -\s \vbar_\x - \ubar_\x = -\dot{X}(v^S)_\x, \\
\ubar_t -\s \ubar_\x + \pbar_\x = Q_1 - \dot{X}(u^S)_\x, \\
\frac{R}{\g-1}(\thbar_t -\s \thbar_\x) + \pbar \ubar_\x + \qbar_\x
=Q_2 - \frac{R}{\g-1}\dot{X}(\th^S)_\x, \\
-\left(\frac{\qbar_\x}{\vbar}\right)_\x + \a \vbar \qbar + \b (\thbar^4)_\x = Q_3.
\end{cases}
\end{equation}

We show the existence of global strong solutions to the system \eqref{main-x} for any sufficiently small \(H^2\) initial perturbation of the composite wave \((\vbar,\ubar,\thbar)\) verifying \eqref{bar-eq}.
Note that the initial radiative flux \(q_0\) is determined by \(v_0\) and \(\th_0\) through the elliptic equation.
As a first step, let us recall the local well-posedness theory for strong solutions to \eqref{main-x} in the following proposition.

\begin{proposition}\label{prop:lwp}
Let $\underline{v}$, $\underline{u}$, and $\underline{\theta}$ be smooth functions so that
\begin{align*}
(\underline{v}(x), \underline{u}(x), \underline{\theta}(x))=(v_\pm, u_\pm, \theta_\pm), 
\qquad \text{for } \pm x \geq 1.
\end{align*}
For any positive constants \(\underline{\kappa}_0\), \(\overline{\kappa}_0\), \(\underline{\kappa}_1\), \(\overline{\kappa}_1\) and \(M_0\) with $\overline{\kappa}_1 > \overline{\kappa}_0>\underline{\kappa}_0>\underline{\kappa}_1 > 0$, there exists a constant $C_* >1$, so that, for any $M_1 > C_*M_0$, there exists $T_0 > 0$ so that the following holds. Consider any initial data $(v_0, u_0, \theta_0)$ verifying
\begin{align*}
\norm{(v_0-\underline{v},u_0-\underline{u}, \theta_0-\underline{\theta})}_{H^2(\RR)} \leq M_0, \\
0 < \underline{\kappa}_0 \leq v_0(x),\theta_0(x)\leq \overline{\kappa}_0, \ \forall x \in \RR.
\end{align*}
Then, there exists a unique solution $(v,u,\theta, q)$ to \eqref{main-x} on $[0,T_0]$ so that
\begin{align*}
v-\underline{v}, \, u-\underline{u}, \, \theta-\underline{\theta} &\in C([0,T_0];H^2(\RR)), \\
q &\in C([0,T_0];H^3(\RR)),
\end{align*}
and
\begin{align*}
\norm{(v-\underline{v}, u-\underline{u}, \theta-\underline{\theta})}_{C([0,T_0];H^2(\RR))}
+\norm{q}_{C([0,T_0];H^3(\RR))} \leq M_1.
\end{align*}
Moreover,
\begin{equation}\label{eq:lwp-bdfrombelow}
\underline{\kappa}_1 \leq v(t,x),\theta(t,x) \leq \overline{\kappa}_1, \ \forall (t,x) \in [0,T_0] \times \RR.
\end{equation}
\end{proposition}

The proof follows from standard local well-posedness arguments for hyperbolic-elliptic coupled systems.
An appropriate reference is \cite{MR390516}; see also \cite[Theorem 2.1]{MR1925900}.
With this result in hand, we establish the \textit{a priori} estimate \Cref{prop:main}, then apply a standard continuation argument to obtain global existence.

\subsection{Construction of Shift} 
The continuation argument is based on the \textit{a priori} estimate established in \Cref{prop:main}.
As this estimate depends on the shift $X$, we give the definition of the shift here.
A key input in the construction of the shift is the weight function \(a\colon\RR\to\RR\), which is defined in \eqref{weight}.
For the purposes of the present section, it suffices to know that \(a\) satisfies the bound \(\norm{a}_{C^1}\le 2\) (although much finer estimates will be available later).
With this weight in hand, we define the shift \(X\) as the solution to the following ODE:
\begin{equation}\label{eq:shiftdef}
\begin{cases}
\begin{aligned}
\dot{X}(t)&=-\dfrac{M}{\delta_S}\bigg[
\displaystyle\int_{\RR} a(\xi-X(t)) \Big((u^S)_\xi(\xi-X(t))\psi\\
&\qquad\qquad\qquad
+\dfrac{R}{(\gamma-1)\bar \theta}(\theta^S)_\xi(\xi-X(t))\chi
+\dfrac{\bar p}{\bar v}(v^S)_\xi(\xi-X(t))\phi \Big)d\xi \bigg], \\
\end{aligned} \\
X(0)=0,
\end{cases}
\end{equation}
where \(M\) is the constant defined in \eqref{def:M} by \(M \coloneqq \frac{(\gamma-1)^2(v^*)^3\alpha A^*}{2(\gamma R K_1)^2\beta (\theta^*)^4}\) with \(A^*\) and \(K_1\) defined in \eqref{eq:A^*def} and \eqref{eq:K1def}, respectively.
This choice of \(M\) is used in the proof of \Cref{lem:L2estpart1}.
In \eqref{eq:shiftdef}, the shifted arguments of the weight and shock profiles are displayed explicitly.

\vspace{2mm}
This shift may be uniquely constructed on any interval for which \eqref{eq:lwp-bdfrombelow} is verified (see \cite[Lemma 3.1]{KVW-ARMA25}) as a consequence of the Cauchy--Lipschitz theorem.
To be precise, we have

\begin{lemma}\label{lem:shiftexistence}
For any \(0<c_1<c_2\), there exists a constant \(C>0\) such that the following holds.
For any \(T>0\) and any \(v,u,\th \in L^\infty((0,T) \times \RR)\) satisfying 
\[
c_1 \leq v(t,x),\theta(t,x) \leq c_2, \qquad
\abs{u(t,x)} \le c_2
\]
for all \((t,x)\in(0,T)\times\RR\), the ODE \eqref{eq:shiftdef} has a unique absolutely continuous solution \(X\) on \([0,T]\).
Furthermore, 
\[
\abs{X(t)} \le C t, \qquad \forall t \in[0,T].
\]
\end{lemma}

\subsection{Global Well-Posedness and Asymptotic Stability}
To prove \Cref{thm:main}, it suffices to show the following proposition on global existence and uniform estimates.
\begin{proposition}\label{prop:gwp}
Let \((v_-,u_-,\th_-) \in \RR_+\times\RR\times\RR_+\).
Let \((\vbar,\ubar,\thbar,\qbar)\) denote the composite wave constructed in \eqref{Ubar}--\eqref{qbar}, where the shock component is shifted by \(X\).
The shift \(X\) is defined as the solution to \eqref{eq:shiftdef} using the weight function \(a\) defined in \eqref{weight}.
Then, there exist positive constants \(\d_0\) and \(\e_0\) such that if the wave strengths satisfy 
\[
|v_+ - v^*|
+|v^* - v_*|
+|v_* - v_-|
\le \d_0,
\]
and the initial perturbation satisfies \eqref{eq:initpertsmall}, then the Cauchy problem \eqref{main-x} (and \eqref{main}) admits a unique global-in-time solution \((v,u,\th,q)\).
Furthermore, the following estimate holds:
\begin{align*}
&\sup_{t\ge 0}\left[\norm{(v-\vbar,u-\ubar,\th-\thbar)(t,\cdot)}_{H^2(\RR)}^2\right]
+\sup_{t\ge 0} \norm{(q-\qbar)(t,\cdot)}_{H^3(\RR)}^2
+\d_S\int_0^\infty |\dot{X}(\t)|^2 d\t\\
&\quad
+\int_0^\infty (G^R+G^S) d\t
+\int_0^\infty \norm{(v-\vbar,u-\ubar,\th-\thbar)_x}_{H^1(\RR)}^2 d\t
+\int_0^\infty \norm{(q-\qbar)_x}_{H^2(\RR)}^2 d\t\\
&\le C\norm{(v-\vbar,u-\ubar,\th-\thbar)(0,\cdot)}_{H^2(\RR)}^2
+C\d_0^{\frac{1}{4}},
\end{align*}
where 
\[
G^R \coloneqq \int_{\RR}|(v^R)_\xi||(\phi, \chi)|^2 d\xi, \qquad
G^S \coloneqq \int_{\RR}|(v^S)_\xi||(\phi, \psi, \chi)|^2 d\xi.
\]
\end{proposition}

\subsection{Main Proposition: \textit{A Priori} Estimates}
The proof of \Cref{prop:gwp} boils down to getting the following \textit{a priori} estimate, formulated in terms of the perturbations defined in \eqref{eq:pertdefs}.

\begin{proposition}\label{prop:main}
Let \((v_-,u_-,\th_-) \in \RR_+\times\RR\times\RR_+\) be a given state.
Then, there exist positive constants \(C\), \(\d_1\) and \(\e_1\) such that the following holds:\\
Assume that \((v,u,\th,q)\) is a solution to the system \eqref{main-x} on \([0,T]\) for some \(T>0\).
Let \((\vbar,\ubar,\thbar,\qbar)\) denote the composite wave constructed in \eqref{Ubar}--\eqref{qbar}, with the shift \(X\) applied only to the radiative shock.
Here, \(X\) is the absolutely continuous solution to \eqref{eq:shiftdef} with respect to the weight function \(a\) defined in \eqref{weight}.
Suppose that, for some \(0<\d_0\le\d_1\), the wave strengths satisfy 
\[
\d_R + \d_C + \d_S \le \d_0,
\]
and the perturbations in each component satisfy
\begin{align*}
\phi, \psi, \chi &\in C([0,T];H^2(\RR)), \\
r &\in C([0,T];H^2(\RR)) \cap L^2(0,T;H^3(\RR)),
\end{align*}
and the \textit{a priori} assumption:
\begin{equation}\label{suppose}
\sup_{t\in[0,T]} \norm{(\phi,\psi,\chi)(t)}_{H^2} \le \e_1.
\end{equation}
Then, for any $t \leq T$, the following holds:
\begin{align*}
&\sup_{t \in [0,T]}\left[\norm{(\phi,\psi,\chi)(t,\cdot)}_{H^2(\RR)}^2\right] +\delta_S \int_0^t |\dot{X}(\tau)|^2 d\tau
+\int_0^t (G^R+G^S) d\tau\\
&\qquad
+\int_0^t \norm{(\phi,\psi,\chi)_x}_{H^1(\RR)}^2d\tau
+\int_0^t \norm{r_x}_{H^2(\RR)}^2 d\tau\\
&\leq C\norm{(\phi,\psi,\chi)(0,\cdot)}_{H^2(\RR)}^2
+C\delta_0^{\frac14}.
\end{align*}
Finally, by \eqref{eq:shiftdef}, it follows that
\[
|\dot{X}(t)|
\le C \norm{(\phi, \psi, \chi)(t,\cdot)}_{L^\infty(\RR)}, \qquad \forall t\in[0,T].
\]
\end{proposition}

\subsection{Proof of Theorem \ref{thm:main}}
Based on the local existence result \Cref{prop:lwp} and the \textit{a priori} estimate \Cref{prop:main}, a standard bootstrap argument yields \Cref{prop:gwp}.
Here, the bound on \(\sup_{t\ge 0}\norm{q-\qbar}_{H^3(\RR)}^2\) can be easily recovered by elliptic regularity theory on the elliptic equation \eqref{pert-qeq} for \(r=q-\qbar\).
Then, using \Cref{prop:gwp}, we obtain \eqref{eq:maincontinuity}.
The long-time behaviors \eqref{eq:pertgoesto0main} and \eqref{eq:shiftgoesto0main} follow from standard arguments as well.
These proofs are the same as those in \cite{GLWX-2026} (see also \cite{KVW-ADV23}).
Thus, we omit the details and devote the remainder of the paper to proving \Cref{prop:main}.

\section{Zeroth-Order Estimates}\label{sec:L2}
\setcounter{equation}{0}
\subsection{Wave Interaction Estimates}
In this section, we begin the proof of Proposition \ref{prop:main} by obtaining the \(L^2\) estimate.
We first derive preliminary estimates for the interaction terms \(Q_i\).

\begin{lemma}\label{lem:interactionests}
Let \(X\) be the shift defined by \eqref{eq:shiftdef} and let \(Q_i^I\), \(i=1,2,3,\) be defined in \eqref{eq:Qidefs}.
Then, under the same hypotheses as in \Cref{prop:main}, the following holds:\\
There exists a constant \(C>0\), independent of \(T\), \(\d_R\), \(\d_C\) and \(\d_S\), so that for any \(0\le t\le T\),
\begin{equation}\label{Qiests}
\norm{Q_1^I}_{H^2} + \norm{Q_2^I}_{H^2} + \norm{Q_3^I}_{H^1} 
\le C\d_S (\d_R+\d_C) e^{-c \d_S t}
+C\d_R\d_C e^{-c t}.
\end{equation}
\end{lemma}

\begin{proof}
The proof follows the same argument as in \cite[Lemma 4.2]{KVW-ARMA25}.
Although each wave is defined differently, the properties needed for the proof remain essentially unchanged, as shown in Lemmas \ref{lem:uR}--\ref{lem:VS}.
Moreover, the expressions for \(Q_1^I\) and \(Q_2^I\) are simpler than their counterparts in \cite{KVW-ARMA25}, since the terms involving viscosity and heat conductivity are absent in the present setting.
So, the bounds obtained in \cite{KVW-ARMA25} remain sufficient for our purposes.
While an additional estimate for \((Q_2^I)_{\x\x}\), which was not considered in \cite{KVW-ARMA25}, is required here, it also follows by a similar argument.
We therefore omit the details of these estimates.

\vspace{2mm}
It remains to estimate \(Q_3^I\).
To this end, we first observe that 
\[
|\dot{X}(t)| \le C\e_1, \qquad t\in[0,T].
\]
We may therefore choose \(\e_1>0\) sufficiently small to ensure that 
\[
\abs{X(t)} \le C\e_1 t \le \frac{\s t}{4}, \qquad t\in[0,T].
\]
We also recall the following inequalities established in \cite{KVW-ARMA25}:
\begin{equation}\begin{aligned}\label{Riest}
\norm{R_1}_{L^2}
&\le C\d_R \d_S e^{-c\d_S t}
+C\d_R (\d_C+\d_S) e^{-c t},\\
\norm{R_2}_{L^2}
&\le C\d_C \d_S e^{-c\d_S t}
+C\d_C \d_R e^{-c t},\\
\norm{R_3}_{L^2}
&\le C (\d_S)^{\frac{3}{2}} (\d_R+\d_C) e^{-c\d_S t},
\end{aligned}\end{equation}
where \(R_1\), \(R_2\) and \(R_3\) are given as follows:
\begin{align*}
R_1 &\coloneqq \abs{\big((v^R)_\x,(\th^R)_\x\big)}
\abs{(v^C-v_*,\th^C-\th_*,v^S-v^*,\th^S-\th^*)},\\
R_2 &\coloneqq \abs{\big((v^C)_\x,(\th^C)_\x\big)}
\abs{(v^R-v_*,\th^R-\th_*,v^S-v^*,\th^S-\th^*)},\\
R_3 &\coloneqq \abs{\big((v^S)_\x,(\th^S)_\x\big)}
\abs{(v^R-v_*,\th^R-\th_*,v^C-v^*,\th^C-\th^*)}.
\end{align*}

To establish \eqref{Qiests}, we decompose \(Q_3^I\) as follows:
\begin{align*}
Q_3^I &=
\underbrace{-\bigg(\frac{1}{\vbar}-\frac{1}{v^R}\bigg) (q^R)_{\x\x}
-\bigg(\frac{1}{\vbar}-\frac{1}{v^C}\bigg) (q^C)_{\x\x}
-\bigg(\frac{1}{\vbar}-\frac{1}{v^S}\bigg) (q^S)_{\x\x}}_{\eqqcolon L_1}\\
&\qquad
\underbrace{+\bigg(\frac{\vbar_\x}{\vbar^2}-\frac{(v^R)_\x}{(v^R)^2}\bigg)(q^R)_\x
+\bigg(\frac{\vbar_\x}{\vbar^2}-\frac{(v^C)_\x}{(v^C)^2}\bigg)(q^C)_\x
+\bigg(\frac{\vbar_\x}{\vbar^2}-\frac{(v^S)_\x}{(v^S)^2}\bigg)(q^S)_\x}_{\eqqcolon L_2}\\
&\qquad
\underbrace{+\a(\vbar-v^R)q^R
+\a(\vbar-v^C)q^C
+\a(\vbar-v^S)q^S}_{\eqqcolon L_3}\\
&\qquad
\underbrace{+4\b(\thbar^3-(\th^R)^3)(\th^R)_\x
+4\b(\thbar^3-(\th^C)^3)(\th^C)_\x
+4\b(\thbar^3-(\th^S)^3)(\th^S)_\x}_{\eqqcolon L_4}.
\end{align*}

\bci{L_1} We first note from \eqref{Ubar} that
\begin{equation}\begin{aligned}\label{ub-rcs}
\vbar-v^R
&=(v^C-v_*)+(v^S-v^*),
&\qquad\thbar-\th^R
&=(\th^C-\th_*)+(\th^S-\th^*),\\
\vbar-v^C
&=(v^R-v_*)+(v^S-v^*),
&\qquad\thbar-\th^C
&=(\th^R-\th_*)+(\th^S-\th^*),\\
\vbar-v^S
&=(v^R-v_*)+(v^C-v^*),
&\qquad\thbar-\th^S
&=(\th^R-\th_*)+(\th^C-\th^*).
\end{aligned}\end{equation}
By Lemmas \ref{lem:uR}--\ref{lem:VS}, we also have
\begin{equation}\begin{aligned}\label{qrcs-012}
\abs{q^R}+\abs{(q^R)_\x}+\abs{(q^R)_{\x\x}}
&\le C\abs{(v^R)_\x},\\
\abs{q^C}+\abs{(q^C)_\x}+\abs{(q^C)_{\x\x}}
&\le C\frac{\d_C}{\sqrt{1+t}}e^{-\frac{c}{1+t}(\x+\s t)^2},\\
\abs{q^S}+\abs{(q^S)_\x}+\abs{(q^S)_{\x\x}}
&\le C\abs{(v^S)_\x}.
\end{aligned}\end{equation}
Since the derivatives of \(q^C\) no longer retain the same self-similar structure, the pointwise comparisons \(|(q^C)_\x| \le |(\th^C)_\x|\) and \(|(q^C)_{\x\x}| \le |(\th^C)_\x|\) are not valid globally.
Nevertheless, as illustrated by the second estimate in \eqref{qrcs-012}, we can still obtain upper bounds of the same (or even smaller) scale as \(|(\th^C)_\x|\).
Thus, following the argument in \cite{KVW-ARMA25}, we obtain the following estimate, which is of the same order as \(R_2\):
\begin{equation}\begin{aligned}\label{R2p}
\norm{R_2'}_{L^2}
&= \Big\|\abs{\big(q^C,(q^C)_\x,(q^C)_{\x\x}\big)}
\abs{(v^R-v_*,\th^R-\th_*,v^S-v^*,\th^S-\th^*)}\Big\|_{L^2}\\
&\le C\d_C \d_S e^{-c\d_S t}
+C\d_C \d_R e^{-c t}.
\end{aligned}\end{equation}
It then follows from \eqref{ub-rcs} that 
\begin{multline*}
\abs{L_1}
\le C |(q^R)_{\x\x}|\abs{(v^C-v_*,v^S-v^*)}\\
+ C |(q^C)_{\x\x}|\abs{(v^R-v_*,v^S-v^*)}
+ C |(q^S)_{\x\x}|\abs{(v^R-v_*,v^C-v^*)}.
\end{multline*}
Using \eqref{Riest} and \eqref{R2p} together with \eqref{qrcs-012}, we find that 
\[
\norm{L_1}_{L^2}
\le C(\norm{R_1}_{L^2}+\norm{R_2'}_{L^2}+\norm{R_3}_{L^2})
\le C\d_S(\d_R+\d_C) e^{-c\d_S t}
+C\d_R\d_C e^{-c t}.
\]

\bci{L_2} We begin by observing that
\begin{align*}
\abs{\frac{\vbar_\x}{\vbar^2}-\frac{(v^R)_\x}{(v^R)^2}}
&\le C\Big(\abs{\vbar-v^R}\abs{(v^R)_\x}+\abs{(v^C)_\x}+\abs{(v^S)_\x}\Big),\\
\abs{\frac{\vbar_\x}{\vbar^2}-\frac{(v^C)_\x}{(v^C)^2}}
&\le C\Big(\abs{\vbar-v^C}\abs{(v^C)_\x}+\abs{(v^R)_\x}+\abs{(v^S)_\x}\Big),\\
\abs{\frac{\vbar_\x}{\vbar^2}-\frac{(v^S)_\x}{(v^S)^2}}
&\le C\Big(\abs{\vbar-v^S}\abs{(v^S)_\x}+\abs{(v^R)_\x}+\abs{(v^C)_\x}\Big).
\end{align*}
Thus, we have
\begin{align*}
\abs{L_2}
&\le C\Big(\abs{\vbar-v^R}\abs{(v^R)_\x}+\abs{(v^C)_\x}+\abs{(v^S)_\x}\Big) |(q^R)_\x|\\
&\qquad
+C\Big(\abs{\vbar-v^C}\abs{(v^C)_\x}+\abs{(v^R)_\x}+\abs{(v^S)_\x}\Big) |(q^C)_\x|\\
&\qquad
+C\Big(\abs{\vbar-v^S}\abs{(v^S)_\x}+\abs{(v^R)_\x}+\abs{(v^C)_\x}\Big) |(q^S)_\x|.
\end{align*}
Then, we proceed by adopting the spatial decomposition used in \cite[Lemma 4.2]{KVW-ARMA25} to get \eqref{Riest}.
More precisely, for the first two terms on the right-hand side, we divide the domain into the regions \(\x+\s t \le \frac{\l_{1*}(1+t)}{2}\) and \(\x+\s t > \frac{\l_{1*}(1+t)}{2}\), whereas for the last term, we use the regions \(\x\le-\frac{\s t}{2}\) and \(\x>-\frac{\s t}{2}\).
This gives
\begin{align*}
\norm{(q^R)_\x (v^C)_\x}_{L^2}
+\norm{(q^C)_\x (v^R)_\x}_{L^2}
&\le C \d_R \d_C e^{-ct},\\
\norm{(q^R)_\x (v^S)_\x}_{L^2}
+\norm{(q^S)_\x (v^R)_\x}_{L^2}
&\le C \d_R \d_S e^{-c \d_S t},\\
\norm{(q^C)_\x (v^S)_\x}_{L^2}
+\norm{(q^S)_\x (v^C)_\x}_{L^2}
&\le C \d_C \d_S e^{-c \d_S t}.
\end{align*}
This together with \eqref{Riest} and Lemmas \ref{lem:uR}--\ref{lem:VS} yields that
\[
\norm{L_2}_{L^2}
\le C\d_S(\d_R+\d_C) e^{-c\d_S t}
+C\d_R\d_C e^{-c t}.
\]

\bci{L_3+L_4} Note from Lemmas \ref{lem:uR}--\ref{lem:VS} that
\[
\abs{q^R} \le C \abs{(v^R)_\x}, \qquad
\abs{q^C} \le C \abs{(\th^C)_\x}, \qquad
\abs{q^S} \le C \abs{(v^S)_\x}.
\]
Hence, applying \eqref{Riest}, we obtain
\[
\norm{L_3}_{L^2}+\norm{L_4}_{L^2}
\le C(\norm{R_1}_{L^2}+\norm{R_2}_{L^2}+\norm{R_3}_{L^2})
\le C\d_S(\d_R+\d_C) e^{-c\d_S t}
+C\d_R\d_C e^{-c t}.
\]
Thus, we conclude that 
\[
\norm{Q_3^I}_{L^2}
\le C\d_S(\d_R+\d_C) e^{-c\d_S t}
+C\d_R\d_C e^{-c t}.
\]

On the other hand, a similar argument yields the corresponding upper bound for \(\norm{(Q_3^I)_\x}_{L^2}\).
Indeed, since 
\[
\abs{(q^R)_{\x\x\x}}
\le C\abs{(v^R)_\x}, \qquad
\abs{(q^C)_{\x\x\x}}
\le C\frac{\d_C}{\sqrt{1+t}}e^{-\frac{c}{1+t}(\x+\s t)^2}, \qquad
\abs{(q^S)_{\x\x\x}}
\le C\abs{(v^S)_\x},
\]
we can establish the same bound as for \(\norm{Q_3^I}_{L^2}\).
We omit the details, as the argument is essentially identical to the preceding one for \(\norm{Q_3^I}_{L^2}\).
Therefore, we get the desired estimate: 
\[
\norm{Q_3^I}_{H^1}
\le C\d_S(\d_R+\d_C) e^{-c\d_S t}
+C\d_R\d_C e^{-c t}.
\]
This completes the proof of Lemma \ref{lem:interactionests}.
\end{proof}

Moreover, the remaining \(Q\) terms must also be estimated.

\begin{lemma}\label{lem:Q}
Let \(X\) be the shift defined by \eqref{eq:shiftdef} and let \(Q_i^R\) and \(Q_i^C\) be defined in \eqref{eq:Qrcdefs}.
Then, under the same hypotheses as in \Cref{prop:main}, the following holds:\\
There exists a constant \(C>0\), independent of \(T\), \(\d_R\), \(\d_C\) and \(\d_S\), so that for any \(0\le t\le T\),
\begin{align*}
\int_0^t \norm{Q_2^R}_{L^1}^{\frac{4}{3}} d\t
&\le C \d_R^{\frac{1}{3}},\\
\int_0^t \Big(\norm{Q_2^R}_{H^2}^2 + \norm{Q_3^R}_{H^1}^2\Big) d\t
&\le C\d_R,\\
\int_0^t \Big(\norm{Q_1^C}_{H^2}^2 + \norm{Q_3^C}_{H^1}^2\Big) d\t
&\le C\d_C^2.
\end{align*}
\end{lemma}
\begin{proof}
First, we recall the definitions of \(Q_2^R\) and \(Q_3^R\) given in \eqref{eq:Qrcdefs}:
\[
Q_2^R = (q^R)_\x, \qquad
Q_3^R = -\left(\frac{(q^R)_\x}{v^R}\right)_\x
= -\frac{(q^R)_{\x\x}}{v^R} + \frac{(q^R)_\x (v^R)_\x}{(v^R)^2}.
\]
Then, using Lemma \ref{lem:qR}, we find that
\[
\norm{\big(Q_2^R,(Q_2^R)_\x,(Q_2^R)_{\x\x}\big)}_{L^2}
=\norm{\big((q^R)_\x,(q^R)_{\x\x},(q^R)_{\x\x\x}\big)}_{L^2}
\le C\min\{\d_R, (1+t)^{-1}\}.
\]
Thus, we obtain 
\begin{equation}\label{Q2RH2}
\norm{Q_2^R}_{H^2}
=\norm{(q^R)_\x}_{H^2}
\le C\min\{\d_R, (1+t)^{-1}\}.
\end{equation}
Moreover, Lemma \ref{lem:qR} also implies the following \(L^1\) bound: 
\begin{equation}\label{Q2RL1}
\norm{Q_2^R}_{L^1}
=\norm{(q^R)_\x}_{L^1}
\le C\min\{\d_R, (1+t)^{-1}\}.
\end{equation}
On the other hand, the other rarefaction residual \(Q_3^R\) satisfies 
\[
\abs{Q_3^R}
\le C\left(\abs{(q^R)_{\x\x}}+\abs{(q^R)_\x}\abs{(v^R)_\x}\right),
\]
and 
\[
\abs{(Q_3^R)_\x}
\le C\left(\abs{(q^R)_{\x\x\x}}
+\abs{(q^R)_{\x\x}}\abs{(v^R)_\x}
+\abs{(q^R)_\x}\abs{(v^R)_{\x\x}}
+\abs{(q^R)_\x}\abs{(v^R)_\x}^2
\right).
\]
Using Lemmas \ref{lem:uR}--\ref{lem:qR}, we have
\begin{equation}\label{Q3RH1}
\norm{Q_3^R}_{H^1}
\le \norm{Q_3^R}_{L^2}
+\norm{(Q_3^R)_\x}_{L^2}
\le C\min\{\d_R, (1+t)^{-1}\}.
\end{equation}

We now turn to the contact residuals.
We first recall the definitions of \(Q_1^C\) and \(Q_3^C\) given in \eqref{eq:Qrcdefs}.
Then, the choice of the viscous contact wave \eqref{uc} yields that
\[
Q_1^C = -\frac{u^C-u^*}{2(1+t)}-\frac{\x+\s t}{2(1+t)}(u^C)_\x, \qquad
Q_3^C = -\frac{(q^C)_{\x\x}}{v^C} + \frac{(q^C)_\x (v^C)_\x}{(v^C)^2}.
\]
It then follows from Lemma \ref{lem:VC} that 
\[
\abs{Q_1^C}+\abs{(Q_1^C)_\x}+\abs{(Q_1^C)_{\x\x}}+\abs{Q_3^C}+\abs{(Q_3^C)_\x}
\le C\d_C (1+t)^{-\frac{3}{2}} e^{-\frac{c}{1+t}(\x+\s t)^2}.
\]
This implies that
\begin{equation}\label{Q13CH21}
\norm{Q_1^C}_{H^2}
+\norm{Q_3^C}_{H^1}
\le C\d_C (1+t)^{-\frac{5}{4}}.
\end{equation}

We now turn to the time-integrated estimates.
Integrating \eqref{Q2RL1} over \([0,t]\), we obtain
\[
\int_0^t \norm{Q_2^R}_{L^1}^{\frac{4}{3}} d\t
\le C\int_0^{\frac{1}{\d_R}} (\d_R)^{\frac{4}{3}}d\t
+C\int_{\frac{1}{\d_R}}^\infty (1+\t)^{-\frac{4}{3}} d\t
\le C \d_R^{\frac{1}{3}}.
\]
Then, using \eqref{Q2RH2} and \eqref{Q3RH1}, we also find that 
\[
\int_0^t \Big(\norm{Q_2^R}_{H^2}^2 + \norm{Q_3^R}_{H^1}^2\Big) d\t
\le C\int_0^t \min\{\d_R, (1+\t)^{-1}\}^2 d\t
\le C\d_R.
\]
Lastly, for the contact residuals, \eqref{Q13CH21} yields that 
\[
\int_0^t \Big(\norm{Q_1^C}_{H^2}^2 + \norm{Q_3^C}_{H^1}^2\Big) d\t
\le C\d_C^2 \int_0^t (1+\t)^{-\frac{5}{2}}d\t
\le C\d_C^2.
\]
This completes the proof of Lemma \ref{lem:Q}.
\end{proof}

\begin{remark}\label{rmk:Q}
A straightforward calculation implies that 
\[
\int_0^t \Big(\d_S (\d_R+\d_C) e^{-c \d_S \t}
+\d_R\d_C e^{-c \t}\Big)^2 d\t
\le C\d_S(\d_R+\d_C)^2 + C\d_R^2 \d_C^2.
\]
Combining this with the estimates for the individual terms shown in Lemmas \ref{lem:interactionests} and \ref{lem:Q}, we obtain bounds for \(Q_1\), \(Q_2\) and \(Q_3\) as follows:
\[
\int_0^t \Big(\norm{\big(Q_1,Q_2\big)}_{H^2}^2
+\norm{Q_3}_{H^1}^2\Big) d\t
\le C(\d_R+\d_C^2).
\]
These bounds will be used in the subsequent analysis.
\end{remark}

\subsection{Relative Entropy Method}
As a first step in the proof of \Cref{prop:main}, we employ the \(a\)-contraction method to execute the \(L^2\) energy estimate.
This method relies on the weighted Poincar\'e inequality in \Cref{lem:poincare}, which allows us to control the hyperbolic part of the energy estimate.
In the previous works \cite{GLWX-2026,KVW-ARMA25}, the application of this inequality is based on the dissipation provided by the parabolic term.
However, the system \eqref{main-x} contains no such term, and the energy estimate structures developed in existing studies for radiative shocks are not compatible with this approach.
To this end, we identify a new variable (see \eqref{eq:zetadef}), which is not readily apparent in the system \eqref{main-x}, to which the Poincar\'e inequality is applicable.

\vspace{2mm}
The following lemma provides the \(L^2\) energy estimate.
\begin{lemma}\label{lem:l2est}
Under the hypotheses of \Cref{prop:main}, there exists a constant $C >0$, which is independent of \(\d_0\), \(\e_1\) and \(T\), such that for all \(t\in[0,T]\), the following holds:
\begin{align*}
&\sup_{\t \in [0,t]} \norm{(\phi, \psi, \chi)(\t)}_{L^2}^2
+\delta_S\int_0^t|\dot{X}(\tau)|^2d\tau
+\int_0^t(G^R+G^S)d\tau+\int_0^t(\norm{r}_{L^2}^2+\norm{r_\xi}_{L^2}^2)d\tau \\
&\qquad
\leq C\norm{(\phi, \psi, \chi)(0)}_{L^2}^2
+C(\e_1 + \sqrt{\delta_0})\int_0^t \norm{(\psi_\xi, \phi_\xi,\chi_\xi)}_2^2d\tau
+C\delta_0^{\frac{1}{4}}.
\end{align*}
\end{lemma}
The remainder of this section is devoted to the proof of this lemma.

\vspace{2mm}
The method of \(a\)-contraction with shifts builds on the relative entropy method by incorporating a weight function and a time-dependent shift.
We begin by defining the relative entropy:
\begin{align}\label{eq:relentdef}
\et(U|\Ubar) \coloneqq R\Phif{v}{\vbar}+ \frac{R}{\g-1}\Phif{\th}{\thbar} + \frac{(u-\ubar)^2}{2\thbar}
\end{align}
where \(\Phi(z)=z-1-\log z\).
We also introduce the following functional: 
\[
\Hcal(U|\Ubar) \coloneqq R\Phif{v}{\vbar}+ \frac{R}{\g-1}\Phif{\th}{\thbar}.
\]
Recall that \(\d_S = v_+ - v^* >0\).
We then define the weight function as follows:
\begin{equation}\label{weight}
a(\x) \coloneqq 1 + \frac{v^S(\x) -v^*}{\sqrt{\d_S}}, \qquad
a^{-X}(\x) \coloneqq a(\x-X(t)).
\end{equation}
Note that \((v^S)_\x > 0\) for a \(3\)-shock and that
\begin{equation}\label{der-a}
a_\x = \frac{(v^S)_\x}{\sqrt{\d_S}} > 0, \qquad
1< a < 1+\sqrt{\d_S}, \qquad
(a^{-X})_t = -\dot{X} (a^{-X})_\x.
\end{equation}
In what follows, the weight function \(a\) always appears with the shift \(-X\).
Thus, for notational simplicity, we omit the superscript \(-X\) throughout.

\vspace{2mm}
The following lemma provides the time evolution of the weighted relative entropy with the time-dependent shift \(X\).

\begin{lemma}\label{lem:RHS}
Let \(a\) be the weight function defined by \eqref{weight}.
Let \(U\) and \(q\) be a solution to \eqref{main-x} and \(\Ubar\) and \(\qbar\) be the composite wave defined by \eqref{Ubar} and \eqref{qbar}.
Then, it holds that
\begin{equation}\label{RHS}
\frac{d}{dt} \int_\RR a \thbar \et(U|\Ubar) d\x
=\dot{X}Y + \Jcal^{hyp} + \Jcal^{rad} + \Jcal^{int}  - \Jcal^{good}
\end{equation}
where
\begin{align*}
Y &\coloneqq 
- \int_\RR a_\x \thbar \et(U|\Ubar) d\x
+\int_\RR a\left(-\Hcal(U|\Ubar) (\th^S)_\x
+\frac{\pbar \phi}{\vbar} (v^S)_\x
+\frac{R}{\g-1}\frac{\chi}{\thbar}(\th^S)_\x
+\psi (u^S)_\x\right)d\x,\\
\Jcal^{hyp} &\coloneqq 
\int_\RR a_\x \psi (p-\pbar) d\x
-\int_\RR a \pbar \left(\g\Phif{v}{\vbar}+\Phif{p}{\pbar}\right) \left((u^C)_\x+(u^S)_\x\right) d\x,\\
\Jcal^{rad} &\coloneqq
\int_\RR a_\x \frac{\chi}{\th}r d\x
+\int_\RR a \left(\frac{\chi^2}{\th\thbar} -\frac{\g-1}{R}\Hcal(U|\Ubar) \right) \qbar_\x d\x
+\int_\RR a \frac{\chi}{\th\thbar} r \thbar_\x d\x
- \int_\RR (a A_\th)_\x \frac{r r_\x}{v} d\x\\
&\qquad\qquad
+ \int_\RR (a A_\th r)_\x \frac{\phi}{v\vbar} \qbar_\x d\x
- \frac{\a}{4 \b} \int_\RR a \thbar \qbar r \Big(\frac{v}{\th^5}-\frac{\vbar}{\thbar^5}\Big) d\x
+ \int_\RR a (A_\th-\Abar_\th) \Big(\frac{\qbar_\x}{\vbar}\Big)_\x r d\x,\\
\Jcal^{int} &\coloneqq 
-\int_\RR a \left(\psi Q_1 + \frac{\chi}{\thbar} Q_2
-\frac{\g-1}{R}\Hcal(U|\Ubar) Q_2
+\Abar_\th r Q_3 \right) d\x,
\end{align*}
and 
\begin{align*}
\Jcal^{good} &\coloneqq
\s \int_\RR a_\x \thbar \et(U|\Ubar) d\x
+ \int_\RR a \pbar \left(\g\Phif{v}{\vbar}+\Phif{p}{\pbar}\right) (u^R)_\x d\x\\
&\qquad\qquad
+ \frac{1}{4\b}\int_\RR a \frac{\thbar}{v\th^5} (r_\x)^2 d\x
+ \frac{\a}{4\b}\int_\RR a\frac{v \thbar}{\th^5}r^2 d\x,
\end{align*}
where \(A_\th\) and \(\Abar_\th\) are given by
\[
A_\th \coloneqq \frac{\thbar}{4\b \th^5}, \qquad
\Abar_\th \coloneqq \frac{1}{4\b \thbar^4}.
\]
\end{lemma}

\begin{proof}
The radiative Euler system has an elliptic coupling, whereas the Navier--Stokes--Fourier system studied in \cite{KVW-ARMA25} has parabolic effects arising from viscosity and heat conduction.
The two systems, however, share the same hyperbolic part and the same definition of relative entropy.
Thus, the time evolution of the relative entropy takes a similar form in both systems.

Adapting the observations in \cite[Lemma 4.3]{KVW-ARMA25} to the radiative Euler system, we obtain
\begin{equation}\label{t-sx}
\begin{aligned}
(\rd_t - \s \rd_\x)\big(\thbar\et(U|\Ubar)\big)
&=(\thbar_t - \s \thbar_\x) \Hcal(U|\Ubar)
+\left(\frac{\pbar \phi}{\vbar} (v^S)_\x
+\frac{R}{\g-1}\frac{\chi}{\thbar}(\th^S)_\x
+\psi (u^S)_\x\right)\dot{X}\\
&\quad
+\left(- \frac{\pbar \phi^2}{v\vbar}
-\frac{\chi}{\th}(p-\pbar)
+ \frac{\pbar}{\th\thbar}\chi^2\right) \ubar_\x\\
&\quad
- \big(\psi(p-\pbar)\big)_\x
- \left(\frac{\chi}{\th}r\right)_\x
+ \left(\frac{\chi}{\th}\right)_\x r
+\frac{\chi^2}{\th\thbar} \qbar_\x
-\psi Q_1
-\frac{\chi}{\thbar}Q_2.
\end{aligned}
\end{equation}
Note that the definition of \(\Phi(\cdot/\cdot)\) yields that 
\[
- \frac{\pbar \phi^2}{v\vbar}
-\frac{\chi}{\th}(p-\pbar)
+ \frac{\pbar}{\th\thbar}\chi^2
=\frac{(\g-1)\pbar}{R} \Hcal(U|\Ubar)
-\pbar \left(\g\Phif{v}{\vbar}+\Phif{p}{\pbar}\right).
\]
We also see from the third equation of \eqref{bar-eq} that 
\[
\thbar_t - \s \thbar_\x
=\frac{\g-1}{R}\big(Q_2 - \pbar \ubar_\x -\qbar_\x\big)
-\dot{X}(\th^S)_\x.
\]
Consequently, \eqref{t-sx} becomes
\begin{equation}\label{t-sx0}
\begin{aligned}
(\rd_t - \s \rd_\x)\big(\thbar\et(U|\Ubar)\big)
&= \left(-\Hcal(U|\Ubar)(\th^S)_\x
+\frac{\pbar \phi}{\vbar} (v^S)_\x
+\frac{R}{\g-1}\frac{\chi}{\thbar}(\th^S)_\x
+\psi (u^S)_\x\right)\dot{X}\\
&\quad
-\pbar \left(\g\Phif{v}{\vbar}+\Phif{p}{\pbar}\right) \ubar_\x
- \big(\psi(p-\pbar)\big)_\x
- \left(\frac{\chi}{\th}r\right)_\x
+ \left(\frac{\chi}{\th}\right)_\x r\\
&\quad
+\left(\frac{\chi^2}{\th\thbar} -\frac{\g-1}{R}\Hcal(U|\Ubar) \right) \qbar_\x
-\psi Q_1
-\left(\frac{\chi}{\thbar}-\frac{\g-1}{R}\Hcal(U|\Ubar)\right)Q_2.
\end{aligned}
\end{equation}
This establishes that
\begin{align*}
&\frac{d}{dt}\int_\RR a \thbar \et(U|\Ubar) d\x
= -\dot{X} \int_\RR a_\x \thbar \et(U|\Ubar) d\x
+ \int_\RR a( \rd_t -\s \rd_\x) \big[\thbar \et(U|\Ubar) \big] d\x
-\s \int_\RR a_\x \thbar \et(U|\Ubar) d\x \\
&\qquad
= -\s \int_\RR a_\x \thbar \et(U|\Ubar) d\x
-\dot{X} \int_\RR a_\x \thbar \et(U|\Ubar) d\x\\
&\qquad\qquad
+\dot{X} \int_\RR a \left(-\Hcal(U|\Ubar)(\th^S)_\x
+\frac{\pbar \phi}{\vbar} (v^S)_\x
+\frac{R}{\g-1}\frac{\chi}{\thbar}(\th^S)_\x
+\psi (u^S)_\x\right) d\x\\
&\qquad\qquad
-\int_\RR a \pbar \left(\g\Phif{v}{\vbar}+\Phif{p}{\pbar}\right) \ubar_\x d\x
+\int_\RR a_\x \psi(p-\pbar) d\x
+\int_\RR a_\x \frac{\chi}{\th}r d\x
+\int_\RR a \left(\frac{\chi}{\th}\right)_\x r d\x\\
&\qquad\qquad
+\int_\RR a \left(\frac{\chi^2}{\th\thbar} -\frac{\g-1}{R}\Hcal(U|\Ubar) \right) \qbar_\x d\x
-\int_\RR a \left(\psi Q_1 + \frac{\chi}{\thbar} Q_2
-\frac{\g-1}{R}\Hcal(U|\Ubar) Q_2 \right) d\x.
\end{align*}
Since \(\ubar_\x = (u^R)_\x + (u^C)_\x + (u^S)_\x\) and \((u^R)_\x \ge 0\), the following holds:
\[
-\int_\RR a \pbar \left(\g\Phif{v}{\vbar}+\Phif{p}{\pbar}\right) (u^R)_\x d\x 
\le 0,
\]
which has a non-positive contribution in the subsequent analysis.

We now extract good terms from the radiative heat flux effect.
To this end, we define
\[
A_\th \coloneqq \frac{\thbar}{4\b \th^5}, \qquad
\Abar_\th \coloneqq \frac{1}{4\b \thbar^4},
\]
and recall the radiative heat transfer equations:
\[
-\left(\frac{q_\x}{v}\right)_\x + \a v q + \b (\th^4)_\x =0, \qquad
-\left(\frac{\qbar_\x}{\vbar}\right)_\x + \a \vbar \qbar + \b (\thbar^4)_\x = Q_3.
\]
Then, we obtain 
\[
\left(\frac{\chi}{\th}\right)_\x
= \frac{\chi}{\th\thbar}\thbar_\x
+ A_\th \Big(\frac{q_\x}{v} - \frac{\qbar_\x}{\vbar}\Big)_\x
-\frac{\a\thbar}{4\b} \Big(\frac{vq}{\th^5} - \frac{\vbar \qbar}{\thbar^5}\Big)
+ (A_\th-\Abar_\th) \Big(\frac{\qbar_\x}{\vbar}\Big)_\x
- \Abar_\th Q_3.
\]
The second and third terms on the right-hand side can be written as follows: 
\[
A_\th \Big(\frac{q_\x}{v} - \frac{\qbar_\x}{\vbar}\Big)_\x
= A_\th \Big(\frac{r_\x}{v}\Big)_\x
-A_\th \Big(\frac{\phi}{v\vbar}\qbar_\x\Big)_\x
\]
and 
\[
-\frac{\a\thbar}{4\b} \Big(\frac{vq}{\th^5} - \frac{\vbar \qbar}{\thbar^5}\Big)
= -\frac{\a v \thbar}{4\b \th^5}r
-\frac{\a \thbar \qbar}{4\b} \Big(\frac{v}{\th^5} - \frac{\vbar}{\thbar^5}\Big).
\]
Then, we have
\begin{align*}
\int_\RR a \left(\frac{\chi}{\th}\right)_\x r d\x 
&= -\frac{1}{4\b}\int_\RR a \frac{\thbar}{v\th^5} (r_\x)^2 d\x
-\frac{\a}{4\b}\int_\RR a\frac{v \thbar}{\th^5}r^2 d\x\\
&\qquad
+ \int_\RR a \frac{\chi}{\th\thbar} r \thbar_\x d\x
- \int_\RR (a A_\th)_\x \frac{r r_\x}{v} d\x
+ \int_\RR (a A_\th r)_\x \frac{\phi}{v\vbar} \qbar_\x d\x\\
&\qquad
- \frac{\a}{4 \b} \int_\RR a \thbar \qbar r \Big(\frac{v}{\th^5}-\frac{\vbar}{\thbar^5}\Big) d\x
+ \int_\RR a (A_\th-\Abar_\th) \Big(\frac{\qbar_\x}{\vbar}\Big)_\x r d\x
- \int_\RR a \Abar_\th r Q_3 d\x.
\end{align*}
The two terms on the first line are the desired good terms.
Therefore, it follows that 
\[
\frac{d}{dt} \int_\RR a \thbar \et(U|\Ubar) d\x
=\dot{X}Y + \Jcal^{hyp} + \Jcal^{rad} + \Jcal^{int} - \Jcal^{good}.
\]
This completes the proof of Lemma \ref{lem:RHS}.
\end{proof}

\subsection{Elliptic Representation of Temperature}\label{sec:chir}
The radiative dissipation in \(\Jcal^{good}\) provides control of \(r\) and \(r_\x\).
In this subsection, we derive estimates for \(\chi\) and its derivatives in terms of the \(r\)-terms.
We first introduce the following to control the terms arising from the contact wave:
\begin{equation}\label{eq:Wdef}
W(t,\x) \coloneqq \frac{1}{\sqrt{1+t}} e^{-\frac{c_1}{2(1+t)}(\x+\s t)^2}.
\end{equation}

From the heat transfer equation, we find that \(r=q-\qbar\) verifies 
\[
-\Big(\frac{r_\xi}{v^*}\Big)_\xi+\alpha v^* r=-4\beta(\theta^*)^3\chi_\xi+E_1,
\]
where 
\begin{equation}\begin{aligned}\label{eq:E1def}
E_1
&=-Q_3
-\Big(\frac{r_\xi\phi}{v\bar{v}}\Big)_\xi
-\Big(\frac{\bar{q}_\xi\phi}{v\bar{v}}\Big)_\xi
-\beta(4\theta^3-4\bar{\theta}^3)\bar{\theta}_\xi
-\alpha \phi r
-\alpha \phi \bar{q} \\
&\qquad
+4\beta((\theta^*)^3-\theta^3)\chi_\xi
+\Big(\frac{r_\xi}{\bar v}-\frac{r_\xi}{v^*}\Big)_\xi
+\alpha r (v^*-\bar v).
\end{aligned}\end{equation}

Define the differential operator 
\[
L \coloneqq \alpha v^*-\partial_\xi\Big(\frac{\partial_\xi}{v^*}\Big).
\]
We now introduce the Green's function associated with the operator \(L\)
\[
G(y)=\frac{1}{2\sqrt{\alpha}}\exp(-\sqrt{\alpha}v^*|y|),
\]
and its action on a function \(f\in L^2(\RR)\): 
\[
K_{v^*}f(\xi)
\coloneqq G*f(\xi).
\]

The following lemma establishes several useful properties of \(K_{v^*}\).
\begin{lemma}\label{lem:greensfunction}
The operator $K_{v^*}$ verifies the following properties:
\begin{enumerate}
\item The operator \(L:H^2(\RR) \to L^2(\RR)\) is invertible, with $L^{-1}(f)(\x)=K_{v^*}f(\x)$.
\item For every \(f \in L^2(\RR)\), $\partial_\xi K_{v^*}f= G_\xi*f = G*f_\xi =K_{v^*}f_\xi$ (in the sense of distributions).
\item $K_{v^*}$ is a bounded operator from $L^2(\RR)$ to $H^2(\RR)$: there exists a constant \(C>0\) such that for every \(f\in L^2(\RR)\),
\[
\norm{K_{v^*} f}_{H^2(\RR)} \le C \norm{f}_{L^2(\RR)}.
\]
\end{enumerate}
\end{lemma}

\begin{proof}
\textbf{Proof of (1):}
Consider the change of variable $z(\x)=v^*\x$.
Under this transformation, the operator \(L\) takes the form 
\[
L=v^*(-\partial_{z}^2+\alpha).
\]
Since the Green's function of $(-\partial_{z}^2+\alpha)$ is given by $\frac{1}{2\sqrt{\alpha}}e^{-\sqrt{\alpha}|z|}$, transforming back to the original variable \(\x\) yields (1).

\noindent \textbf{Proof of (2):} It follows straightforwardly from the properties of convolution.

\noindent \textbf{Proof of (3):} We first note that, by Young's convolution inequality 
\[
\norm{K_{v^*}f}_{L^2(\RR)} \le C \norm{f}_{L^2(\RR)}.
\]
Using (2) together with the fact that 
\begin{align*}
\partial_{\xi}G(\x-y)=-\frac{v^*}{2}\text{sgn}(\x-y)e^{-\sqrt{\alpha}v^*|\x-y|},
\end{align*}
we obtain the estimate for the \(L^2\) norm of the first-order derivative identically.\\
To control the \(H^2\) estimate, we may use (1) to find that 
\[
f
=\alpha v^*K_{v^*}f
-\partial_\xi \Big(\frac{\partial_{\xi}K_{v^*}f}{v^*}\Big).
\]
Re-arranging yields the following: 
\[
\partial_{\xi\xi}K_{v^*}f
=\alpha (v^*)^2 K_{v^*}f-v^*f.
\]
Taking the $L^2$ norms gives the result (3). 
\end{proof}

By \Cref{lem:greensfunction}, we obtain the following representation of \(r\):
\begin{equation}\label{eq:rrep}
r(\x)=-4\beta(\theta^*)^3K_{v^*}\chi_\xi+K_{v^*}E_1.
\end{equation}
We now introduce the new variable
\begin{equation}\label{eq:zetadef}
\zeta(\x) \coloneqq
-4\beta(\theta^*)^3K_{v^*}\chi.
\end{equation}
Applying \Cref{lem:greensfunction} again, we obtain
\begin{equation}\label{eq:E2def}
\zeta_\xi=r-K_{v^*}E_1=:r+E_2.
\end{equation}
Moreover, the same lemma allows us to express \(\chi\) in terms of \(\z\) and \(r_\x\) as 
\begin{equation}\begin{aligned}\label{eq:chirep}
\chi=\alpha v^*K_{v^*}\chi-\partial_\xi\left(\frac{K_{v^*}\chi_\xi}{v^*}\right) 
&=\frac{-\alpha v^*}{4\beta(\theta^*)^3}\zeta+\frac{r_\xi}{4\beta (\theta^*)^3v^*}+E_3,
\end{aligned}\end{equation}
where
\begin{equation}\label{eq:E3def}
E_3 \coloneqq
-\partial_\xi\left(\frac{K_{v^*}E_1}{4\beta(\theta^*)^3v^*}\right).
\end{equation}

We conclude this subsection with useful bounds for estimating the error terms.
\begin{lemma}\label{lem:errorsL2}
Let \(E_2\) and \(E_3\) be defined by \eqref{eq:E2def} and \eqref{eq:E3def}, respectively.
Under the hypotheses of \Cref{prop:main}, there exists a constant \(C>0\) such that
\begin{multline*}
\norm{E_2}_{L^2}^2
+\norm{E_3}_{L^2}^2
\leq C\bigg(\norm{Q_3}_{L^2}^2+\delta_S G^S+\delta_R G^R
+\delta_C\int_{\RR}W(t,\xi)^2 (\phi^2+\chi^2) d\xi\\
+(\e_1+\delta_0)\big(\norm{r}_{L^2}^2+\norm{r_\x}_{L^2}^2+\norm{\chi_\x}_{L^2}^2\big)\bigg).
\end{multline*}
\end{lemma}
\begin{proof}
Firstly, we establish the bound for \(\norm{E_2}_{L^2}^2=\norm{K_{v^*}E_1}_{L^2}^2\).
We estimate each term in the definition of \(E_1\) in \eqref{eq:E1def} separately.
For the first term, \Cref{lem:greensfunction} gives 
\[
\norm{K_{v^*}Q_3} _{L^2} \le C\norm{Q_3}_{L^2}.
\]
For the second term, we obtain 
\[
\norm{K_{v^*} \Big(\frac{r_\xi \phi}{v \bar v}\Big)_\xi}_{L^2}^2
=\norm{\partial_{\xi}\Big[K_{v^*}\Big(\frac{r_\xi \phi}{v \bar v}\Big)\Big]}_{L^2}^2
\le C\norm{\frac{r_\xi \phi}{v \bar v}}_{L^2}^2
\le C\e_1 \norm{r_\x}_{L^2}^2,
\]
where we have used \eqref{suppose} to ensure that \(v\) and \(\vbar\) are uniformly bounded away from zero by a positive constant of \(\Ocal(1)\) and that \(\abs{\phi}\le C \e_1\).
Similarly, for the third term, we note that 
\[
\norm{K_{v^*}\Big(\frac{\bar{q}_\xi\phi}{v\bar{v}}\Big)_\xi}_{L^2}^2
\le C\norm{\bar{q}_\xi\phi}_{L^2}^2.
\]
Recalling that $\bar{q}=q^R+q^C+q^S$ and using Lemmas \ref{lem:qR} and \ref{lem:VS} with \eqref{c-uq-der}, we then obtain
\[
\norm{K_{v^*}\Big(\frac{\bar{q}_\xi\phi}{v\bar{v}}\Big)_\xi}_{L^2}^2
\le C\norm{\bar{q}_\xi\phi}_{L^2}^2
\le C\left(\delta_R G^R+\delta_S G^S
+\delta_C\int_{\RR}W(t,\xi)^2 \abs{\phi}^2 d\xi\right).
\]
The fourth through ninth terms are estimated using Lemmas \ref{lem:uR}--\ref{lem:VS}, Lemma \ref{lem:greensfunction}, and the \textit{a priori} assumption \eqref{suppose} as follows:
\begin{align*}
\norm{K_{v^*}\beta(4\theta^3-4\bar{\theta}^3)\bar{\theta}_\xi}_{L^2}^2
&\le C\left(\delta_R G^R+\delta_S G^S
+\delta_C\int_{\RR} W(t,\xi)^2 \abs{\chi}^2 d\xi\right),\\[0.2em]
\norm{K_{v^*}\alpha \phi r}_{L^2}^2
&\le C\e_1 \norm{r}_{L^2}^2,\\[0.2em]
\norm{K_{v^*}\alpha \phi \bar{q}}_{L^2}^2
&\le C\left(\delta_R G^R+\delta_S G^S
+\delta_C\int_{\RR} W(t,\xi)^2 \abs{\phi}^2 d\xi\right),\\[0.2em]
\norm{K_{v^*}4\beta((\theta^*)^3-\theta^3)\chi_\xi}_{L^2}^2
&\le C(\delta_0+\e_1) \norm{\chi_{\xi}}_{L^2}^2,\\[0.2em]
\norm{K_{v^*}\Big[\Big(\frac{r_\xi}{\bar v}-\frac{r_\xi}{v^*}\Big)_\xi\Big]}_{L^2}^2
&=\norm{G_\xi*\Big(\frac{r_\xi (v^*-\bar v)}{\bar v v^*}\Big)}_{L^2}^2
\le C(\delta_0+\e_1) \norm{r_\x}_{L^2}^2,\\[0.2em]
\norm{K_{v^*}\alpha r (v^*-\bar v)}_{L^2}^2
&\le C\delta_0 \norm{r}_{L^2}^2.
\end{align*}
Combining all the above considerations proves the desired bound on \(\norm{E_2}_{L^2}^2=\norm{K_{v^*}E_1}_{L^2}^2\).

It remains to estimate \(\norm{E_3}_{L^2}^2\).
To this end, we first note that 
\[
\norm{E_3}_{L^2}^2
\le C\norm{\rd_\x K_{v^*}E_1}_{L^2}^2.
\]
In the preceding estimate for \(\norm{K_{v^*}E_1}_{L^2}^2\), at most one derivative falls on the operator \(K_{v^*}\).
Since this operator can accommodate up to two derivatives, each term in \(\partial_\xi(K_{v^*}E_1)\) can be estimated in the same way as above, putting the extra derivative onto \(K_{v^*}\).
This completes the proof.
\end{proof}

\subsection{Decomposition}
In order to prove \Cref{lem:l2est}, we consider the following decomposition of the right-hand side of the evolution identity in \Cref{lem:RHS}:
\begin{align}\label{eq:RHSdecomp}
\frac{d}{dt} \int_\RR a \thbar \et(U|\Ubar) d\x
=\dot{X} \sum_{i=1}^6 Y_i
+\sum_{i=1}^5 B_i+\sum_{i=1}^3S_i-G_1-G_2-D_1-D_2, 
\end{align}
where the decompositions of the \(Y\) term and the bad terms are
\begin{align*}
&Y_1=\int_{\RR}a(u^S)_\xi\psi d\xi,
&\hspace{5mm}
Y_2 &=\int_{\RR}a\frac{(v^S)_\xi\bar p}{\bar v}\phi d\xi, 
&\hspace{5mm}
Y_3 &=\int_{\RR}a\frac{R}{\gamma-1}\frac{(\theta^S)_\xi}{\bar \theta}\chi d\xi,\\
&Y_4 =-\int_{\RR}aR(\theta^S)_\xi\Phi\Big(\frac{v}{\bar v}\Big)d\xi, 
&\hspace{5mm}
Y_5 &=-\int_{\RR}a\frac{R}{\gamma-1}(\theta^S)_\xi\Phi\Big(\frac{\theta}{\bar \theta}\Big)d\xi,
&\hspace{5mm}
Y_6 &=-\int_{\RR}a_{\xi}\bar{\theta}\eta(U|\bar U)d\xi,
\end{align*}
\begin{align*}
&B_1=\int_{\RR}a(v^S)_\xi\left(\frac{(\gamma+1)p^*\sigma^*}{2(v^*)^2}\phi^2+\frac{R\sigma^*}{2v^*\theta^*}\chi^2-\frac{p^*\sigma^*}{v^*\theta^*}\phi \chi\right)d\xi,
&\hspace{-3mm}
B_2=\int_{\RR}a_\xi\psi(p-\bar{p})d\xi,\\[0.1em]
&B_3=\Jcal^{rad},
&\hspace{-13mm}
B_4=-\int_{\RR}a\bar{p}(u^C)_\xi \bigg(\Phi\Big(\frac{p}{\bar p}\Big)+\gamma \Phi\Big(\frac{v}{\bar v}\Big)\bigg)d\xi, 
\end{align*} \vspace{-4mm}
\begin{align*}
B_5=-\int_{\RR}a\bar{p}(u^S)_\xi\bigg(\Phi\Big(\frac{p}{\bar p}\Big)+\gamma \Phi\Big(\frac{v}{\bar v}\Big)\bigg)d\xi
-\int_{\RR}a(v^S)_\xi\bigg(\frac{(\gamma+1)p^*\sigma^*}{2(v^*)^2}\phi^2+\frac{R\sigma^*}{2v^*\theta^*}\chi^2-\frac{p^*\sigma^*}{v^*\theta^*}\phi \chi\bigg)d\xi,
\end{align*}
\begin{align*}
&S_1=-\int_{\RR}a\psi Q_1d\xi,
&\hspace{8.7mm}
S_2&=-\int_{\RR}a\left(\frac{\chi}{\thbar}
-\frac{\g-1}{R}\Hcal(U|\Ubar)\right) Q_2 d\xi,
&\hspace{8.7mm}
S_3&=-\int_{\RR}a\bar{A}_\theta r Q_3 d\xi,
\end{align*}
and the decomposition of the good terms is
\begin{align*}
&G_1=\sigma \int_{\RR}a_\xi \bar{\theta}\eta(U|\bar{U})d\xi,
&\hspace{20mm}
G_2&=\int_{\RR}a\bar{p}(u^R)_\xi \bigg(\Phi\Big(\frac{p}{\bar p}\Big)
+\gamma \Phi\Big(\frac{v}{\bar v}\Big)\bigg)d\xi, \\
&D_1=\frac{\a}{4\b}\int_\RR a\frac{v \thbar}{\th^5}r^2 d\x,
&\hspace{20mm}
D_2&=\frac{1}{4\b}\int_\RR a \frac{\thbar}{v\th^5} (r_\x)^2 d\x.
\end{align*}
Then, from \eqref{eq:shiftdef}, we find that
\begin{equation}\label{eq:shiftrep}
\dot{X}(t)=-\frac{M}{\delta_S}(Y_1+Y_2+Y_3),
\end{equation}
and thus it follows that
\begin{equation}\label{eq:shifttermrep}
\dot{X}(t)Y=-\frac{\delta_S}{M}|\dot{X}(t)|^2+\dot{X}(t)\sum_{i=4}^6Y_i.
\end{equation}
Using \eqref{eq:RHSdecomp} and \eqref{eq:shifttermrep} together with Young's inequality, we find that
\begin{multline}\label{eq:RHSdecomp2}
\frac{d}{dt} \int_\RR a \thbar \et(U|\Ubar) d\x
\leq -\frac{\delta_S}{2M}|\dot{X}|^2+B_1+B_2-G_1-\frac{3}{4}D_1-\frac{3}{4}D_2 \\ 
-\frac{\delta_S}{4M}|\dot{X}|^2+\frac{C}{\delta_S}\sum_{i=4}^6|Y_i|^2+\sum_{i=3}^5B_i+\sum_{i=1}^3S_i-G_2-\frac{1}{4}D_1-\frac{1}{4}D_2.
\end{multline}
The first line on the right-hand side contains the leading-order terms, whereas the second line consists of the lower-order terms.
We estimate the leading-order and lower-order terms in the following two subsections, respectively.

\subsection{Leading-Order Estimates}
We first estimate the leading-order terms in \eqref{eq:RHSdecomp2}.

\begin{lemma}\label{lem:L2estpart1}
There exist positive constants \(c\) and \(C\) such that
\begin{equation}\begin{aligned}\label{eq:L2estpart1mainresult}
&-\frac{\delta_S}{2M}|\dot{X}|^2
+B_1+B_2-G_1
-\frac{3}{4}D_1
-\frac{3}{4}D_2 \\
&\leq
-c(G^S+\norm{r}_{L^2}^2+\norm{r_\x}_{L^2}^2)
+C\bigg((\e_1+\delta_0)\norm{(\psi_\xi,\chi_\xi)}_{L^2}^2
+\norm{Q_3}_{L^2}^2
+\delta_R G^R \\
&\qquad
+\delta_C\int_{\RR}W(t,\xi)^2\abs{(\phi,\psi,\chi)}^2 d\xi
+(\delta_S)^{\frac43}((\delta_R)^{\frac43}+(\delta_C)^{\frac43})e^{-c\d_S t}\bigg).
\end{aligned}\end{equation}
\end{lemma}
\begin{proof}
To make the $a$-contraction structure evident, we rewrite all of the leading-order terms with respect to the new variables $y$ and $w$, which are defined as follows:
\begin{align}\label{eq:cov}
y\coloneqq\frac{v^S-v^*}{\delta_S}, \qquad
w(y)\coloneqq\zeta(\xi),
\end{align}
where $\zeta$ has been defined in \eqref{eq:zetadef}.
Since $\delta_S=v_+-v^* > 0$ and $v^S_\xi > 0$, the change of variable $\xi \mapsto y \in (0,1)$ is well-defined, and 
\[
\frac{dy}{d\xi}=\frac{v^S_\xi}{\delta_S}>0, \qquad
a=1+\lambda y, \qquad
a_\xi=\lambda \frac{dy}{d\x} > 0
\]
where \(\l \coloneqq \sqrt{\d_S}\) is the total variation of the weight.

\vspace{2mm}
\noindent\underline{\textbf{Step 1: Estimates on $B_2-G_1$}}
We first extract from \(B_2-G_1\) some good terms, which we call ``translators'' because they allow us to convert between \(\psi\), \(\phi\), and \(\chi\).
Firstly, recalling  \eqref{eq:relentdef} and using $p-\bar p=\frac{R}{v}\chi-\frac{\bar p}{v}\phi$ and $|\sigma-\sigma^*| \leq C\delta_S$, we obtain
\begin{align*}
&\psi(p-\bar p)-\sigma \bar {\theta} \eta(U|\bar U)
=\psi\left(\frac{R}{v}\chi-\frac{\bar p}{v}\phi\right)
-\sigma\left(R\bar \theta \Phi\Big(\frac{v}{\bar v}\Big)
+\frac{R \bar \theta}{\gamma-1}\Phi\Big(\frac{\theta}{\bar \theta}\Big)
+\frac{\psi^2}{2}\right) \\
&\qquad
\leq 
\psi\left(\frac{R}{v^*}\chi-\frac{p^*}{v^*}\phi\right)
-\sigma^*\left(\frac{R\theta^*}{2(v^*)^2}\phi^2
+\frac{R}{2(\gamma-1)\theta^*}\chi^2+\frac{\psi^2}{2}\right) \\
&\qquad\qquad
+C\Big(|(\phi,\chi)|+\delta_S+|(v^R-v_*, \theta^R-\theta_*)|+|(v^C-v^*, \theta^C-\theta^*)|\Big)|(\phi, \psi,\chi)|^2 \\
&\qquad
=-\frac{R\sigma^*\theta^*}{2(v^*)^2}
\left(\phi+\frac{\psi}{\sigma^*}\right)^2
-\frac{R\sigma^*}{2(\gamma-1)\theta^*}
\left(\chi-\frac{(\gamma-1)\theta^*}{v^*\sigma^*}\psi \right)^2 \\
&\qquad\qquad
+C\Big(|(\phi,\chi)|+\delta_S+|(v^R-v_*, \theta^R-\theta_*)|+|(v^C-v^*, \theta^C-\theta^*)|\Big)|(\phi, \psi,\chi)|^2, \\
\end{align*}
where the inequality follows from \eqref{suppose}, while the equality follows from the definition of \(\s^*\).\\
Thus, we have
\begin{equation}\begin{aligned}\label{eq:newgoodtermsdef}
&B_2-G_1
=\int_{\RR}a_\xi\Big(\psi(p-\bar p)-\sigma \bar {\theta} \eta(U|\bar U)\Big)d\xi\\
&\leq
-\frac{R\sigma^*\theta^*}{2(v^*)^2}\int_{\RR}a_\xi\left(\phi+\frac{\psi}{\sigma^*}\right)^2d\xi
-\frac{R\sigma^*}{2(\gamma-1)\theta^*}\int_{\RR}a_\xi\left(\chi-\frac{(\gamma-1)\theta^*}{v^*\sigma^*}\psi \right)^2d\xi \\
&\qquad
+\bigg(C\delta_S\int_{\RR}a_\xi|(\phi,\psi,\chi)|^2d\xi+C\int_{\RR}a_\xi|(\phi,\psi,\chi)|^3d\xi\\
&\qquad\qquad
+C\int_{\RR}a_\xi \Big(|(v^R-v_*, \theta^R-\theta_*)|+|(v^C-v^*,\theta^C-\theta^*)|\Big)\cdot|(\phi,\psi,\chi)|^2d\xi \bigg) \\
&\eqqcolon
-G_1^{new}-G_2^{new}+B_{new}.
\end{aligned}\end{equation}
The two new good terms \(G_1^{new}\) and \(G_2^{new}\) will be used to estimate the other terms, while the new bad term \(B_{new}\) can be estimated as follows.
First, by \eqref{weight}, we have 
\begin{equation}\label{eq:Bnew1}
C\delta_S\int_{\RR}a_\xi \abs{(\phi,\psi,\chi)}^2d\xi \leq C\lambda G^S.
\end{equation}
For the second term, using \eqref{weight} and \eqref{suppose} with the Gagliardo--Nirenberg inequality, we obtain
\begin{equation}\begin{aligned}\label{eq:Bnew2}
&\int_{\RR}a_\xi|(\phi,\psi,\chi)|^3d\xi\\
&\leq C\int_{\RR}|a_\xi| \abs{\phi+\frac{\psi}{\sigma^*}}^3d\xi
+C\int_{\RR}|a_\xi| |\psi|^3 d\xi 
+C\int_{\RR}|a_\xi| \abs{\chi-\frac{(\gamma-1)\theta^*}{v^*\sigma^*}\psi}^3d\xi \\
&\leq C\e_1(G_1^{new}+G_2^{new})
+C\frac{\lambda}{\delta_S} \norm{\psi}_{L^\infty}^2 \int_{\RR}\abs{(v^S)_\xi} |\psi|d\xi \\
&\leq C\e_1(G_1^{new}+G_2^{new})
+C\frac{\lambda}{\delta_S}\norm{\psi_{\xi}}_{L^2}\norm{\psi}_{L^2}
\sqrt{\int_{\RR}\abs{(v^S)_\xi} \psi^2 d\xi}\sqrt{\int_{\RR} \abs{(v^S)_\xi} d\xi} \\
&\leq C\e_1\left(G_1^{new}+G_2^{new}+\norm{\psi_\xi}_{L^2}^2+G^S\right).
\end{aligned}\end{equation}
Finally, for the last term, using \eqref{Riest}, \eqref{suppose} and the interpolation inequality, we find that
\begin{equation}\begin{aligned}\label{eq:Bnew3}
&C\int_{\RR}a_\xi\Big(|(v^R-v_*, \theta^R-\theta_*)|+|(v^C-v^*,\theta^C-\theta^*)|\Big) |(\phi,\psi,\chi)|^2d\xi \\
&\leq C(\delta_R+\delta_C)(G_1^{new}+G_2^{new}) \\
&\qquad
+C\frac{\lambda}{\delta_S}\int_{\RR} \abs{(v^S)_\x} \abs{\psi}^2
\abs{\big(v^R-v_*,\th^R-\th_*,v^C-v^*,\th^C-\th^*\big)}  d\xi\\
&\leq C\delta_0(G_1^{new}+G_2^{new})
+C\frac{\lambda}{\delta_S} \norm{\psi}_{L^4}^2
\norm{\abs{(v^S)_\x}\abs{\big(v^R-v_*,\th^R-\th_*,v^C-v^*,\th^C-\th^*\big)}}_{L^2} \\
&\leq C\delta_0(G_1^{new}+G_2^{new})
+C\frac{1}{\sqrt{\delta_S}} \norm{\psi_{\xi}}_{L^2}^{\frac12}\norm{\psi}_{L^2}^{\frac32}(\delta_S)^{\frac32}(\delta_R+\delta_C)e^{-c\d_S t} \\ 
&\leq C\delta_0(G_1^{new}+G_2^{new})+C\e_1 \norm{\psi_\xi}_{L^2}^2+C(\delta_S)^{\frac43}
\big((\delta_R)^{\frac43}+(\delta_C)^{\frac43}\big)e^{-c\d_S t}.
\end{aligned}\end{equation}
Collecting \eqref{eq:Bnew1}--\eqref{eq:Bnew3}, together with \eqref{eq:newgoodtermsdef}, and taking \(\l\), \(\e_1\), and \(\d_0\) small enough, we have
\begin{equation}\begin{aligned}\label{eq:B2-G1simplified}
B_2-G_1 &\leq 
-\frac34 (G_1^{new}+G_2^{new})
+C(\lambda+\e_1+\delta_0)G^S+
C(\e_1+\delta_0)\norm{\psi_\xi}_{L^2}^2 \\ 
&\qquad
+C(\delta_S)^{\frac43} \big((\delta_R)^{\frac43}+(\delta_C)^{\frac43}\big) e^{-c\d_S t}.
\end{aligned}\end{equation}

\noindent\underline{\textbf{Step 2: Estimates on $-\frac{\delta_S}{2M}|\dot{X}|^2$}}
Thanks to \eqref{eq:shiftrep}, we first observe \(Y_1\), \(Y_2\), and \(Y_3\) as follows:
For \(Y_1\), using $(u^S)_\xi=-\sigma (v^S)_\xi$, we have
\begin{align*}
Y_1
&=\int_{\RR}a(u^S)_\xi\frac{v^*\sigma^*}{(\gamma-1)\theta^*}\chi d\xi
+\int_{\RR}a(u^S)_\xi\Big(\psi-\frac{v^*\sigma^*}{(\gamma-1)\theta^*}\chi\Big)d\xi\\
&=- \frac{v^*(\sigma^*)^2}{(\gamma-1)\theta^*}
\int_{\RR}a(v^S)_\xi \chi d\xi
+ \frac{v^*\sigma^*}{(\gamma-1)\theta^*}
\int_{\RR}a\big((u^S)_\xi+\sigma^*(v^S)_\xi\big) \chi d\xi\\
&\qquad
-\s\int_{\RR}a (v^S)_\xi \Big(\psi-\frac{v^*\sigma^*}{(\gamma-1)\theta^*}\chi\Big)d\xi.
\end{align*}
Substituting \eqref{eq:chirep} into the above identity, we have
\begin{align*}
Y_1&=-\frac{v^*(\sigma^*)^2}{(\gamma-1)\theta^*}
\int_{\RR}a (v^S)_\xi \left(\frac{-\alpha v^*}{4\beta(\theta^*)^3}\zeta+\frac{r_\x}{4\beta (\theta^*)^3v^*}+E_3\right) d\xi \\
&\qquad
+\frac{v^*\sigma^*}{(\gamma-1)\theta^*}\int_{\RR}a\big((u^S)_\xi+\sigma^*(v^S)_\xi\big)
\left(\frac{-\alpha v^*}{4\beta(\theta^*)^3}\zeta+\frac{r_\x}{4\beta (\theta^*)^3v^*}+E_3\right) d\xi \\
&\qquad
-\sigma\int_{\RR}a (v^S)_\xi\Big(\psi-\frac{v^*\sigma^*}{(\gamma-1)\theta^*}\chi\Big)d\xi.
\end{align*}
Then, using \eqref{eq:cov}, \eqref{sm1}, \(\abs{a-1}\le \l\) and \(\abs{\bar{v}-v^*} \le \d_0\), we obtain
\begin{multline*}
\abs{Y_1-\frac{\gamma RK_1}{(\gamma-1)v^*}\delta_S\int_0^1wdy}\\
\leq C(\lambda+\delta_0) \delta_S \int_0^1|w|dy
+C\int_{\RR}|(v^S)_\xi|(\abs{r_\xi}+\abs{E_3})d\xi
+C\frac{\delta_S}{\lambda}\int_{\RR}a_\xi \abs{\chi-\frac{(\gamma-1)\theta^*}{v^*\sigma^*}\psi} d\xi,
\end{multline*}
where the constant \(K_1\) is defined by
\begin{equation}\label{eq:K1def}
K_1 \coloneqq \frac{\alpha v^*}{4\beta (\theta^*)^3}.  
\end{equation}
Similarly, the following holds for \(Y_2\) and \(Y_3\):
\begin{multline*}
\abs{Y_2-\frac{p^*K_1}{\theta^*(\gamma-1)}\delta_S\int_0^1wdy}
+\abs{Y_3-\frac{p^*K_1}{\theta^*}\delta_S\int_0^1wdy}
\le 
C(\lambda+\delta_0)\delta_S\int_0^1|w|dy\\
+C\int_{\RR} \abs{(v^S)_\xi} (|r_\xi|+|E_3|)d\xi
+C\frac{\delta_S}{\lambda}\int_{\RR}a_\xi
\abs{\chi-\frac{(\gamma-1)\theta^*}{v^*\sigma^*}\psi} d\xi
+C\frac{\delta_S}{\lambda}\int_{\RR}a_\xi \abs{\phi+\frac{\psi}{\sigma^*}}d\xi.
\end{multline*}
Thus, using the choice of the shift \eqref{eq:shiftrep} and the definition of \(p^*\), we obtain
\begin{align*}
\abs{\dot{X}+\frac{2\gamma RK_1M}{(\gamma-1)v^*}\int_0^1wdy}
&\leq CM(\lambda+\delta_0)\int_0^1|w|dy
+C\frac{M}{\delta_S}\int_{\RR}|(v^S)_\xi|(|r_\xi|+|E_3|)d\xi \\
&\qquad
+C\frac{M}{\lambda}\int_{\RR}a_\xi \abs{\chi-\frac{(\gamma-1)\theta^*}{v^*\sigma^*}\psi} d\xi
+C\frac{M}{\lambda}\int_{\RR}a_\xi \abs{\phi+\frac{\psi}{\sigma^*}}d\xi.
\end{align*}
Here, \(M>0\) is an \(\Ocal(1)\) constant to be specified later in \eqref{def:M}.\\
It then follows from the elementary inequality \(\frac{p^2}{2}-q^2 \leq (p-q)^2\) for all \(p, q\ge 0\) that
\begin{multline*}
\frac{2(\gamma RK_1M)^2}{(\gamma-1)^2(v^*)^2}\left(\int_0^1wdy\right)^2-|\dot{X}|^2\\
\leq C(\lambda+\delta_0)^2\int_0^1|w|^2dy
+C\delta_S\int_{\RR}(|r_\xi|^2+|E_3|^2)d\xi
+\frac{C}{\lambda}(G_1^{new}+G_2^{new}).
\end{multline*}
This implies the following estimate:
\begin{equation}\begin{aligned}\label{eq:shifttermsimplified}
-\frac{\delta_S}{2M}|\dot{X}|^2
&\leq -\frac{(\gamma RK_1)^2M}{(\gamma-1)^2(v^*)^2}\delta_S
\left(\int_0^1wdy\right)^2
+C\delta_S(\lambda+\delta_0)^2\int_0^1|w|^2dy \\
&\qquad
+C\delta_S\int_{\RR}(|r_\xi|^2+|E_3|^2)d\xi
+\frac{C \delta_S}{\lambda}(G_1^{new}+G_2^{new}).
\end{aligned}\end{equation}

\noindent\underline{\textbf{Step 3: Estimates on $B_1$}}
We now extract the main contribution in \(w\) from \(B_1\).
We first write the further decomposition.
\begin{align*}
B_1=\int_{\RR}a(v^S)_\xi\left(\frac{(\gamma+1)p^*\sigma^*}{2(v^*)^2}\phi^2+\frac{R\sigma^*}{2v^*\theta^*}\chi^2-\frac{p^*\sigma^*}{v^*\theta^*}\phi \chi\right)d\xi=:B_{11}+B_{12}+B_{13}.
\end{align*}
We estimate these terms one by one.
For \(B_{11}\), using Young's inequality and \(\abs{a-1}\le \l\), we have
\begin{align*}
B_{11}
&=\frac{(\gamma+1)p^*\sigma^*}{2(v^*)^2}\int_{\RR}a(v^S)_\xi
\left(-\frac{v^*}{(\gamma-1)\theta^*}\chi
+\frac{v^*}{(\gamma-1)\theta^*}\chi
-\frac{1}{\sigma^*}\psi+\frac{1}{\sigma^*}\psi+\phi\right)^2 d\xi \\
&\leq \frac{(\gamma+1)p^*\sigma^*}{2(v^*)^2}
\left(1+\lambda+C\Big(\frac{\delta_S}{\lambda}\Big)^{\frac12}\right)
\int_{\RR}(v^S)_\xi \abs{\frac{v^*}{(\gamma-1)\theta^*}\chi}^2 d\xi \\
&\qquad
+C\Big(\frac{\delta_S}{\lambda}\Big)^{-\frac{1}{2}}
\int_{\RR}(v^S)_\xi\bigg(\abs{\chi-\frac{(\gamma-1)\theta^*}{v^*\sigma^*}\psi}^2+\abs{\phi+\frac{1}{\sigma^*}\psi}^2\bigg)d\xi \\
&\leq \frac{(\gamma+1)p^*\sigma^*}{2(v^*)^2}
\left(1+\lambda+C\Big(\frac{\delta_S}{\lambda}\Big)^{\frac12}\right)
\int_{\RR}(v^S)_\xi \abs{\frac{v^*}{(\gamma-1)\theta^*}\chi}^2d\xi
+C\Big(\frac{\delta_S}{\lambda}\Big)^{\frac{1}{2}}(G_1^{new}+G_2^{new}).
\end{align*}
Now, using \eqref{eq:chirep}, we replace \(\chi\) by \(\z\) to attain
\begin{align*}
&\frac{(\gamma+1)p^*\sigma^*}{2(v^*)^2}
\left(1+\lambda+C\Big(\frac{\delta_S}{\lambda}\Big)^{\frac12}\right)
\int_{\RR}(v^S)_\xi \abs{\frac{v^*}{(\gamma-1)\theta^*}\chi}^2d\xi \\
&=\frac{(\gamma+1)p^*\sigma^*}{2(\gamma-1)^2(\theta^*)^2}
\left(1+\lambda+C\Big(\frac{\delta_S}{\lambda}\Big)^{\frac12}\right)
\int_{\RR}(v^S)_\xi\left(-\frac{\alpha v^*}{4 \beta (\theta^*)^3}\zeta+\frac{1}{4\beta(\theta^*)^3v^*}r_\xi+E_3\right)^2d\xi \\
&\leq \frac{(\gamma+1)p^*\sigma^*}{2(\gamma-1)^2(\theta^*)^2}
\left(1+\lambda+C\Big(\frac{\delta_S}{\lambda}\Big)^{\frac12}\right)
\int_{\RR}(v^S)_{\xi}(K_1^2+C\delta_0)\zeta^2d\xi
+\frac{C}{\delta_0}\int_{\RR}(v^S)_\xi\left(|r_\xi|^2+|E_3|^2\right)d\xi \\
&\leq \frac{(\gamma+1)p^*\sigma^*K_1^2}{2(\gamma-1)^2(\theta^*)^2}
\left(1+\lambda+C\Big(\frac{\delta_S}{\lambda}\Big)^{\frac12}+C\delta_0\right)
\int_{\RR}(v^S)_\xi \zeta^2d\xi
+\frac{C}{\delta_0}\int_{\RR}(v^S)_\xi\left(|r_\xi|^2+|E_3|^2\right)d\xi,
\end{align*}
where we recall \eqref{eq:K1def} and have applied Young's inequality.
So in total, we obtain 
\begin{equation}\begin{aligned}\label{eq:B11simplified}
B_{11}
&\leq \frac{(\gamma+1)p^*\sigma^*K_1^2}{2(\gamma-1)^2(\theta^*)^2}
\left(1+\lambda+C\Big(\frac{\delta_S}{\lambda}\Big)^{\frac12}+C\delta_0\right)
\int_{\RR}(v^S)_\xi \zeta^2d\xi \\ 
&\qquad
+\frac{C}{\delta_0}\int_{\RR}(v^S)_\xi\left(|r_\xi|^2+|E_3|^2\right)d\xi
+C\Big(\frac{\delta_S}{\lambda}\Big)^{\frac{1}{2}}(G_1^{new}+G_2^{new}).
\end{aligned}\end{equation}
For $B_{12}$, similar considerations using \eqref{eq:chirep} give
\begin{equation}\begin{aligned}\label{eq:B12simplified}
B_{12}
&=\int_{\RR}a(v^S)_\xi\frac{R\sigma^*}{2v^*\theta^*}\chi^2d\xi
=\int_{\RR}a(v^S)_\xi\frac{R\sigma^*}{2v^*\theta^*}\left(\frac{-\alpha v^*}{4\beta(\theta^*)^3}\zeta+\frac{r_\x}{4\beta (\theta^*)^3v^*}+E_3\right)^2d\xi \\
&\leq \frac{R \sigma^* K_1^2}{2v^*\theta^*}(1+\lambda+C\delta_0)
\int_{\RR}(v^S)_\xi\zeta^2d\xi
+\frac{C}{\delta_0}\int_{\RR}(v^S)_\xi\left(|r_\xi|^2+|E_3|^2\right)d\xi.
\end{aligned}\end{equation}
Finally, for $B_{13}$, we may have
\begin{equation}\begin{aligned}\label{eq:B13simplified}
&B_{13}
=\frac{p^*\sigma^*}{v^*\theta^*}\int_{\RR}a(v^S)_\xi\left(-\phi-\frac{\psi}{\sigma^*}+\frac{\psi}{\sigma^*}-\frac{v^* \chi}{(\gamma-1)\theta^*}+\frac{v^* \chi}{(\gamma-1)\theta^*}\right)\chi d\xi \\
&\leq \frac{p^*\sigma^*}{v^*\theta^*}
\int_{\RR}a (v^S)_\xi\frac{v^* \chi^2}{(\gamma-1)\theta^*} d\xi
+C\int_{\RR}(v^S)_\xi \left(\abs{\phi+\frac{\psi}{\sigma^*}}+\abs{\frac{\psi}{\sigma^*}-\frac{v^* \chi}{(\gamma-1)\theta^*}}\right) \abs{\chi} d\xi \\
&\leq \frac{p^* \sigma^*}{(\theta^*)^2(\gamma-1)}\left(1+\lambda+C\Big(\frac{\delta_S}{\lambda}\Big)^{\frac12}\right)\int_{\RR}(v^S)_\xi \chi^2 d\xi+C\Big(\frac{\delta_S}{\lambda}\Big)^{\frac{1}{2}}(G_1^{new}+G_2^{new}) \\
& \leq \frac{p^* \sigma^*K_1^2}{(\theta^*)^2(\gamma-1)}\left(1+\lambda+C\Big(\frac{\delta_S}{\lambda}\Big)^{\frac12}
+C\delta_0 \right)\int_{\RR}(v^S)_\xi \zeta^2 d\xi\\
&\qquad
+\frac{C}{\delta_0}\int_{\RR}(v^S)_\xi\left(|r_\xi|^2+|E_3|^2\right)d\xi
+C\Big(\frac{\delta_S}{\lambda}\Big)^{\frac{1}{2}}(G_1^{new}+G_2^{new}).
\end{aligned}\end{equation}
Noting from \eqref{eq:A^*def} that 
\[
\frac{\alpha^2 R \gamma(\gamma+1)v^* \sigma^*}{32\beta^2 (\gamma-1)^2 (\theta^*)^7}
=\frac{\alpha v^* A^*}{4\beta (\theta^*)^4},
\]
combining \eqref{eq:B11simplified}--\eqref{eq:B13simplified} and using the change of variables \eqref{eq:cov}, we obtain
\begin{equation}\begin{aligned}\label{eq:B1simplified}
B_1 &\leq \frac{\alpha v^* A^*}{4\beta (\theta^*)^4}
\left(1+\lambda+C\Big(\frac{\delta_S}{\lambda}\Big)^{\frac12}+C\delta_0 \right)\delta_S\int_0^1w^2dy \\
&\qquad
+\frac{C}{\delta_0}\int_{\RR}(v^S)_\xi\left(|r_\xi|^2+|E_3|^2\right)d\xi
+C\Big(\frac{\delta_S}{\lambda}\Big)^{\frac{1}{2}}(G_1^{new}+G_2^{new}).
\end{aligned}\end{equation}

\noindent\underline{\textbf{Step 4: Estimates on $D_1$}}
In this step, it becomes apparent that the good term \(D_1\) plays the role of the derivative square term in \Cref{lem:poincare} to cancel the main bad term (the first term on the right-hand side of \eqref{eq:B1simplified}).
Indeed, using \eqref{eq:E2def}, we have
\begin{align*}
D_1&
=\frac{\alpha}{4\beta}\int_{\RR}a\frac{v\bar \theta}{\theta^5}(\zeta_\xi-E_2)^2d\xi
=\frac{\alpha}{4\beta}\left(\int_{\RR}a\frac{v\bar \theta}{\theta^5}|\zeta_\xi|^2d\xi
-2\int_{\RR}a\frac{v\bar \theta}{\theta^5}\zeta_\xi E_2d\xi
+\int_{\RR}a\frac{v\bar \theta}{\theta^5}|E_2|^2d\xi\right).
\end{align*}
Since the coefficient $a\frac{v\bar \theta}{\theta^5}$ is uniformly bounded, we apply Young's inequality to find that 
\[
D_1
\geq (1-\nu)\frac{\alpha}{4\beta}\int_{\RR}a\frac{v\bar \theta}{\theta^5}|\zeta_\xi|^2d\xi-C_\nu\norm{E_2}_{L^2}^2,
\]
where \(\n>0\) is a small parameter arising from Young's inequality, and \(C_\n\) is the corresponding large constant.
Then, using \(\abs{a-1}\le \l\) and \(|v-v^*|+|\th-\th^*|+|\thbar-\th^*| \le C(\d_0+\e_1)\), we have 
\[
D_1
\ge \Big(1-\nu-C(\delta_0+\e_1+\lambda)\Big) \frac{\alpha v^*}{4 \beta (\theta^*)^4}
\int_{\RR}|\zeta_\xi|^2 d\xi
-C_\nu \norm{E_2}_{L^2}^2.
\]
Now, using the change of variables \eqref{eq:cov} together with \eqref{tail-0}, we have
\begin{align*}
D_1 &\geq \Big(1-\nu-C(\delta_0+\e_1+\lambda)\Big)
\frac{\alpha v^*}{4 \beta (\theta^*)^4}
\int_0^1|w_y|^2\frac{dy}{d\xi}dy
-C_\nu \norm{E_2}_{L^2}^2 \\
&\geq \Big(1-\nu-C(\delta_0+\e_1+\lambda)\Big)
\frac{\alpha v^*(A^*-C\delta_S)}{4 \beta (\theta^*)^4}\delta_S
\int_0^1y(1-y)|w_y|^2 dy
-C_\nu \norm{E_2}_{L^2}^2.
\end{align*}
Finally, using the weighted Poincar\'e inequality \Cref{lem:poincare}, alongside 
\[
\int_0^1|w-\bar{w}|^2dy=\int_0^1w^2dy-\bar{w}^2, \qquad \text{where }
\bar{w}=\int_0^1 w dy, 
\]
we obtain
\begin{equation}\label{eq:D1simplified}
D_1 \geq \Big(1-\nu-C(\delta_0+\e_1+\lambda)\Big)
\frac{\alpha v^*A^*}{2 \beta (\theta^*)^4}\delta_S
\left(\int_0^1w^2dy-\bar{w}^2\right)
-C_\nu \norm{E_2}_{L^2}^2.
\end{equation}

\noindent\underline{\textbf{Conclusion:}}
Combining \eqref{eq:B2-G1simplified}, \eqref{eq:shifttermsimplified}, \eqref{eq:B1simplified}, and \eqref{eq:D1simplified}, we have
\begin{align*}
&-\frac{\delta_S}{2M}|\dot{X}|^2+B_1+B_2-G_1-\frac{3}{4}D_1-\frac{3}{4}D_2\\
&\leq-\frac34 (G_1^{new}+G_2^{new})
-\frac{\a}{48\b}\int_\RR a\frac{v \thbar}{\th^5}r^2 d\x
-\frac{1}{48\b}\int_\RR a \frac{\thbar}{v\th^5} (r_\x)^2 d\x\\
&\qquad-\frac{(\gamma RK_1)^2M}{(\gamma-1)^2(v^*)^2}\delta_S\bar w^2
+\frac{\alpha v^* A^*}{4\beta (\theta^*)^4}
\left(1+\lambda+C\Big(\frac{\delta_S}{\lambda}\Big)^{\frac12}+C\delta_0 \right)\delta_S \int_0^1w^2dy \\
&\qquad
-\frac{2}{3}\Big(1-\nu-C(\delta_0+\e_1+\lambda)\Big)
\frac{\alpha v^*A^*}{2 \beta (\theta^*)^4}\delta_S
\left(\int_0^1w^2dy-\bar{w}^2\right) \\ 
&\qquad
+C(\e_1+\delta_0)\norm{\psi_\xi}_{L^2}^2
+C(\delta_S)^{\frac43}\Big((\delta_R)^{\frac43}+(\delta_C)^{\frac43}\Big)e^{-c\d_S t}
+C(\lambda+\e_1+\delta_0)G^S\\
&\qquad
+C\delta_0\int_{\RR}(|r_\xi|^2+|E_3|^2)d\xi
+C\Big(\frac{\delta_S}{\lambda}\Big)^{\frac{1}{2}}(G_1^{new}+G_2^{new})
+C_\nu \norm{E_2}_{L^2}^2.
\end{align*}
Let \(M\) be the \(\Ocal(1)\) constant defined by
\begin{equation}\label{def:M}
M:=\frac{(\gamma-1)^2(v^*)^3\alpha A^*}{2(\gamma R K_1)^2\beta (\theta^*)^4}.
\end{equation} 
Then, choosing \(\n\), \(\d_0\), \(\e_1\) and \(\l\) sufficiently small, we may have
\begin{align*}
&-\frac{\delta_S}{2M}|\dot{X}|^2+B_1+B_2-G_1-\frac{3}{4}D_1-\frac{3}{4}D_2\\
&\leq -\frac12 (G_1^{new}+G_2^{new})
-\frac{\a}{48\b}\int_\RR a\frac{v \thbar}{\th^5}r^2 d\x
-\frac{1}{48\b}\int_\RR a \frac{\thbar}{v\th^5} (r_\x)^2 d\x 
-\frac{\alpha v^*A^*}{24 \beta (\theta^*)^4}\delta_S\int_0^1w^2dy\\
&\qquad
+C(\e_1+\delta_0) \norm{\psi_\xi}_{L^2}^2
+C(\delta_S)^{\frac43}\Big((\delta_R)^{\frac43}+(\delta_C)^{\frac43}\Big)e^{-c\d_S t}
+C(\lambda+\e_1+\delta_0)G^S \\ 
&\qquad
+C\delta_0\int_{\RR}(|r_\xi|^2+|E_3|^2)d\xi
+C\norm{E_2}_{L^2}^2.
\end{align*}
Next, applying the representation \eqref{eq:chirep} and $|\bar v-v^*| \leq \delta_0$, we obtain
\begin{align*}
\delta_S\int_0^1w^2dy
=\int_{\RR}|(v^S)_\xi|\zeta^2d\xi
&=\int_{\RR}|(v^S)_\xi|\abs{\frac{-4\beta (\theta^*)^3}{\alpha  v^*}\chi+\frac{1}{\alpha (v^*)^2}r_\xi+\frac{4\beta (\theta^*)^3}{\alpha v^*}E_3}^2 d\xi \\
& \geq c\int_{\RR}|(v^S)_\xi|\chi^2d\xi
-C\int_{\RR}|(v^S)_\xi|(|r_\xi|^2+|E_3|^2)d\xi.
\end{align*}
Using the translation inequalities
\begin{align*}
\int_{\RR}|(v^S)_\xi|\phi^2d\xi
&\le 2 \int_{\RR}|(v^S)_\xi|\left(\phi+\frac{\psi}{\sigma^*}\right)^2d\xi
+2 \int_{\RR}|(v^S)_\xi|\left(\frac{\psi}{\sigma^*}\right)^2 d\xi\\
&\le C\frac{\delta_S}{\lambda}G_1^{new}
+C\int_{\RR}|(v^S)_\xi|\abs{\psi}^2d\xi,\\
\int_{\RR}|(v^S)_\xi|\psi^2d\xi
&\le 2 \int_{\RR}|(v^S)_\xi|\left(\psi-\frac{v^*\sigma^*}{(\gamma-1)\theta^*}\chi\right)^2d\xi+2\int_{\RR}|(v^S)_\xi|\left(\frac{v^*\sigma^*}{(\gamma-1)\theta^*}\chi\right)^2d\xi\\
&\le C\frac{\delta_S}{\lambda}G_2^{new}
+C\int_{\RR}|(v^S)_\xi|\abs{\chi}^2 d\xi,
\end{align*}
and taking \(\d_S\) (and \(\l=\sqrt{\d_S}\)) small enough, we have
\begin{align*}
&-\frac{\delta_S}{2M}|\dot{X}|^2+B_1+B_2-G_1-\frac{3}{4}D_1-\frac{3}{4}D_2\\
&\leq -\frac14 (G_1^{new}+G_2^{new})
-\frac{\a}{48\b}\int_\RR a\frac{v \thbar}{\th^5}r^2 d\x
-\frac{1}{48\b}\int_\RR a \frac{\thbar}{v\th^5} (r_\x)^2 d\x
-c G^S
+C(\e_1+\delta_0)\norm{\psi_\xi}_{L^2}^2\\
&\qquad
+C(\delta_S)^{\frac43}\Big((\delta_R)^{\frac43}+(\delta_C)^{\frac43}\Big)e^{-c\d_S t} 
+C\delta_0\int_{\RR}(|r_\xi|^2+|E_3|^2)d\xi
+C\norm{E_2}_{L^2}^2.
\end{align*}
Finally, using the smallness of \(\d_0\) and \(\e_1\) with \Cref{lem:errorsL2}, we obtain exactly \Cref{lem:L2estpart1}.
\end{proof}

\subsection{Lower-Order Estimates}
We now provide the estimate on the lower-order terms in \eqref{eq:RHSdecomp2}.
\begin{lemma}\label{lem:L2estpart2}
There exist positive constants \(c\) and \(C\) such that
\begin{equation}\begin{aligned}\label{eq:L2estpart2mainresult}
&\frac{C}{\delta_S}\sum_{i=4}^6|Y_i|^2
+\sum_{i=3}^5B_i
+\sum_{i=1}^3S_i-G_2-\frac{1}{4}D_1-\frac{1}{4}D_2 \\
&\leq -cG^R+C\bigg((\e_1+\d_0+\sqrt{\delta_S})G^S
+\delta_C\int_{\RR}W^2|(\phi, \psi, \chi)|^2d\xi
+\delta_C(1+t)^{-\frac{5}{4}}\norm{\psi}_{L^2} \\
&\hspace{28mm}
+(\d_R)^{-\frac{1}{6}} \norm{Q_2^R}_{L^1}^{\frac43}\norm{(\phi,\chi)}_{L^2}^{\frac23}
+\sqrt{\d_R} \norm{(\phi_\xi, \chi_\xi)}_{L^2}^2
+\norm{Q_2^I}_{L^2}\norm{(\phi,\chi)}_{L^2} \\
&\hspace{28mm}
+\norm{Q_1^I}_{L^2}\norm{\psi}_{L^2}
+\norm{Q_3^R}_{L^2}^2
+\norm{Q_3^C}_{L^2}^2
+\norm{Q_3^I}_{L^2}^2\bigg).
\end{aligned}\end{equation}
\end{lemma}

\begin{proof} We estimate the first four terms on the left-hand side of \eqref{eq:L2estpart2mainresult} one by one.

\vspace{2mm}
\noindent\underline{\textbf{Step 1: Estimates on \(\frac{C}{\delta_S}\sum_{i=4}^6 \abs{Y_i}^2\)}}
By \eqref{ratio-vth}, 
\[
\abs{(Y_4,Y_5)}
\le C\int_\RR \abs{(v^S)_\x} \abs{(\phi,\chi)}^2 d\x
\]
Using \eqref{derivdecay} alongside the \textit{a priori} assumption \eqref{suppose}, we obtain 
\[
\frac{C}{\d_S}\abs{(Y_4,Y_5)}^2
\le C\d_S \e_1^2 G^S.
\]
For the term with \(Y_6\), we compute 
\begin{align*}
\frac{C}{\delta_S}|Y_6|^2 
&\leq \frac{C\lambda^2}{\delta_S^3}
\left(\int_{\RR}\abs{(v^S)_\xi} \abs{(\phi,\psi,\chi)}^2d\xi\right)^2 \\
&\leq \frac{C\lambda^2}{\delta_S}
\norm{(\phi,\psi,\chi)}_{L^2}^2\int_{\RR} \abs{v^S_\xi} \abs{(\phi,\psi,\chi)}^2d\xi
\leq C\e_1^2G^S.
\end{align*}
Thus, we have
\begin{equation}\label{eq:Y4-Y6simplified}
\frac{C}{\delta_S}\sum_{i=4}^6 \abs{Y_i}^2
\leq C\e_1 G^S.
\end{equation}

\noindent\underline{\textbf{Step 2: Estimates on \(B_3\)}}
We consider the following decomposition: 
\begin{multline*}
B_3
=\int_\RR a_\x \frac{\chi}{\th}r d\x
+\int_\RR a \left(\frac{\chi^2}{\th\thbar} -\frac{\g-1}{R}\Hcal(U|\Ubar) \right) \qbar_\x d\x
+\int_\RR a \frac{\chi}{\th\thbar} r \thbar_\x d\x
- \int_\RR (a A_\th)_\x \frac{r r_\x}{v} d\x\\
+ \int_\RR (a A_\th r)_\x \frac{\phi}{v\vbar} \qbar_\x d\x
- \frac{\a}{4 \b} \int_\RR a \thbar \qbar r \Big(\frac{v}{\th^5}-\frac{\vbar}{\thbar^5}\Big) d\x
+ \int_\RR a (A_\th-\Abar_\th) \Big(\frac{\qbar_\x}{\vbar}\Big)_\x r d\x=:\sum_{i=1}^7J_i.
\end{multline*}
To control $J_1$, using Young's inequality and \eqref{derivdecay}, we have 
\[
\abs{J_1}
\leq C\frac{\lambda}{\delta_S}\int_{\RR} \abs{(v^S)_\xi} \abs{\chi r} d\xi
\leq C\lambda
\int_{\RR}\Big(\sqrt{\abs{(v^S)_\xi}}\abs{\chi r}\Big) d\xi
\leq C\sqrt{\delta_S}\bigg(G^S+\int_{\RR}r^2 d\x\bigg).
\]
For $J_2$, using \eqref{qhder} and \eqref{c-uq-der}, we obtain
\begin{align*}
|J_2| 
&\leq C\int_{\RR}(|(q^S)_\xi|+|(q^C)_\xi|)|(\phi,\chi)|^2d\xi
+\abs{\int_{\RR}a \left(\frac{\chi^2}{\th\thbar} -\frac{\g-1}{R}\Hcal(U|\Ubar) \right) (q^R)_\x d\x} \\
&\leq C\left(\delta_S G^S+\delta_C \int_{\RR} W^2 \abs{(\chi,\phi)}^2 d\xi\right)
+\abs{\int_{\RR}a \left(\frac{\chi^2}{\th\thbar} -\frac{\g-1}{R}\Hcal(U|\Ubar) \right) (q^R)_\x d\x}.
\end{align*}
Integrating by parts and using \eqref{derivdecay} and \Cref{lem:qR}, we estimate the second term on the right-hand side as follows:
\begin{align*}
&\abs{\int_{\RR}a \left(\frac{\chi^2}{\th\thbar} -\frac{\g-1}{R}\Hcal(U|\Ubar) \right) q^R_\x d\x}
=\abs{\int_{\RR}\bigg[a \left(\frac{\chi^2}{\th\thbar} -\frac{\g-1}{R}\Hcal(U|\Ubar) \right)\bigg]_{\xi} q^R d\x}\\
&\leq C\delta_S G^R
+C\int_{\RR}v^R_\xi \bigg(\abs{\chi \chi_\xi}
+(\abs{\theta_\xi}
+\abs{\bar{\theta}_\xi})\abs{\chi}^2
+ \abs{\bar{v}_\xi} \abs{\phi}^2
+\abs{\phi \phi_\xi}
+\abs{\bar{\theta}_\xi} \abs{\chi}^2
+\abs{\chi \chi_\xi}\bigg)d\xi \\
&\leq \left(C(\delta_0+\e_1)+\frac{c^*}{8}\right)G^R+C\delta_R\norm{(\phi_\xi, \chi_\xi)}_{L^2}^2,
\end{align*}
where \(c^*\) is the constant which will be defined later in \eqref{eq:c^*def} as the coefficient of the good term \(-c^* G^R\) in \eqref{eq:G2simplified}.
So in total, using the smallness of \(\d_0\) and \(\e_1\), we have 
\[
\abs{J_2}
\le C \left(\delta_S G^S+\delta_C \int_{\RR}W^2 \abs{(\chi,\phi)}^2d\xi
+\delta_R\norm{(\phi_\xi, \chi_\xi)}_{L^2}^2\right)
+\frac{c^*}{4}G^R.
\]
For $J_3$, by Young's inequality, we find that 
\[
\abs{J_3}
\le C \int_{\RR}\big(\bar{\theta}_\xi\big)^2 \chi^2d\xi
+\frac{c_*}{8} \int_{\RR}r^2d\xi,
\]
where the constant \(c_*\) is given by
\begin{equation}\label{eq:c_*def}
c_*
\coloneqq \frac{\alpha v^*}{32 \beta (\theta^*)^4}.
\end{equation}
Then, using the decomposition $\bar{\theta}_\xi=(\theta^S)_\xi+(\theta^C)_\xi+(\theta^R)_\xi$ alongside Lemmas \ref{lem:uR}--\ref{lem:VS}, we have 
\[
\int_{\RR}\big(\bar{\theta}_\xi\big)^2 \chi^2d\xi
\leq C\delta_S G^S
+C\delta_R G^R
+C\delta_C \int_{\RR}W^2\chi^2d\xi.
\]
Thus, we obtain 
\[
\abs{J_3}
\le C\left(\delta_S G^S +\delta_R G^R +\delta_C \int_{\RR} W^2 \chi^2 d\xi\right)
+\frac{c_*}{8} \int_{\RR}r^2 d\xi.
\]
To estimate \(J_4\), since \(A_\th=\frac{\thbar}{4\b \th^5}\) and \(\abs{(a A_\th)_\x} \le C(\d_0+\e_1)\), we find that
\[
\abs{J_4}
\le C(\d_0+\e_1) \left(\int_{\RR}r^2 d\xi
+\int_{\RR}(r_\xi)^2 d\xi\right).
\]
For \(J_5\), we first note that
\[
\abs{J_5}
\le \int_\RR \Big(\abs{a_\x A_\th r}+\abs{a (A_\th)_\x r} +\abs{a A_\th r_\x}\Big)
\abs{\frac{\phi}{v \bar v}\bar{q}_\xi} d\x.
\]
Then, using Lemmas \ref{lem:uR}--\ref{lem:VS} and the fact that \(\abs{(A_\theta)_\xi} \le C(\delta_0+\e_1)\), we obtain
\begin{align*}
\int_{\RR} \abs{a_\xi A_\theta \frac{r\phi}{v \bar v}\bar{q}_\xi} d\x
&\le C \frac{\lambda}{\delta_S}\delta_0 \int_{\RR} \abs{(v^S)_\xi} \abs{r \phi} d\x
\le C \Big(\delta_S G^S +\delta_0 \int_{\RR} r^2 d\xi\Big),\\
\int_\RR \abs{a (A_\th)_\x \frac{r\phi}{v \bar v}\bar{q}_\xi} d\x
&\le C(\delta_0+\e_1) \int_{\RR}r^2 d\xi
+C(\delta_0+\e_1) \int_{\RR}\abs{\bar {q}_\xi} \phi^2d\xi\\
&\le C(\delta_0+\e_1)\int_{\RR}r^2d\xi+C\Big(\delta_SG^S+(\delta_0+\e_1)G^R+\delta_C\int_{\RR}W^2|\phi|^2d\xi\Big),\\
\int_\RR \abs{a A_\th \frac{r_\x \phi}{v \bar v}\bar{q}_\xi} d\x
&\le C\int_{\RR}\abs{\bar{q}_\xi} \abs{r_\xi}^2 d\xi
+C\Big(\delta_S G^S+\delta_C \int_{\RR}W^2 \abs{(\chi,\phi)}^2 d\xi\Big)
+\frac{c^*}{8}\int_{\RR} \abs{(v^R)_\xi} \phi^2 d\xi\\
&\le C\delta_0 \int_{\RR} \abs{r_\xi}^2 d\xi
+C\Big(\delta_S G^S +\delta_C\int_{\RR} W^2 \abs{(\chi,\phi)}^2 d\xi\Big)
+\frac{c^*}{8}G^R,
\end{align*}
where \(c^*\) is the same constant used in the estimate for \(J_2\) and will be determined later in \eqref{eq:c^*def}.
Thus, combining the above estimates, we have 
\[
\abs{J_5}
\le C\Big(\delta_S G^S +\delta_C \int_{\RR}W^2 \abs{\phi}^2 d\xi
+(\delta_0+\e_1)\int_{\RR} r^2 d\xi+ \delta_0 \int_{\RR}\abs{r_\xi}^2 d\xi\Big)
+\frac{c^*}{4}G^R.
\]
For \(J_6\), using Lemmas \ref{lem:uR}--\ref{lem:VS} again, we compute
\begin{align*}
|J_6|
&\leq C\int_{\RR}\abs{\bar{q}r\Big(\frac{v}{\theta^5}-\frac{\bar v}{\bar{\theta}^5}\Big)}d\xi
\leq C\int_{\RR}\abs{q^S+q^C+q^R} \abs{r} \abs{(\phi,\chi)} d\xi\\
&\leq \frac{c_*}{8}\int_{\RR}r^2d\xi
+C\int_{\RR}(q^S+q^C+q^R)^2|(\phi,\chi)|^2d\xi \\
&\leq \frac{c_*}{8}\int_{\RR}r^2d\xi
+C\Big(\delta_S G^S+\delta_R G^R+\delta_C \int_{\RR}W^2 \abs{(\chi,\phi)}^2 d\xi\Big).
\end{align*}
Finally, for \(J_7\), using $\abs{A_\theta-A_{\bar \theta}} \le C\abs{\chi}$ and Lemmas \ref{lem:uR}--\ref{lem:VS}, we obtain
\begin{align*}
\abs{J_7}
&\le C\int_{\RR} \abs{\Big(\frac{\bar{q}_\xi}{\bar v}\Big)_\xi \chi r} d\xi
\le C\int_{\RR}\ \big(\abs{\bar{q}_{\xi\xi}} +\abs{\bar{q}_\xi \bar{v}_\xi}\big) \abs{\chi r} d\x\\
&\le \frac{c_*}{8}\int_{\RR}r^2 d\xi
+C\int_{\RR}\ \big(\abs{\bar{q}_{\xi\xi}} +\abs{\bar{q}_\xi \bar{v}_\xi}\big)^2 \chi^2 d\x\\
&\le \frac{c_*}{8}\int_{\RR}r^2 d\xi
+C\Big(\delta_S G^S+\delta_R G^R+\delta_C\int_{\RR}W^2 \abs{(\chi,\psi)}^2 d\xi\Big).
\end{align*}
So, combining the bounds on the \(J_i\)'s, \(i=1,\ldots,7\) and using the smallness of \(\d_0\) and \(\e_1\) gives
\begin{multline}\label{eq:B3simplified}
\abs{B_3} \leq
\frac{c_*}{2}\int_{\RR}r^2 d\xi
+\frac{3c^*}{4}G^R\\
+C\bigg((\delta_0+\e_1)\int_{\RR}\abs{r_\xi}^ 2d\xi
+\sqrt{\delta_S} G^S +\delta_C \int_{\RR}W^2 \abs{(\phi, \psi, \chi)}^2 d\xi
+\delta_R \norm{(\phi_\xi, \chi_\xi)}_{L^2}^2\bigg).
\end{multline}

\noindent\underline{\textbf{Step 3: Estimates on \(B_4\)}}
We simply use \eqref{c-uq-der} to have
\begin{equation}\label{eq:B4simplified}
\abs{B_4}
\le C\int_\RR \abs{(u^C)_\xi} \abs{(\phi,\chi)}^2 d\x
\le C\delta_C \int_{\RR} W(t,\xi)^2 \abs{(\phi,\chi)}^2d\xi.
\end{equation}

\noindent\underline{\textbf{Step 4: Estimates on \(B_5\)}}
Using \eqref{ratio-vu} and \eqref{suppose}, we observe
\begin{align*}
&-\int_{\RR}a\bar{p}(u^S)_\xi\Big(\Phi\Big(\frac{p}{\bar p}\Big)
+\gamma \Phi\Big(\frac{v}{\bar v}\Big)\Big)d\xi \\
&=-\int_{\RR}a \bar p (u^S)_\xi
\bigg(\Phi\Big(\frac{p}{\bar p}\Big)
-\frac{(p- \bar p)^2}{2\pbar^2}
+\frac{(p- \bar p)^2}{2\pbar^2}
+\gamma \Phi\Big(\frac{v}{\bar v}\Big)
-\gamma \frac{\phi^2}{2\bar{v}^2}+\gamma \frac{\phi^2}{2\bar{v}^2}\bigg)d\xi \\
&\leq \sigma^* p^* \int_{\RR} a \abs{(v^S)_\xi}
\bigg(\frac{(p- \bar p)^2}{2\pbar^2}
+\gamma \frac{\phi^2}{2\bar{v}^2}\bigg) d\xi
+C(\e_1+\delta_0)G^S.
\end{align*}
Since the pressure perturbation can be expressed by
\begin{equation}\label{eq:pdiffid}
p-\bar p
=R\left(\frac{\theta}{v}-\frac{\bar \theta}{\bar v}\right)
=\bar p \left(\frac{\chi}{\bar \theta}-\frac{\phi}{\bar v}\right)
+R\left(\frac{\chi}{v}-\frac{\chi}{\bar v}-\frac{\bar \theta \phi}{v \bar v}+\frac{\bar \theta \phi}{\bar{v}^2}\right),
\end{equation}
the \textit{a priori} assumption \eqref{suppose} implies
\begin{align*}
&\sigma^* p^*\int_{\RR}a \abs{(v^S)_\xi}
\left(\frac{(p- \bar p)^2}{2\pbar^2}+ \frac{\gamma\phi^2}{2\bar{v}^2}\right)d\xi+C(\e_1+\delta_0)G^S \\
&\leq \sigma^* p^*\int_{\RR}a \abs{(v^S)_\xi} \bigg(\frac{1}{2}\left(\frac{\chi}{\bar \theta}-\frac{\phi}{\bar v}\right)^2+ \frac{\gamma\phi^2}{2(v^*)^2}\bigg)d\xi
+C(\e_1+\delta_0)G^S.
\end{align*}
Thanks to \eqref{suppose}, we expand the right-hand side to obtain
\begin{align*}
&\sigma^* p^* \int_{\RR}a \abs{(v^S)_\xi}
\bigg(\frac{1}{2}\bigg(\frac{\chi}{\bar \theta}-\frac{\phi}{\bar v}\bigg)^2
+ \frac{\gamma \phi^2}{2(v^*)^2}\bigg)d\xi
+C(\e_1+\delta_0)G^S \\
&\leq \int_{\RR}a \abs{(v^S)_\xi} \left(\frac{(\gamma+1)p^*\sigma^*}{2(v^*)^2}\phi^2+\frac{R\sigma^*}{2v^*\theta^*}\chi^2-\frac{p^*\sigma^*}{v^*\theta^*}\phi \chi\right)d\xi+C(\e_1+\delta_0)G^S.
\end{align*}
So, by the definition of \(B_5\), we find that
\begin{equation}\label{eq:B5simplified}
B_5 \leq C(\e_1+\delta_0) G^S.
\end{equation}

\noindent\underline{\textbf{Step 5: Estimates on \(G_2\)}}
Next, similarly to the case of \(B_5\), we simplify \(G_2\) as follows:
\begin{align*}
G_2
&=\int_{\RR} a\bar{p} (u^R)_\xi
\left(\Phi\Big(\frac{p}{\bar p}\Big)
-\frac{(p- \bar p)^2}{2\pbar^2}
+\frac{(p- \bar p)^2}{2\pbar^2}
+\gamma \Phi(\frac{v}{\bar v})
-\frac{\gamma \phi^2}{2\bar{v}^2}
+\frac{\gamma \phi^2}{2\bar{v}^2}\right)d\xi \\
&\geq \int_{\RR}a\bar p (u^R)_\xi
\left(\frac{(p- \bar p)^2}{2\pbar^2}
+\frac{\gamma\phi^2}{2\bar{v}^2}\right)d\xi
-C\e_1G^R.
\end{align*}
Using \eqref{eq:pdiffid} and the arithmetic mean--geometric mean inequality \(2\phi \chi \leq \frac{2\bar \theta}{\bar v}\phi^2+\frac{\bar v}{2 \bar \theta}\chi^2\), we obtain
\begin{align*}
&\int_{\RR}a\bar p(u^R)_\xi
\left(\frac{(p- \bar p)^2}{2\pbar^2}
+\frac{\gamma\phi^2}{2\bar{v}^2}\right)d\xi
+C \e_1 G^R\\
&\qquad
\geq
\int_{\RR}a \frac{R(u^R)_\xi}{2\bar{v}^2}
\bigg(\frac{\gamma \bar \theta}{\bar v}\phi^2
+\bigg(\sqrt{\frac{\bar \theta}{\bar v}}\phi-\sqrt{\frac{\bar v}{\bar\theta}}\chi\bigg)^2\bigg)d\xi
-C \e_1 G^R\\
&\qquad
\geq \int_{\RR}a\frac{R(u^R)_\xi}{2 \bar{v}^2}\left(\frac{(\gamma-1)\bar \theta}{\bar v}\phi^2+\frac{\bar v}{2 \bar \theta}\chi^2\right)d\xi-C\e_1G^R
\geq \frac{3}{2}c^*G^R-C\e_1G^R,
\end{align*}
where the constant \(c^*\) is given by
\begin{align}\label{eq:c^*def}
c^*:=\min \left\{
\frac{1}{2}\cdot
\frac{R\sqrt{R \gamma \theta^*}}{2(v^*)^3}\cdot
\frac{(\gamma-1)\theta^*}{v^*},
\frac{1}{2}\cdot\frac{R\sqrt{R \gamma \theta^*}}{2(v^*)^3}\cdot
\frac{v^*}{2 \theta^*}\right\}.
\end{align}
Here, \Cref{lem:uR} was used to determine the constant \(c^*\), alongside the smallness of \(\d_0\) and \(\e_1\).
Thus, using the smallness of \(\e_1\), we finally obtain
\begin{equation}\label{eq:G2simplified}
-G_2 \leq -c^* G^R.
\end{equation}

\noindent\underline{\textbf{Step 6: Estimates on \(S_i\), \(i=1,2,3\)}}
Let us start with recalling \eqref{eq:Qidefs}.

\vspace{2mm}
\bci{S_1} Using \eqref{Q13CH21}, we get
\begin{equation}\begin{aligned}\label{eq:S1simplified}
\abs{S_1}
\le C\int_\RR \big(\abs{Q_1^I}+\abs{Q_1^C}\big) \abs{\psi} d\x
&\le C\left(\norm{Q_1^I}_{L^2}+\norm{Q_1^C}_{L^2}\right)\norm{\psi}_{L^2}\\
&\le C\left(\norm{Q_1^I}_{L^2}+\delta_C(1+t)^{-\frac{5}{4}}\right)\norm{\psi}_{L^2}. 
\end{aligned}\end{equation}

\bci{S_2} Applying H\"older's inequality and the Gagliardo--Nirenberg inequality, we obtain
\begin{equation}\begin{aligned}\label{eq:S2simplified}
\abs{S_2}
&\le C\int_\RR \abs{Q_2} \abs{(\phi,\chi)} d\x
= C\int_\RR \abs{Q_2^I+Q_2^R} \abs{(\phi,\chi)} d\x\\
&\le C\norm{Q_2^I}_{L^2}\norm{(\phi,\chi)}_{L^2}
+C \norm{Q_2^R}_{L^1}\norm{(\phi,\chi)}_{L^2}^{\frac12} 
\norm{(\phi_\xi,\chi_\xi)}_{L^2}^{\frac12}\\
&\le C\norm{Q_2^I}_{L^2}\norm{(\phi,\chi)}_{L^2}
+C (\d_R)^{\frac{1}{2}} \norm{(\phi_\xi, \chi_\xi)}_{L^2}^2
+C (\d_R)^{-\frac{1}{6}} \norm{Q_2^R}_{L^1}^{\frac43}\norm{(\phi,\chi)}_{L^2}^{\frac23}
\end{aligned}\end{equation}
Here, Young's inequality was used as well.

\bci{S_3} Lastly, for \(S_3\), another use of Young's inequality yields that
\begin{equation}\label{eq:S3simplified}
\abs{S_3}
\le \frac{c_*}{8}\int_{\RR} r^2 d\xi
+C\left(\norm{Q_3^C}_{L^2}^2+\norm{Q_3^R}_{L^2}^2+\norm{Q_3^I}_{L^2}^2\right),
\end{equation}
where we recall the definition of \(c_*\) from \eqref{eq:c_*def}.

Combining \eqref{eq:Y4-Y6simplified}, \eqref{eq:B3simplified}, \eqref{eq:B4simplified}, \eqref{eq:B5simplified}, \eqref{eq:G2simplified}, \eqref{eq:S1simplified}, \eqref{eq:S2simplified}, and \eqref{eq:S3simplified}, together with the smallness of \(\d_0\) and \(\e_1\), and using the fact that 
\begin{align*}
-\frac{1}{4}D_1-\frac{1}{4}D_2 \leq -c_*\int_{\RR}r^2d\xi-c\int_{\RR}r_\xi^2d\xi,  
\end{align*}
for some $c>0$, we obtain \Cref{lem:L2estpart2}.
This completes the proof.
\end{proof}

\subsection{Proof of \Cref{lem:l2est}}
We now prove Lemma \ref{lem:l2est} using the estimates shown in this section, particularly those in Lemmas \ref{lem:L2estpart1} and \ref{lem:L2estpart2}.
To this end, we introduce the following lemma.

\begin{lemma}\label{lem:W}
The following weighted energy estimate holds:
\begin{multline*}
\int_0^t \int_\RR W^2(\phi^2+\psi^2+\chi^2)d\x d\t
\le C\sup_{\t\in[0,t]} \norm{\big(\phi,\psi,\chi\big)(\t)}_{L^2}^2
+C\d_S \int_0^t |\dot{X}|^2 d\t\\
+C\int_0^t \left[G^R+G^S+\norm{r}_{L^2}^2
+\norm{\big(\phi_\x,\psi_\x,\chi_\x\big)(\t)}_{L^2}^2\right] d\t
+C \d_0^{\frac{1}{3}}.
\end{multline*}
\end{lemma}
\begin{proof}
The proof follows the arguments in \cite[Lemma 5]{HLM-ARMA10} and \cite[Lemma 3.2]{WX-JDE11}, with the terms arising from the shift handled as in \cite[Appendix D]{KVW-ARMA25}.
Thus, we omit the details.
\end{proof}

We recall \eqref{eq:RHSdecomp2}.
Combining Lemmas \ref{lem:L2estpart1}--\ref{lem:L2estpart2} and taking \(\d_0\) and \(\e_1\) small enough, we obtain
\begin{align*}
&\frac{d}{dt} \int_\RR a \thbar \et(U|\Ubar) d\x
\leq -c\Big(G^S+G^R+\norm{r}_{L^2}^2+\norm{r_\x}_{L^2}^2+\delta_S|\dot{X}|^2\Big)\\
&+C\bigg((\e_1+\sqrt{\d_0})\norm{(\phi_\xi, \psi_\xi,\chi_\xi)}_{L^2}^2 
+\delta_C\int_{\RR}W(t,\xi)^2|(\phi,\psi,\chi)|^2d\xi
+\delta_C(1+t)^{-\frac{5}{4}}\norm{\psi}_{L^2}\\
&+(\delta_S)^{\frac43}\Big((\delta_R)^{\frac43}+(\delta_C)^{\frac43}\Big)e^{-c\d_S t}
+(\d_R)^{-\frac{1}{6}}\norm{Q_2^R}_{L^1}^{\frac43}\norm{(\phi,\chi)}_{L^2}^{\frac23}+\norm{Q_2^I}_{L^2}\norm{(\phi,\chi)}_{L^2} \\
&+\norm{Q_1^I}_{L^2}\norm{\psi}_{L^2}
+\norm{\big(Q_3^R,Q_3^C,Q_3^I\big)}_{L^2}^2\bigg).
\end{align*}
Then, integrating the above inequality over \([0,t]\) for any \(t\le T\), using Lemmas \ref{lem:interactionests}, \ref{lem:Q} and \ref{lem:W}, and taking \(\d_0\) and \(\e_1\) sufficiently small, we obtain
\begin{align*}
&\sup_{\tau \in [0,t]}\int_{\RR}\eta(U|\bar U)d\xi
+\delta_S \int_0^t|\dot{X}|^2d\tau
+\int_0^t\left(G^S+G^R+\norm{r}_{L^2}^2+\norm{r_\x}_{L^2}^2\right)d\tau \\
&\leq C\int_{\RR}\eta(U_0|\bar{U}(0,\xi))d\xi
+C(\e_1 + \sqrt{\delta_0})\int_0^t \norm{(\psi_\xi, \phi_\xi,\chi_\xi)}_{L^2}^2d\tau
+C\delta_0 \\
&\qquad
+\delta_C\norm{\psi}_{L^\infty(0,t;L^2(\RR))}
+C\delta_0^{\frac16}\norm{(\phi,\psi,\chi)}_{L^\infty(0,t;L^2(\RR))}^{\frac23}.
\end{align*}
Finally, using Young's inequality and the locally quadratic property of the relative entropy, i.e., 
\begin{align*}
\norm{U-\bar U}_{L^2}^2
\sim \int_{\RR}\eta(U|\bar U)d\xi, \qquad \forall t \in [0,T],
\end{align*}
we obtain the desired estimate stated in \Cref{lem:l2est}.
This completes the proof. \qed

\section{Higher Order Estimates}\label{sec:h2}
\setcounter{equation}{0}
We now turn to the higher-order energy estimates.
The zeroth-order estimate established in the previous section provides control of the perturbation at the \(L^2\) level.
To obtain the desired \(H^2\) estimate, we derive estimates for the first- and second-order spatial derivatives of the perturbation.
We first establish the derivative estimates using suitable energy multipliers and then exploit the radiative effect to obtain dissipative structures for the conserved quantities from the radiative heat flux and close the full \textit{a priori} estimate.
\par First of all, we recall the \textit{a priori} assumption \eqref{suppose}.
\begin{equation}\label{assYap}
\sup_{t\in[0,T]} \norm{(\phi,\psi,\chi)(t)}_{H^2} \le \e_1.
\end{equation}
It simply follows that both \(v\) and \(\th\) are uniformly bounded above and away from zero: 
\begin{equation}\label{vt-1}
0< c < v < C, \qquad
0< c < \th < C.
\end{equation}
Moreover, using Lemmas \ref{lem:uR}, \ref{lem:qR}, \ref{lem:VC} and \ref{lem:VS}, we obtain 
\begin{equation}\label{each-sm}
\sup_{t\in[0,T]} \Big(\norm{(\vbar_\x,\ubar_\x,\thbar_\x)(t)}_{H^1} + \norm{\qbar(t)}_{H^3}\Big)
\le C(\d_R + \d_C + \d_S) \le C\d_0.
\end{equation}
This together with the assumption \eqref{assYap} shows that 
\[
\sup_{t\in[0,T]} \norm{(v_\x, u_\x, \th_\x)(t)}_{H^1} \le C(\e_1+\d_0).
\]
Then we also have 
\begin{equation}\label{vxtx-1}
\sup_{t\in[0,T]} \Big(\norm{(v_\x,u_\x,\th_\x)(t)}_{L^\infty}
+\norm{(v_{\x\x},u_{\x\x},\th_{\x\x})(t)}_{L^2}\Big) \le C(\e_1+\d_0).
\end{equation}
We note that \eqref{vt-1} and \eqref{vxtx-1} will be used repeatedly in the analysis that follows.

\subsection{\textit{A Priori} Bounds for Radiative Heat Flux}
The \textit{a priori} assumption \eqref{suppose} is given in terms of the variables \((v,u,\th)\).
To carry out the higher-order estimates, however, additional control of the radiative heat flux \(q\) is required.
We therefore derive suitable estimates for \(q\) from the \textit{a priori} assumption together with the structure of the system.
We emphasize that, unlike the other variables \(v\), \(u\) and \(\th\), the radiative heat flux \(q\) has the same far-field state at both ends, namely zero.
Consequently, while the integrability of \(v\), \(u\) and \(\th\) must be considered relative to the corresponding composite wave, no such subtraction is needed for \(q\).
In particular, \(q\) itself can be treated directly in the relevant estimates.

\vspace{2mm}
We perform the standard energy estimate.
We recall the fourth equation in \eqref{main-x}: 
\begin{equation}\label{q-eqn}
-\left(\frac{q_\x}{v}\right)_\x + \a v q + \b (\th^4)_\x =0.
\end{equation}
Multiplying both sides by \(q\), we find that
\begin{align*}
\int_\RR \left[\frac{(q_\x)^2}{v} + \a v q^2\right] d\x
&=\int_\RR \left[- q \left(\frac{q_\x}{v}\right)_\x + \a v q^2\right] d\x
=-4\b\int_\RR \th^3 \th_\x q d\x
\le C\norm{\th_\x}_{L^2}\norm{q}_{L^2}.
\end{align*}
It follows from \eqref{vt-1} that
\[
\int_\RR \left[\frac{(q_\x)^2}{v} + \a v q^2\right] d\x
\ge c \norm{q}_{H^1}^2,
\]
which immediately yields
\begin{equation}\label{qh1}
\norm{q}_{H^1} \le C \norm{\th_\x}_{L^2}.
\end{equation}

To derive an \(H^2\) estimate, we rewrite the elliptic equation \eqref{q-eqn} in the following form:
\[
\frac{v_\x q_\x}{v^2} + \a v q + \b (\th^4)_\x = \frac{q_{\x\x}}{v}.
\]
Since \(v\) is uniformly bounded above and away from zero, this is equivalent to 
\begin{equation}\label{qxx-eqn}
q_{\x\x}=\frac{v_\x}{v} q_\x + \a v^2 q + 4\b v \th^3 \th_\x.
\end{equation}
Combining this with \eqref{vxtx-1} and \eqref{qh1}, we obtain
\[
\norm{q_{\x\x}}_{L^2}
\le C\left(\norm{q_\x}_{L^2} + \norm{q}_{L^2} + \norm{\th_\x}_{L^2}\right)
\le C \norm{\th_\x}_{L^2}.
\]
This, together with \eqref{qh1}, further implies that
\begin{equation}\label{qh2}
\norm{q}_{H^2} \le C \norm{\th_\x}_{L^2}, \qquad
\norm{q_\x}_{L^\infty} \le C \norm{\th_\x}_{L^2}.
\end{equation}

Differentiating \eqref{qxx-eqn} with respect to \(\x\), we see that 
\[
q_{\x\x\x}
=-\frac{(v_\x)^2}{v^2} q_\x
+\frac{v_{\x\x}}{v} q_\x
+\frac{v_\x}{v} q_{\x\x}
+ \a v \left(2 v_\x q + v q_\x\right)
+ 4\b \left(v_\x \th^3 \th_\x
+3v \th^2 (\th_\x)^2
+v \th^3 \th_{\x\x}
\right).
\]
Thus, using \eqref{vxtx-1}, we deduce that 
\[
\norm{q_{\x\x\x}}_{L^2}
\le C\Big(
\norm{q_\x}_{L^\infty}\norm{v_{\x\x}}_{L^2}
+\norm{q}_{L^2}
+\norm{q_\x}_{L^2}
+\norm{q_{\x\x}}_{L^2}
+\norm{\th_\x}_{L^2}
+\norm{\th_{\x\x}}_{L^2}
\Big).
\]
It then follows from \eqref{qh2} that
\[
\norm{q_{\x\x\x}}_{L^2}
\le C\Big(\big(1+\norm{v_{\x\x}}_{L^2}\big)\norm{\th_\x}_{L^2}
+\norm{\th_{\x\x}}_{L^2}\Big).
\]
Therefore, we conclude that
\begin{equation}\label{qh3}
\norm{q}_{H^3} 
\le C(\e_1 + \d_0).
\end{equation}

Finally, the bound for \(\qbar\) in \eqref{each-sm} implies that 
\begin{equation}\label{rh3}
\norm{r}_{H^3}
\le \norm{q}_{H^3} + \norm{\qbar}_{H^3}
\le C(\e_1 + \d_0).
\end{equation}

\subsection{Preliminary Estimates}
In this subsection, we state two lemmas under the \textit{a priori} assumption.
The first lemma provides bounds for terms arising in the first-order energy estimate, while the second deals with those appearing in the second-order energy estimate. These results will be used in the subsequent energy analysis.

\begin{lemma}\label{lem:first-o}
Under the \textit{a priori} assumption \eqref{suppose}, the following estimates hold:
\begin{multline}\label{first-o-1}
\int_\RR \left[\bigg(\frac{R\th}{v^2}-\frac{R\thbar}{\vbar^2}\bigg)\vbar_\x
-R\bigg(\frac{1}{v}-\frac{1}{\vbar}\bigg)\thbar_\x\right]^2 d\x
+\int_\RR \big((p-\pbar) \ubar_\x\big)^2 d\x\\
\le C\d_0(G^R+G^S)
+C\d_0\d_C\int_\RR W^2 (\phi^2+\chi^2) d\x.
\end{multline}
\begin{multline}\label{first-o-2}
\int_\RR \left(\left[\left(\frac{R\th}{v^2}-\frac{R\thbar}{\vbar^2}\right)\vbar_\x
-R\left(\frac{1}{v}-\frac{1}{\vbar}\right)\thbar_\x\right]_\x\right)^2 d\x
+\int_\RR \big(\big((p-\pbar)\ubar_\x\big)_\x\big)^2 d\x\\
\le C\d_0^2 \norm{\big(\phi_\x,\psi_\x,\chi_\x\big)}_{L^2}^2
+C\d_0 \big(G^R+G^S \big)
+C\d_0\d_C\int_\RR W^2 (\phi^2+\chi^2) d\x.
\end{multline}
\begin{multline}\label{first-o-3}
\norm{Q_3}_{L^2}^2
+\norm{\a \qbar \phi}_{L^2}^2
+\norm{4\b(\th^3-\thbar^3)\thbar_\x}_{L^2}^2
+\norm{\left(\frac{\phi\qbar_\x}{v\vbar}\right)_\x}_{L^2}^2 \\
\le C\norm{Q_3}_{L^2}^2
+C\d_0^2 \norm{\big(\phi_\x,\psi_\x,\chi_\x\big)}_{L^2}^2
+C\d_0 \big(G^R+G^S \big)
+C\d_0\d_C\int_\RR W^2 (\phi^2+\chi^2) d\x.
\end{multline}
\end{lemma}

\begin{proof}
In the proof, the derivatives of the composite wave need to be analyzed.
In particular, since they localize the bad terms through the derivatives of the smooth rarefaction wave, the viscous contact wave, and the radiative shock, these localization effects require careful analysis.
To this end, by the definition of the composite wave in \eqref{Ubar}, we have 
\[
\vbar_\x = (v^R)_\x + (v^C_\x) + (v^S)_\x, \quad
\ubar_\x = (u^R)_\x + (u^C_\x) + (u^S)_\x, \quad
\thbar_\x = (\th^R)_\x + (\th^C_\x) + (\th^S)_\x,
\]
and applying Lemmas \ref{lem:uR}--\ref{lem:VS}, we obtain 
\[
\abs{\vbar_\x}^2 + \abs{\ubar_\x}^2 + \abs{\thbar_\x}^2
\le C\d_0 \left(|(v^R)_\x|+|(v^S)_\x|
+\frac{\d_C}{1+t}e^{-\frac{c_1|\x+\s t|^2}{(1+t)}}\right)
\le C\d_0^2.
\]
Alternatively, the right-hand side can be rewritten in terms of \(W\) as follows:
\begin{equation}\label{vbx2}
\abs{\vbar_\x}^2 + \abs{\ubar_\x}^2 + \abs{\thbar_\x}^2
\le C\d_0 \left(|(v^R)_\x|+|(v^S)_\x|
+\d_C W^2\right)
\le C\d_0^2.
\end{equation}
The second-order derivatives can be estimated similarly using Lemmas \ref{lem:uR}--\ref{lem:VS}.
This yields
\begin{equation}\label{vbxx2}
\abs{\vbar_{\x\x}}^2 + \abs{\ubar_{\x\x}}^2 + \abs{\thbar_{\x\x}}^2
\le C\d_0 \left(|(v^R)_\x|+|(v^S)_\x|
+\d_C W^2\right)
\le C\d_0^2.
\end{equation}
Moreover, the definition of the radiative heat flux component in \eqref{qbar} implies that 
\[
\qbar = q^R + q^C + q^S, \quad
\qbar_\x = (q^R)_\x + (q^C)_\x + (q^S)_\x, \quad
\qbar_{\x\x} = (q^R)_{\x\x} + (q^C)_{\x\x} + (q^S)_{\x\x}.
\]
Combining this with the estimates in Lemmas \ref{lem:uR}--\ref{lem:VS}, we deduce that
\begin{equation}\label{qx2}
\abs{\qbar}^2
+\abs{\qbar_\x}^2
+\abs{\qbar_{\x\x}}^2
\le C\d_0 \left(|(v^R)_\x|+|(v^S)_\x| +\d_C W^2\right)
\le C\d_0^2.
\end{equation}

\bpf{first-o-1}
We begin by observing that
\[
\bigg(\frac{R\th}{v^2}-\frac{R\thbar}{\vbar^2}\bigg)^2
+ R^2 \bigg(\frac{1}{v}-\frac{1}{\vbar}\bigg)^2 
+ (p-\pbar)^2
\le C(\phi^2+\chi^2).
\]
The left-hand side of \eqref{first-o-1} can then be bounded by 
\[
C\int_\RR \left(\abs{\vbar_\x}^2 + \abs{\ubar_\x}^2 + \abs{\thbar_\x}^2\right) (\phi^2+\chi^2) d\x.
\]
In view of \eqref{vbx2}, this estimate further implies that 
\begin{multline*}
\int_\RR \left[\bigg(\frac{R\th}{v^2}-\frac{R\thbar}{\vbar^2}\bigg)\vbar_\x
-R\bigg(\frac{1}{v}-\frac{1}{\vbar}\bigg)\thbar_\x\right]^2 d\x
+\int_\RR \big((p-\pbar) \ubar_\x\big)^2 d\x\\
\le C\d_0(G^R+G^S)
+C\d_0\d_C\int_\RR W^2 (\phi^2+\chi^2) d\x,
\end{multline*}
as desired.

\bpf{first-o-2}
First, we note that 
\[
\rd_\x \left[\left(\frac{R\th}{v^2}-\frac{R\thbar}{\vbar^2}\right)\vbar_\x\right]
=\left(\frac{R\th}{v^2}-\frac{R\thbar}{\vbar^2}\right)_\x \vbar_\x
+\left(\frac{R\th}{v^2}-\frac{R\thbar}{\vbar^2}\right)\vbar_{\x\x}
\]
and that 
\[
\left(\frac{R\th}{v^2}-\frac{R\thbar}{\vbar^2}\right)_\x
=\left(\frac{R\chi}{v^2} - \frac{R(v+\vbar)\thbar}{v^2\vbar^2}\phi\right)_\x
=\frac{R\chi_\x}{v^2} + \left(\frac{R}{v^2}\right)_\x \chi
-\frac{R(v+\vbar)\thbar}{v^2\vbar^2}\phi_\x
-\left(\frac{R(v+\vbar)\thbar}{v^2\vbar^2}\right)_\x \phi.
\]
Then, by \eqref{vt-1}, \eqref{vxtx-1}, and Lemmas \ref{lem:uR}--\ref{lem:VS}, we deduce that
\begin{equation}\label{firsto2-1}
\abs{\left(\frac{R\th}{v^2}-\frac{R\thbar}{\vbar^2}\right)_\x}
\le C \big(\abs{\phi_\x}+\abs{\chi_\x}\big)
+C(\e_1+\d_0) \big(\abs{\phi}+\abs{\chi}\big).
\end{equation}
Similarly, we obtain 
\[
\rd_\x \left[-R\left(\frac{1}{v}-\frac{1}{\vbar}\right)\thbar_\x\right]
=-R\left(\frac{1}{v}-\frac{1}{\vbar}\right)_\x\thbar_\x
-R\left(\frac{1}{v}-\frac{1}{\vbar}\right)\thbar_{\x\x}
\]
and
\begin{equation}\label{firsto2-2}
\abs{R\left(\frac{1}{v}-\frac{1}{\vbar}\right)_\x}
\le C \abs{\phi_\x}
+C(\e_1+\d_0) \abs{\phi}.
\end{equation}
An analogous calculation shows that 
\[
\rd_\x \left[(p-\pbar)\ubar_\x\right]
=\rd_\x \left[\left(\frac{R\th}{v}-\frac{R\thbar}{\vbar}\right)\ubar_\x\right]
=\left(\frac{R\th}{v}-\frac{R\thbar}{\vbar}\right)_\x \ubar_\x
+\left(\frac{R\th}{v}-\frac{R\thbar}{\vbar}\right)\ubar_{\x\x}
\]
and that
\begin{equation}\label{firsto2-3}
\abs{\left(\frac{R\th}{v}-\frac{R\thbar}{\vbar}\right)_\x}
\le C \big(\abs{\phi_\x}+\abs{\chi_\x}\big)
+C(\e_1+\d_0) \big(\abs{\phi}+\abs{\chi}\big).
\end{equation}
On the other hand, it is straightforward to show that 
\begin{equation}\label{firsto2-n}
\left(\frac{R\th}{v^2}-\frac{R\thbar}{\vbar^2}\right)^2
+R^2\left(\frac{1}{v}-\frac{1}{\vbar}\right)^2
+\left(\frac{R\th}{v}-\frac{R\thbar}{\vbar}\right)^2
\le C(\phi^2+\chi^2).
\end{equation}
The desired estimate now follows from \eqref{vbx2}--\eqref{vbxx2} and \eqref{firsto2-1}--\eqref{firsto2-n}:
\begin{align*}
&\int_\RR \left(\left[\left(\frac{R\th}{v^2}-\frac{R\thbar}{\vbar^2}\right)\vbar_\x
-R\left(\frac{1}{v}-\frac{1}{\vbar}\right)\thbar_\x\right]_\x\right)^2 d\x
+\int_\RR \big(\big((p-\pbar)\ubar_\x\big)_\x\big)^2 d\x\\
&\le C\int_\RR  \left(\abs{\vbar_\x}^2 + \abs{\ubar_\x}^2 + \abs{\thbar_\x}^2\right) (\abs{\phi_\x}^2+\abs{\chi_\x}^2) d\x\\
&\qquad+C(\e_1+\d_0)\int_\RR  \left(\abs{\vbar_\x}^2 + \abs{\ubar_\x}^2 + \abs{\thbar_\x}^2\right) (\phi^2+\chi^2) d\x\\
&\qquad
+C\int_\RR  \left(\abs{\vbar_{\x\x}}^2 + \abs{\ubar_{\x\x}}^2 + \abs{\thbar_{\x\x}}^2\right) (\phi^2+\chi^2) d\x\\
&\le 
C\d_0^2 \norm{\big(\phi_\x,\psi_\x,\chi_\x\big)}_{L^2}^2
+C\d_0 \big(G^R+G^S \big)
+C\d_0\d_C\int_\RR W^2 (\phi^2+\chi^2) d\x.
\end{align*}

\bpf{first-o-3}
Using \eqref{qx2}, the following two terms can be readily controlled as follows: 
\begin{align*}
\norm{\a \qbar \phi}_{L^2}^2
&=\a^2 \int \qbar^2 \phi^2 d\x
\le C\d_0(G^R+G^S)
+C\d_0 \d_C \int_\RR W^2 \phi^2 d\x,\\
\norm{4\b(\th^3-\thbar^3)\thbar_\x}_{L^2}^2
&\le C\int_\RR (\thbar_\x)^2 \chi^2 d\x
\le C\d_0(G^R+G^S)
+C\d_0 \d_C \int_\RR W^2 \chi^2 d\x.
\end{align*}
It remains to estimate the last term on the left-hand side of \eqref{first-o-3}.
We use \eqref{firsto2-2} to obtain
\begin{align*}
\norm{\left(\frac{\phi\qbar_\x}{v\vbar}\right)_\x}_{L^2}^2
&=\int_\RR \left[\left(\frac{\phi\qbar_\x}{v\vbar}\right)_\x\right]^2 d\x\\
&\le C\int_\RR (\qbar_\x)^2 |\phi_\x|^2 d\x
+C\int_\RR (\qbar_{\x\x})^2 \phi^2 d\x
+C(\e_1+\d_0) \int_\RR (\qbar_\x)^2 \phi^2 d\x.
\end{align*}
Then, applying \eqref{qx2}, we deduce that 
\[
\norm{\left(\frac{\phi\qbar_\x}{v\vbar}\right)_\x}_{L^2}^2
\le C\d_0^2 \norm{\phi_\x}_{L^2}^2
+C\d_0(G^R+G^S)
+C\d_0\d_C \int_\RR W^2 \phi^2 d\x.
\]
Therefore, collecting the above estimates, we conclude that
\begin{multline*}
\norm{Q_3}_{L^2}^2
+\norm{\a \qbar \phi}_{L^2}^2
+\norm{4\b(\th^3-\thbar^3)\thbar_\x}_{L^2}^2
+\norm{\left(\frac{\phi\qbar_\x}{v\vbar}\right)_\x}_{L^2}^2 \\
\le C\norm{Q_3}_{L^2}^2
+C\d_0^2 \norm{\big(\phi_\x,\psi_\x,\chi_\x\big)}_{L^2}^2
+C\d_0 \big(G^R+G^S \big)
+C\d_0\d_C\int_\RR W^2 (\phi^2+\chi^2) d\x.
\end{multline*}
This completes the proof of Lemma \ref{lem:first-o}.
\end{proof}

We now present another lemma, which provides useful bounds for the second-order energy estimate.

\begin{lemma}\label{lem:sec-o}
Under the \textit{a priori} assumption \eqref{suppose}, the following estimates hold:
\begin{multline}\label{sec-o-1}
\int_\RR \left(\left[\left(\frac{R\th}{v^2}-\frac{R\thbar}{\vbar^2}\right)\vbar_\x
-R\left(\frac{1}{v}-\frac{1}{\vbar}\right)\thbar_\x\right]_{\x\x}\right)^2 d\x
+\int_\RR \big(\big((p-\pbar)\ubar_\x\big)_{\x\x}\big)^2 d\x\\
\le C\d_0^2 \norm{\big(\phi_\x,\psi_\x,\chi_\x\big)}_{H^1}^2
+C\d_0 \big(G^R+G^S \big)
+C\d_0\d_C\int_\RR W^2 (\phi^2+\chi^2) d\x.
\end{multline}
\begin{multline}\label{sec-o-2}
\norm{Q_3}_{H^1}^2
+\norm{\a \qbar \phi}_{H^1}^2
+\norm{4\b(\th^3-\thbar^3)\thbar_\x}_{H^1}^2
+\norm{\left(\frac{\phi\qbar_\x}{v\vbar}\right)_\x}_{H^1}^2 \\
\le C\norm{Q_3}_{H^1}^2
+C\d_0^2 \norm{\big(\phi_\x,\psi_\x,\chi_\x\big)}_{H^1}^2
+C\d_0 \big(G^R+G^S \big)
+C\d_0\d_C\int_\RR W^2 (\phi^2+\chi^2) d\x.
\end{multline}
\end{lemma}

\begin{proof}
The present lemma involves higher-order derivatives of the composite wave than the preceding one.
Nevertheless, the key ingredients of the proof remain the localization effects of these derivatives and suitable expansions of the relevant terms.
Since Lemmas \ref{lem:uR}--\ref{lem:VS} provide the necessary higher-order derivative estimates for the smooth rarefaction wave, the viscous contact wave, and the radiative shock, we can proceed as in the previous lemma, taking into account the additional derivative on the left-hand sides.
Although the resulting calculations are more elaborate, they require no essentially new arguments, and we therefore omit the details.
\end{proof}

\(\bullet\) \textbf{Notation:} Hereafter, \(\n>0\) denotes a small parameter arising from Young's inequality, while \(C_\n\) denotes the corresponding constant, which may be large.
The value of \(\n\) may vary from line to line and will be chosen sufficiently small later.
Once fixed, both \(\n\) and \(C_\n\) are regarded as \(\Ocal(1)\) constants, independent of \(\e_1\) and \(\d_0\), and will be absorbed into the generic constants \(C\) and \(c\).

\subsection{First-Order Estimate}
In this subsection, we establish the first-order energy estimate.
Unlike the Navier--Stokes--Fourier system, for which the energy estimate can be closed at the \(H^1\)-level (see \cite{KVW-ARMA25}), our system requires a second-order energy estimate to achieve closure.
The difference is essentially due to the absence of viscosity.
A similar difficulty arises in the heat-conducting ideal gas without viscosity, where second-order energy estimates are also required to close the argument; see the recent work \cite{GLWX-2026}.

\vspace{2mm}
We briefly outline the strategy.
We first perform standard energy estimates with suitably chosen multipliers and then exploit the elliptic equation to obtain diffusion for the radiative variable.
Next, we use the radiative equation to derive diffusion for \(\th\).
The diffusion for \(v\) and \(u\) is then obtained through the auxiliary quantity \eqref{key-qty}.
The choice of multipliers and the introduction of the auxiliary quantity are motivated by the arguments in \cite{FRX-AIHPC19,FRX-DCDS21}; see also \cite{GLWX-2026}.
The second-order energy estimate will follow the same argument.

\vspace{2mm}
We now derive the first-order energy estimate.
We recall that the perturbation variables are defined by
\[
\phi = v-\vbar, \qquad
\psi = u-\ubar, \qquad
\chi = \th-\thbar, \qquad
r = q-\qbar.
\]
Using \eqref{main-x} and \eqref{bar-eq}, we obtain the following equations for the perturbation variables:
\begin{equation}\label{pert-eq}
\begin{cases}
\phi_t -\s \phi_\x - \dot{X}(v^S)_\x -\psi_\x =0,\\
\psi_t -\s \psi_\x - \dot{X}(u^S)_\x + (p-\pbar)_\x = -Q_1,\\
\frac{R}{\g-1}(\chi_t -\s \chi_\x -\dot{X}(\th^S)_\x) + (pu_\x - \pbar\ubar_\x) + r_\x = -Q_2.
\end{cases}
\end{equation}
Turning to the elliptic equation, we first observe that 
\[
\frac{q_\x}{v}-\frac{\qbar_\x}{\vbar}
=\frac{1}{v}(q-\qbar)_\x
+\left(\frac{1}{v}-\frac{1}{\vbar}\right)\qbar_\x
=\frac{r_\x}{v}-\frac{\phi}{v\vbar}\qbar_\x,
\]
and that
\[
vq -\vbar\qbar
=v(q-\qbar)
+\qbar(v-\vbar)
=v r + \qbar\phi, \qquad
\b(\th^4-\thbar^4)_\x
=4\b(\th^3\chi_\x
+\thbar_\x(\th^3-\thbar^3)).
\]
These identities then lead to the following perturbation equation for \(r\):
\begin{equation}\label{pert-qeq}
-\left(\frac{r_\x}{v}\right)_\x + \a v r + 4\b \th^3 \chi_\x
=-Q_3
-\left(\frac{\phi\qbar_\x}{v\vbar}\right)_\x
-\a \qbar \phi
-4\b(\th^3-\thbar^3)\thbar_\x.
\end{equation}

Differentiating each equation in \eqref{pert-eq} with respect to \(\x\), multiplying the resulting equations by \(\frac{R\th}{v^2}\phi_\x\), \(\psi_\x\), and \(\frac{\chi_\x}{\th}\), respectively, and integrating over \(\x\), we obtain
\begin{align*}
&\frac{d}{dt}\int_\RR \left[
\frac{R\th}{2v^2}(\phi_\x)^2
+ \frac{1}{2}(\psi_\x)^2
+ \frac{R}{2(\g-1)\th}(\chi_\x)^2 \right] d\x
+\int_\RR \frac{\chi_\x}{\th} r_{\x\x} d\x\\
&=\dot{X}\int_\RR \left[\frac{R\th}{v^2}\phi_\x (v^S)_{\x\x}
+\psi_\x(u^S)_{\x\x}
+\frac{R \chi_\x}{(\g-1)\th}(\th^S)_{\x\x}\right] d\x
-\int_\RR \left[\psi_\x (Q_1)_\x + \frac{\chi_\x}{\th}(Q_2)_\x\right] d\x\\
&-\int_\RR \frac{p_\x}{\th}\psi_\x \chi_\x d\x
-\int_\RR \frac{\chi_\x}{\th} ((p-\pbar)\ubar_\x)_\x d\x
+\int_\RR \psi_\x \left[\left(\frac{R\th}{v^2}-\frac{R\thbar}{\vbar^2}\right)\vbar_\x
-R\left(\frac{1}{v}-\frac{1}{\vbar}\right)\thbar_\x\right]_\x d\x\\
&+\int_\RR \left[-\left(\frac{(\g+1)R\th}{2v^3}u_\x + \frac{\g-1}{2v^2}q_\x\right)(\phi_\x)^2
+ \frac{pu_\x+q_\x}{2\th^2}(\chi_\x)^2 \right] d\x.
\end{align*}
We next derive a diffusion term for the perturbation variable \(r\).
Multiplying \eqref{pert-qeq} by \(-\frac{r_{\x\x}}{4\b\th^4}\) and integrating over \(\RR\), we find that 
\begin{multline*}
\int_\RR \frac{(r_{\x\x})^2}{4\b v \th^4} d\x
+\int_\RR \frac{\a v (r_\x)^2}{4\b \th^4} d\x
-\int_\RR \frac{\chi_\x}{\th} r_{\x\x} d\x
=\int_\RR \frac{v_\x}{4\b v^2 \th^4}r_\x r_{\x\x} d\x
-\int_\RR \a\left(\frac{v}{4\b\th^4}\right)_\x r r_\x d\x\\
+\int_\RR \frac{r_{\x\x}}{4\b\th^4}\left[Q_3 + \a\qbar \phi + 4\b(\th^3-\thbar^3)\thbar_\x + \left(\frac{\phi\qbar_\x}{v\vbar}\right)_\x\right]d\x.
\end{multline*}
Upon adding the two identities above, one term on the left-hand side cancels.
This yields
\begin{equation} \label{h1begin}
\begin{aligned}
&\frac{d}{dt}\int_\RR \left[
\frac{R\th}{2v^2}(\phi_\x)^2
+ \frac{1}{2}(\psi_\x)^2
+ \frac{R}{2(\g-1)\th}(\chi_\x)^2 \right] d\x
+\int_\RR \frac{(r_{\x\x})^2}{4\b v \th^4} d\x
+\int_\RR \frac{\a v (r_\x)^2}{4\b \th^4} d\x\\
&=\dot{X}\int_\RR \left[\frac{R\th}{v^2}\phi_\x (v^S)_{\x\x}
+\psi_\x(u^S)_{\x\x}
+\frac{R \chi_\x}{(\g-1)\th}(\th^S)_{\x\x}\right] d\x\\
&\qquad\qquad
-\int_\RR \left[\psi_\x (Q_1)_\x + \frac{\chi_\x}{\th}(Q_2)_\x\right] d\x\\
&\qquad\qquad
-\int_\RR \frac{\chi_\x}{\th} ((p-\pbar)\ubar_\x)_\x d\x
+\int_\RR \psi_\x \left[\left(\frac{R\th}{v^2}-\frac{R\thbar}{\vbar^2}\right)\vbar_\x
-R\left(\frac{1}{v}-\frac{1}{\vbar}\right)\thbar_\x\right]_\x d\x\\
&\qquad\qquad
+\int_\RR \left[-\left(\frac{(\g+1)R\th}{2v^3}u_\x + \frac{\g-1}{2v^2}q_\x\right)(\phi_\x)^2
+ \frac{pu_\x+q_\x}{2\th^2}(\chi_\x)^2
-\frac{p_\x}{\th}\psi_\x \chi_\x\right] d\x\\
&\qquad\qquad
+\int_\RR \frac{v_\x}{4\b v^2 \th^4}r_\x r_{\x\x} d\x
-\int_\RR \a\left(\frac{v}{4\b\th^4}\right)_\x r r_\x d\x\\
&\qquad\qquad
+\int_\RR \frac{r_{\x\x}}{4\b\th^4}\left[Q_3 + \a\qbar \phi + 4\b(\th^3-\thbar^3)\thbar_\x + \left(\frac{\phi\qbar_\x}{v\vbar}\right)_\x\right]d\x
\eqqcolon \sum_{j=1}^8 L_j.
\end{aligned}
\end{equation}
We then estimate the right-hand side on a term-by-term basis.\\

\bci{L_1} First, Lemma \ref{lem:VS} gives
\[
\norm{\big((v^S)_{\x\x},(u^S)_{\x\x},(\th^S)_{\x\x}\big)}_{L^2}^2
\le C \d_S^5.
\]
This in turn implies that
\begin{align*}
\abs{L_1}
&\le \nu \norm{\big(\phi_\x,\psi_\x,\chi_\x\big)}_{L^2}^2
+C_\n |\dot{X}|^2 \norm{\big((v^S)_{\x\x},(u^S)_{\x\x},(\th^S)_{\x\x}\big)}_{L^2}^2\\
&\le \nu \norm{\big(\phi_\x,\psi_\x,\chi_\x\big)}_{L^2}^2
+C_\n \d_S^5 |\dot{X}|^2.
\end{align*}

\bci{L_2} By Young's inequality, we have
\[
\abs{L_2} 
\le \n \norm{\big(\phi_\x,\psi_\x,\chi_\x\big)}_{L^2}^2 
+ C_\n \norm{\big(Q_1, Q_2\big)}_{H^1}^2.
\]

\bci{L_3+L_4} A straightforward application of Young's inequality gives 
\begin{multline*}
\abs{L_3}+\abs{L_4}
\le \n \norm{\big(\phi_\x,\psi_\x,\chi_\x\big)}_{L^2}^2\\
+C_\n \int_\RR \left(\left[\left(\frac{R\th}{v^2}-\frac{R\thbar}{\vbar^2}\right)\vbar_\x
-R\left(\frac{1}{v}-\frac{1}{\vbar}\right)\thbar_\x\right]_\x\right)^2 d\x
+C_\n \int_\RR \big(\big((p-\pbar)\ubar_\x\big)_\x\big)^2 d\x.
\end{multline*}
It then follows immediately from \eqref{first-o-2} that
\[
\abs{L_3}+\abs{L_4}
\le (\n+C_\n \d_0^2) \norm{\big(\phi_\x,\psi_\x,\chi_\x\big)}_{L^2}^2
+C_\n \d_0 \big(G^R+G^S \big)
+C_\n \d_0\d_C\int_\RR W^2 (\phi^2+\chi^2) d\x.
\]

\bci{L_5} The estimates \eqref{vxtx-1} and \eqref{qh3} provide the following bound for \(L_5\):
\[
\abs{L_5} 
\le C \norm{(v_\x,u_\x,\th_\x,q_\x)}_{L^\infty}
\norm{\big(\phi_\x,\psi_\x,\chi_\x\big)}_{L^2}^2
\le C(\e_1 + \d_0) \norm{\big(\phi_\x,\psi_\x,\chi_\x\big)}_{L^2}^2.
\]

\bci{L_6} By the smallness of \(v_\x\) established in \eqref{vxtx-1}, we have
\[
\abs{L_6}
\le C(\e_1 + \d_0) \norm{r_\x}_{H^1}^2.
\]

\bci{L_7} The smallness of \(v_\x\) and \(\th_\x\) in \eqref{vxtx-1} allows us to estimate
\[
\abs{L_7}
\le C(\e_1 + \d_0) \norm{r}_{H^1}^2.
\]

\bci{L_8} The estimate \eqref{first-o-3}, together with Young's inequality, gives the following bound for \(L_8\):
\begin{multline*}
\abs{L_8}
\le \n \norm{r_{\x\x}}_{L^2}^2
+C_\n \norm{Q_3}_{L^2}^2\\
+C_\n \d_0^2 \norm{\big(\phi_\x,\psi_\x,\chi_\x\big)}_{L^2}^2
+C_\n \d_0 \big(G^R+G^S \big)
+C_\n \d_0\d_C\int_\RR W^2 (\phi^2+\chi^2) d\x.
\end{multline*}
Applying the preceding bounds to the right-hand side of \eqref{h1begin}, we obtain
\begin{align*}
&\frac{d}{dt}\int_\RR \left[
\frac{R\th}{2v^2}(\phi_\x)^2
+ \frac{1}{2}(\psi_\x)^2
+ \frac{R}{2(\g-1)\th}(\chi_\x)^2 \right] d\x
+\int_\RR \frac{(r_{\x\x})^2}{4\b v \th^4} d\x
+\int_\RR \frac{\a v (r_\x)^2}{4\b \th^4} d\x\\
&\le \big(\n + C_\n \d_0^2 + C(\e_1+\d_0)\big) \norm{\big(\phi_\x,\psi_\x,\chi_\x\big)}_{L^2}^2
+C(\e_1 + \d_0) \left(\norm{r}_{H^1}^2 + \norm{r_\x}_{H^1}^2\right)
+\n \norm{r_{\x\x}}_{L^2}^2\\
&\qquad
+C_\n \d_S^5 |\dot{X}|^2
+C_\n \left(\norm{\big(Q_1, Q_2\big)}_{H^1}^2 +\norm{Q_3}_{L^2}^2 \right)\\
&\qquad
+C_\n\d_0 \big(G^R+G^S \big)
+C_\n \d_0\d_C \int_\RR W^2 (\phi^2+\chi^2) d\x.
\end{align*}
Since the coefficients \(\frac{1}{4\b v\th^4}\) and \(\frac{\a v}{4\b \th^4}\)  are uniformly bounded away from zero, we may choose \(\n\), \(\d_0\) and \(\e_1\) small enough to obtain
\begin{equation}\begin{aligned}\label{h1-r}
&\frac{d}{dt}\int_\RR \left[
\frac{R\th}{2v^2}(\phi_\x)^2
+ \frac{1}{2}(\psi_\x)^2
+ \frac{R}{2(\g-1)\th}(\chi_\x)^2 \right] d\x
+ c\norm{r_\x}_{H^1}^2\\
&\le \big(\n + C_\n (\e_1+\d_0)\big) \norm{\big(\phi_\x,\psi_\x,\chi_\x\big)}_{L^2}^2
+C(\e_1 + \d_0) \norm{r}_{H^1}^2
+C_\n \d_S^5 |\dot{X}|^2\\
&\hspace{2mm}
+C_\n \left(\norm{\big(Q_1, Q_2\big)}_{H^1}^2 +\norm{Q_3}_{L^2}^2 \right)
+C_\n\d_0 \big(G^R+G^S \big)
+C_\n \d_0\d_C \int_\RR W^2 (\phi^2+\chi^2) d\x.
\end{aligned}\end{equation}

We now exploit the structure of the system to obtain dissipation for the fluid variables \(v\), \(u\) and \(\th\) through the radiative effects.
We first use the radiative equation to recover the diffusion associated with \(\th\) and then turn to the auxiliary quantity \eqref{key-qty} to derive the corresponding diffusion for \(v\) and \(u\).

To this end, we observe from \eqref{pert-qeq} that
\begin{equation}\label{qeq-x}
4\b\th^3 \chi_\x
=\frac{r_{\x\x}}{v} - \frac{v_\x}{v^2}r_\x
-\a v r
-Q_3
-\a \qbar \phi
-4\b(\th^3-\thbar^3)\thbar_\x
-\left(\frac{\phi\qbar_\x}{v\vbar}\right)_\x.
\end{equation}
Since the coefficient \(4\b\th^3\) is uniformly bounded away from zero, this identity implies that
\begin{multline}\label{thx-df}
\norm{\chi_\x}_{L^2}^2 
\le C \norm{r}_{H^2}^2
+ C\norm{Q_3}_{L^2}^2
+ C\d_0^2 \norm{\big(\phi_\x,\psi_\x,\chi_\x\big)}_{L^2}^2\\
+C\d_0 \big(G^R+G^S \big)
+C\d_0 \d_C\int_\RR W^2 (\phi^2+\chi^2) d\x.
\end{multline}
Here, \eqref{vt-1}, \eqref{vxtx-1} and \eqref{first-o-3} were used.

In order to obtain the diffusion terms for the remaining variables \(v\) and \(u\), we introduce the following auxiliary quantity:
\begin{equation}\label{key-qty}
\frac{d}{dt} \int_\RR \left[\frac{R}{\g-1}\chi \psi_\x - \frac{p^*}{2}\psi \phi_\x\right]d\x.
\end{equation}
Using the perturbation equations \eqref{pert-eq}, we find that
\begin{align*}
\frac{d}{dt} \int_\RR \frac{R}{\g-1}\chi \psi_\x d\x
&= \int_\RR \frac{R}{\g-1} \big(\chi_t \psi_\x - \chi_\x \psi_t\big) d\x\\
&= \frac{R}{\g-1} \dot{X} \int_\RR \left[(\th^S)_\x \psi_\x - (u^S)_\x \chi_\x\right] d\x
+\int_\RR \left[\frac{R}{\g-1}\chi_\x Q_1 -\psi_\x Q_2\right]d\x\\
&\qquad
-\int_\RR r_\x \psi_\x d\x
-\int_\RR (pu_\x -\pbar\ubar_\x)\psi_\x d\x
+\frac{R}{\g-1}\int_\RR (p-\pbar)_\x \chi_\x d\x,
\end{align*}
and that
\begin{align*}
-\frac{p^*}{2} \frac{d}{dt} \int_\RR \psi \phi_\x d\x
&=\frac{p^*}{2} \int_\RR \big(\phi_t \psi_\x - \phi_\x \psi_t\big) d\x\\
&=\frac{p^*}{2} \dot{X} \int_\RR \left[(v^S)_\x \psi_\x - (u^S)_\x \phi_\x\right] d\x
+\frac{p^*}{2} \int_\RR \phi_\x Q_1 d\x\\
&\qquad
+\frac{p^*}{2} \int_\RR (p-\pbar)_\x \phi_\x d\x
+\frac{p^*}{2} \int_\RR (\psi_\x)^2 d\x.
\end{align*}
We also have the identity 
\begin{equation}\begin{aligned}\label{p-px}
(p-\pbar)_\x 
&=-\frac{p^*}{v^*}\phi_\x + \frac{R}{v^*}\chi_\x
+R\bigg(\frac{1}{v}-\frac{1}{v^*}\bigg)\chi_\x 
-\bigg(\frac{p}{v}-\frac{p^*}{v^*}\bigg)\phi_\x\\
&\qquad
+R\bigg(\frac{1}{v}-\frac{1}{\vbar}\bigg)\thbar_\x 
-\bigg(\frac{p}{v}-\frac{\pbar}{\vbar}\bigg)\vbar_\x.
\end{aligned}\end{equation}
From these three relations, we deduce that
\begin{align*}
&\frac{d}{dt} \int_\RR \left[\frac{R}{\g-1}\chi \psi_\x - \frac{p^*}{2}\psi \phi_\x\right]d\x
+\frac{(p^*)^2}{2v^*}\int_\RR (\phi_\x)^2 d\x
+\frac{p^*}{2}\int_\RR (\psi_\x)^2 d\x\\
&=\dot{X}\int_\RR \left[ \frac{R}{\g-1}\left((\th^S)_\x \psi_\x - (u^S)_\x \chi_\x\right)
+\frac{p^*}{2} \left((v^S)_\x \psi_\x - (u^S)_\x \phi_\x\right) \right] d\x\\
&\quad
+\int_\RR \left[ \frac{p^*}{2}\phi_\x Q_1 + \frac{R}{\g-1}\chi_\x Q_1 - \psi_\x Q_2 \right] d\x
-\int_\RR r_\x \psi_\x d\x\\
&\quad
+\int_\RR \left[ \frac{R^2}{(\g-1)v^*} (\chi_\x)^2
+\left(\frac{Rp^*}{2v^*}-\frac{Rp^*}{(\g-1)v^*}\right) \phi_\x \chi_\x \right] d\x \\
&\quad
-\int_\RR \left(\frac{R}{\g-1}\chi_\x + \frac{p^*}{2}\phi_\x\right)
\left[\bigg(\frac{p}{v}-\frac{p^*}{v^*}\bigg)\phi_\x
-R\bigg(\frac{1}{v}-\frac{1}{v^*}\bigg)\chi_\x\right] d\x
-\int_\RR (p-p^*)(\psi_\x)^2 d\x\\
&\quad
-\int_\RR \left(\frac{R}{\g-1}\chi_\x + \frac{p^*}{2}\phi_\x\right)
\left[\bigg(\frac{p}{v}-\frac{\pbar}{\vbar}\bigg)\vbar_\x
-R\bigg(\frac{1}{v}-\frac{1}{\vbar}\bigg)\thbar_\x\right] d\x
-\int_\RR (p-\pbar)\psi_\x \ubar_\x d\x
\eqqcolon \sum_{j=1}^8 L_j.
\end{align*}
We proceed to estimate each term on the right-hand side separately.

\bci{L_1} Young's inequality, along with Lemma \ref{lem:VS}, gives 
\begin{align*}
\abs{L_1}
&\le
\n \norm{\big(\phi_\x,\psi_\x,\chi_\x\big)}_{L^2}^2
+C_\n |\dot{X}|^2 \norm{\big((v^S)_{\x},(u^S)_{\x},(\th^S)_{\x}\big)}_{L^2}^2 \\
&\le
\n \norm{\big(\phi_\x,\psi_\x,\chi_\x\big)}_{L^2}^2
+C_\n \d_S^3 |\dot{X}|^2.
\end{align*}

\bci{L_2} An analogous application of Young's inequality shows that 
\[
\abs{L_2}
\le \n \norm{\big(\phi_\x,\psi_\x,\chi_\x\big)}_{L^2}^2
+C_\n\norm{\big(Q_1, Q_2\big)}_{L^2}^2.
\]

\bci{L_3} Applying Young's inequality once more, we find that
\[
\abs{L_3}
\le \n \norm{\psi_\x}_{L^2}^2
+C_\n \norm{r_\x}_{L^2}^2.
\]
Here, \(\n>0\) will be chosen small enough to absorb the first term on the right-hand side into the \(u\)-diffusion term.

\bci{L_4} For this term, we have 
\[
\abs{L_4}
\le \n \norm{\phi_\x}_{L^2}^2 
+C_\n \norm{\chi_\x}_{L^2}^2.
\]
The parameter \(\n>0\) will again be chosen sufficiently small so that the first term on the right-hand side can be absorbed into the \(v\)-diffusion term.

\bci{L_5+L_6} Using \eqref{assYap} and the smallness of the wave strengths, we obtain 
\[
\abs{L_5}+\abs{L_6}
\le C(\e_1+\d_0)\norm{\big(\phi_\x,\psi_\x,\chi_\x\big)}_{L^2}^2.
\]

\bci{L_7+L_8} The estimate \eqref{first-o-1}, combined with Young's inequality, provides the following bound:
\[
\abs{L_7}+\abs{L_8}
\le \n \norm{\big(\phi_\x,\psi_\x,\chi_\x\big)}_{L^2}^2
+C\d_0(G^R+G^S)
+C\d_0\d_C\int_\RR W^2 (\phi^2+\chi^2) d\x.
\]
The preceding estimates give
\begin{align*}
&\frac{d}{dt} \int_\RR \left[\frac{R}{\g-1}\chi \psi_\x - \frac{p^*}{2}\psi \phi_\x\right]d\x
+\frac{(p^*)^2}{2v^*}\int_\RR (\phi_\x)^2 d\x
+\frac{p^*}{2}\int_\RR (\psi_\x)^2 d\x\\
&\le
\big(\n+C(\e_1+\d_0)\big)\norm{\big(\phi_\x,\psi_\x,\chi_\x\big)}_{L^2}^2
+C_\n \norm{\chi_\x}_{L^2}^2
+C_\n \norm{r_\x}_{L^2}^2\\
&\qquad
+C_\n \d_S^3 |\dot{X}|^2
+C_\n\norm{\big(Q_1, Q_2\big)}_{L^2}^2
+C\d_0(G^R+G^S)
+C\d_0\d_C \int_\RR W^2 (\phi^2+\chi^2) d\x.
\end{align*}
We now fix \(\n=\min\left\{\frac{(p^*)^2}{8v^*},\frac{p^*}{8}\right\}>0\).
By choosing \(\d_0\) and \(\e_1\) sufficiently small, we then have
\begin{align*}
&\frac{d}{dt} \int_\RR \left[\frac{R}{\g-1}\chi \psi_\x - \frac{p^*}{2}\psi \phi_\x\right]d\x
+\frac{(p^*)^2}{4v^*}\int_\RR (\phi_\x)^2 d\x
+\frac{p^*}{4}\int_\RR (\psi_\x)^2 d\x\\
&\le C \norm{\chi_\x}_{L^2}^2
+C \norm{r_\x}_{L^2}^2
+C \d_S^3 |\dot{X}|^2
+C\norm{\big(Q_1, Q_2\big)}_{L^2}^2\\
&\qquad
+C\d_0(G^R+G^S)
+C\d_0\d_C \int_\RR W^2 (\phi^2+\chi^2) d\x.
\end{align*}
The upper bound on \(\norm{\chi_\x}_{L^2}^2\) in \eqref{thx-df} then allows us to recover the diffusion terms for \(v\), \(u\) and \(\th\) as follows:
\begin{multline}\label{h1-vut}
\frac{d}{dt} \int_\RR \left[\frac{R}{\g-1}\chi \psi_\x - \frac{p^*}{2}\psi \phi_\x\right]d\x
+c \norm{\big(\phi_\x,\psi_\x,\chi_\x\big)}_{L^2}^2
\le C \norm{r}_{H^2}^2
+C \d_S^3 |\dot{X}|^2\\
+C\norm{\big(Q_1, Q_2, Q_3\big)}_{L^2}^2
+C\d_0(G^R+G^S)
+C\d_0\d_C \int_\RR W^2 (\phi^2+\chi^2) d\x.
\end{multline}

We next combine the two estimates \eqref{h1-r} and \eqref{h1-vut}.
Since
\[
\int_\RR \left[\frac{R}{\g-1}\chi \psi_\x - \frac{p^*}{2}\psi \phi_\x\right] d\x 
\le C\norm{\big(\phi,\psi,\chi\big)}_{L^2}
\norm{\big(\phi_\x,\psi_\x,\chi_\x\big)}_{L^2},
\]
we choose a fixed constant \(\n_*>0\), small enough and independent of \(\d_0\) and \(\e_1\), such that
\begin{align*}
&\int_\RR \left[
\frac{R\th}{2v^2}(\phi_\x)^2
+ \frac{1}{2}(\psi_\x)^2
+ \frac{R}{2(\g-1)\th}(\chi_\x)^2 \right]
+ \n_*\left[\frac{R}{\g-1}\chi \psi_\x - \frac{p^*}{2}\psi \phi_\x\right] d\x\\
&\hspace{30mm}
\ge c\norm{\big(\phi_\x,\psi_\x,\chi_\x\big)}_{L^2}^2
-C\norm{\big(\phi,\psi,\chi\big)}_{L^2}^2.
\end{align*}
We apply \eqref{h1-r} and \eqref{h1-vut} to the time-derivative on the left-hand side of this inequality, and then fix \(\n>0\) in \eqref{h1-r} sufficiently small and choose \(\d_0\) and \(\e_1\) accordingly.
This gives
\begin{align*}
&\frac{d}{dt}\int_\RR \left[
\frac{R\th}{2v^2}(\phi_\x)^2
+ \frac{1}{2}(\psi_\x)^2
+ \frac{R}{2(\g-1)\th}(\chi_\x)^2 \right]
+ \n_*\left[\frac{R}{\g-1}\chi \psi_\x - \frac{p^*}{2}\psi \phi_\x\right] d\x\\
&\qquad
+c\norm{\big(\phi_\x,\psi_\x,\chi_\x\big)}_{L^2}^2
+c\norm{r_\x}_{H^1}^2\\
&\le C\norm{r}_{H^1}^2
+C \d_S^3 |\dot{X}|^2
+C\norm{\big(Q_1, Q_2\big)}_{H^1}^2
+C\norm{Q_3}_{L^2}^2\\
&\qquad
+C\d_0(G^R+G^S)
+C\d_0\d_C \int_\RR W^2 (\phi^2+\chi^2) d\x.
\end{align*}
Integrating this over time, we obtain
\begin{equation}\label{eq:h1est}
\begin{aligned}
&\sup_{\t\in[0,t]}\norm{\big(\phi_\x,\psi_\x,\chi_\x\big)(\t)}_{L^2}^2
+\int_0^t \left[ \norm{\big(\phi_\x,\psi_\x,\chi_\x\big)(\t)}_{L^2}^2 + \norm{r_\x(\t)}_{H^1}^2 \right] d\t\\
&\le C \norm{\big(\phi,\psi,\chi\big)(0)}_{H^1}^2
+C \sup_{\t\in[0,t]}\norm{\big(\phi,\psi,\chi\big)(\t)}_{L^2}^2
+C\int_0^t \left[\d_S |\dot{X}|^2 + G^R + G^S + \norm{r}_{H^1}^2\right] d\t \\
&\qquad
+C\d_C\int_0^t \int_\RR \big(W(\t,\x)\big)^2 (\phi^2+\chi^2) d\x d\t
+C\int_0^t \left[ \norm{(Q_1,Q_2)}_{H^1}^2 + \norm{Q_3}_{L^2}^2\right] d\t.
\end{aligned}
\end{equation}

\subsection{Second-Order Estimate}
We turn to the second-order estimate, which will close the \textit{a priori} estimate.
As outlined above, we follow the same strategy as in the first-order case, combining standard energy estimates with the same choice of multipliers and auxiliary quantity, but with one additional spatial derivative.

\vspace{2mm}
Taking two derivatives of \eqref{pert-eq} with respect to \(\x\), testing the resulting equations against \(\frac{R\th}{v^2}\phi_{\x\x}\), \(\psi_{\x\x}\), and \(\frac{\chi_{\x\x}}{\th}\), respectively, and integrating over \(\x\), we arrive at
\begin{align*}
&\frac{d}{dt}\int_\RR \left[
\frac{R\th}{2v^2}(\phi_{\x\x})^2
+ \frac{1}{2}(\psi_{\x\x})^2
+ \frac{R}{2(\g-1)\th}(\chi_{\x\x})^2 \right] d\x
+\int_\RR \frac{\chi_{\x\x}}{\th} r_{\x\x\x} d\x\\
&=\dot{X}\int_\RR \left[\frac{R\th}{v^2}\phi_{\x\x} (v^S)_{\x\x\x}
+\psi_{\x\x}(u^S)_{\x\x\x}
+\frac{R \chi_{\x\x}}{(\g-1)\th}(\th^S)_{\x\x\x}\right] d\x\\
&\qquad
-\int_\RR \left[\psi_{\x\x} (Q_1)_{\x\x} + \frac{\chi_{\x\x}}{\th}(Q_2)_{\x\x}\right] d\x\\
&\qquad
+\int_\RR \left[-\left(\frac{(\g+1)R\th}{2v^3}u_\x + \frac{\g-1}{2v^2}q_\x\right)(\phi_{\x\x})^2
+ \frac{pu_\x+q_\x}{2\th^2}(\chi_{\x\x})^2 \right] d\x\\
&\qquad
+\int_\RR \left[ \frac{R\th}{v^2}\phi_{\x\x} \psi_{\x\x\x} -\psi_{\x\x}(p-\pbar)_{\x\x\x} -\frac{\chi_{\x\x}}{\th}(p u_\x -\pbar \ubar_\x)_{\x\x} \right] d\x.
\end{align*}
In view of the identity
\[
-\frac{R\th}{v^2}\phi_{\x\x}
-(p-\pbar)_{\x\x}
+\frac{R}{v}\chi_{\x\x}
=\left(\frac{R\th}{v^2}\right)_\x \phi_\x 
-\left(\frac{R}{v}\right)_\x \chi_\x
+\left[\left(\frac{R\th}{v^2}-\frac{R\thbar}{\vbar^2}\right)\vbar_\x
-R\left(\frac{1}{v}-\frac{1}{\vbar}\right)\thbar_\x\right]_\x,
\]
the last line of the preceding expression reduces to
\begin{align*}
&\int_\RR \left[ \frac{R\th}{v^2}\phi_{\x\x} \psi_{\x\x\x} -\psi_{\x\x}(p-\pbar)_{\x\x\x} -\frac{\chi_{\x\x}}{\th}(p u_\x -\pbar \ubar_\x)_{\x\x} \right] d\x\\
&=\int_\RR \psi_{\x\x} \left[ \left(\frac{R\th}{v^2}\right)_\x \phi_\x 
-\left(\frac{R}{v}\right)_\x \chi_\x \right]_\x d\x
+\int_\RR \psi_{\x\x} \left[\left(\frac{R\th}{v^2}-\frac{R\thbar}{\vbar^2}\right)\vbar_\x
-R\left(\frac{1}{v}-\frac{1}{\vbar}\right)\thbar_\x \right]_{\x\x} d\x\\
&\qquad
-\int_\RR \frac{\chi_{\x\x}}{\th}p_{\x\x} \psi_\x d\x
-2\int_\RR \frac{\chi_{\x\x}}{\th}p_\x \psi_{\x\x} d\x
-\int_\RR \frac{\chi_{\x\x}}{\th} \big((p -\pbar)\ubar_\x\big)_{\x\x} d\x.
\end{align*}
Moreover, to obtain the corresponding higher-order diffusion for \(r\), we differentiate \eqref{pert-qeq} once with respect to \(\x\), test the resulting equation against \(-\frac{r_{\x\x\x}}{4\b\th^4}\), and integrate over \(\RR\).
This yields
\begin{align*}
&\int_\RR \frac{(r_{\x\x\x})^2}{4\b v \th^4} d\x
+\int_\RR \frac{\a v (r_{\x\x})^2}{4\b \th^4} d\x
-\int_\RR \frac{\chi_{\x\x}}{\th} r_{\x\x\x} d\x\\
&=-2\int_\RR \frac{1}{4\b\th^4}\Big(\frac{1}{v}\Big)_\x r_{\x\x}r_{\x\x\x}d\x
-\int_\RR \frac{1}{4\b\th^4}\Big(\frac{1}{v}\Big)_{\x\x} r_\x r_{\x\x\x}d\x\\
&\qquad
-\int_\RR \Big(\frac{\a v}{4\b \th^4}\Big)_\x r_\x r_{\x\x} d\x
+\int_\RR \frac{\a v_\x}{4\b \th^4}r r_{\x\x\x} d\x
+3\int_\RR \frac{\th_\x}{\th^2}\chi_\x r_{\x\x\x} d\x\\
&\qquad
+\int_\RR \frac{r_{\x\x\x}}{4\b\th^4}\left[(Q_3)_\x
+ \a(\qbar \phi)_\x
+ 4\b\big((\th^3-\thbar^3)\thbar_\x\big)_\x
+ \left(\frac{\phi\qbar_\x}{v\vbar}\right)_{\x\x}\right] d\x.
\end{align*}
The preceding estimates together imply that
\begin{equation} \label{h2begin}
\begin{aligned}
&\frac{d}{dt}\int_\RR \left[
\frac{R\th}{2v^2}(\phi_{\x\x})^2
+ \frac{1}{2}(\psi_{\x\x})^2
+ \frac{R}{2(\g-1)\th}(\chi_{\x\x})^2 \right] d\x
+\int_\RR \frac{(r_{\x\x\x})^2}{4\b v \th^4} d\x
+\int_\RR \frac{\a v (r_{\x\x})^2}{4\b \th^4} d\x\\
&=\dot{X}\int_\RR \left[\frac{R\th}{v^2}\phi_{\x\x} (v^S)_{\x\x\x}
+\psi_{\x\x}(u^S)_{\x\x\x}
+\frac{R \chi_{\x\x}}{(\g-1)\th}(\th^S)_{\x\x\x}\right] d\x\\
&\qquad
-\int_\RR \left[\psi_{\x\x} (Q_1)_{\x\x} + \frac{\chi_{\x\x}}{\th}(Q_2)_{\x\x}\right] d\x\\
&\qquad
-\int_\RR \frac{\chi_{\x\x}}{\th} \big((p -\pbar)\ubar_\x\big)_{\x\x} d\x
+\int_\RR \psi_{\x\x} \left[\Big(\frac{R\th}{v^2}-\frac{R\thbar}{\vbar^2}\Big)\vbar_\x
-R\Big(\frac{1}{v}-\frac{1}{\vbar}\Big)\thbar_\x \right]_{\x\x} d\x\\
&\qquad
+\int_\RR \left[-\Big(\frac{(\g+1)R\th}{2v^3}u_\x + \frac{\g-1}{2v^2}q_\x\Big)(\phi_{\x\x})^2
+ \frac{pu_\x+q_\x}{2\th^2}(\chi_{\x\x})^2 \right] d\x\\
&\qquad
+\int_\RR \left[\Big(\frac{R\th}{v^2}\Big)_\x \phi_{\x\x} \psi_{\x\x}
-\left(\Big(\frac{R}{v}\Big)_\x+\frac{2p_\x}{\th}\right)\psi_{\x\x}\chi_{\x\x}\right] d\x\\
&\qquad
+\int_\RR \left[\Big(\frac{R\th}{v^2}\Big)_{\x\x} \phi_\x \psi_{\x\x}
-\frac{p_{\x\x}}{\th} \psi_\x \chi_{\x\x}
-\Big(\frac{R}{v}\Big)_{\x\x} \chi_\x \psi_{\x\x} \right] d\x\\
&\qquad
-\int_\RR \left[\left(
\frac{1}{2\b\th^4}\Big(\frac{1}{v}\Big)_\x r_{\x\x}
+\frac{1}{4\b\th^4}\Big(\frac{1}{v}\Big)_{\x\x} r_\x 
-\frac{\a v_\x}{4\b \th^4}r 
-3\frac{\th_\x}{\th^2}\chi_\x
\right) r_{\x\x\x}
+\Big(\frac{\a v}{4\b \th^4}\Big)_\x r_\x r_{\x\x}\right] d\x\\
&\qquad
+\int_\RR \frac{r_{\x\x\x}}{4\b\th^4}\left[Q_3
+ \a \qbar \phi
+ 4\b(\th^3-\thbar^3)\thbar_\x
+ \Big(\frac{\phi\qbar_\x}{v\vbar}\Big)_\x\right]_\x d\x
\eqqcolon \sum_{j=1}^9 L_j.
\end{aligned}
\end{equation}
We proceed by estimating each term on the right-hand side separately.\\

\bci{L_1} We begin with Lemma \ref{lem:VS}, which gives 
\[
\norm{\big((v^S)_{\x\x\x},(u^S)_{\x\x\x},(\th^S)_{\x\x\x}\big)}_{L^2}^2
\le C \d_S^7.
\]
Consequently, 
\begin{align*}
\abs{L_1}
&\le \nu \norm{\big(\phi_{\x\x},\psi_{\x\x},\chi_{\x\x}\big)}_{L^2}^2
+C_\n |\dot{X}|^2 \norm{\big((v^S)_{\x\x\x},(u^S)_{\x\x\x},(\th^S)_{\x\x\x}\big)}_{L^2}^2\\
&\le \nu \norm{\big(\phi_{\x\x},\psi_{\x\x},\chi_{\x\x}\big)}_{L^2}^2
+C_\n \d_S^7 |\dot{X}|^2.
\end{align*}

\bci{L_2} An application of Young's inequality yields 
\[
\abs{L_2} 
\le \n \norm{\big(\phi_{\x\x},\psi_{\x\x},\chi_{\x\x}\big)}_{L^2}^2 
+ C_\n \norm{\big(Q_1, Q_2\big)}_{H^2}^2.
\]

\bci{L_3+L_4} Young's inequality also provides 
\begin{multline*}
\abs{L_3}+\abs{L_4}
\le \n \norm{\big(\phi_{\x\x},\psi_{\x\x},\chi_{\x\x}\big)}_{L^2}^2\\
+C_\n \int_\RR \left(\left[\left(\frac{R\th}{v^2}-\frac{R\thbar}{\vbar^2}\right)\vbar_\x
-R\left(\frac{1}{v}-\frac{1}{\vbar}\right)\thbar_\x\right]_{\x\x}\right)^2 d\x
+C_\n \int_\RR \big(\big((p-\pbar)\ubar_\x\big)_{\x\x}\big)^2 d\x.
\end{multline*}
In view of \eqref{sec-o-1}, we further deduce that
\begin{multline*}
\abs{L_3}+\abs{L_4}
\le \n \norm{\big(\phi_{\x\x},\psi_{\x\x},\chi_{\x\x}\big)}_{L^2}^2\\
+C_\n \d_0^2 \norm{\big(\phi_\x,\psi_\x,\chi_\x\big)}_{H^1}^2
+C_\n \d_0 \big(G^R+G^S \big)
+C_\n \d_0\d_C\int_\RR W^2 (\phi^2+\chi^2) d\x.
\end{multline*}

\bci{L_5} The bounds in \eqref{vxtx-1} and \eqref{qh3} provide the following bound for \(L_5\): 
\[
\abs{L_5}
\le C \norm{(v_\x,u_\x,\th_\x,q_\x)}_{L^\infty} 
\norm{\big(\phi_{\x\x},\psi_{\x\x},\chi_{\x\x}\big)}_{L^2}^2
\le C(\e_1 + \d_0) \norm{\big(\phi_{\x\x},\psi_{\x\x},\chi_{\x\x}\big)}_{L^2}^2.
\]

\bci{L_6} As in the estimate of \(L_5\), we use \eqref{vxtx-1} to obtain
\begin{align*}
\abs{L_6}
&\le C \norm{(v_\x,u_\x,\th_\x)}_{L^\infty}
\Big(\norm{\phi_{\x\x}}_{L^2}\norm{\psi_{\x\x}}_{L^2}
+\norm{\psi_{\x\x}}_{L^2}\norm{\chi_{\x\x}}_{L^2}\Big)\\
&\le C(\e_1+\d_0) \norm{\big(\phi_{\x\x},\psi_{\x\x},\chi_{\x\x}\big)}_{L^2}^2.
\end{align*}

\bci{L_7} We have the following identities:
\begin{align*}
\Big(\frac{R\th}{v^2}\Big)_{\x\x}
&=\frac{R\th_{\x\x}}{v^2}
-\frac{2R\th}{v^3}v_{\x\x}
+\frac{6R\th}{v^4}(v_\x)^2
-\frac{4R}{v^3}v_\x \th_\x,\\
\frac{p_{\x\x}}{\th}
&=\frac{R}{v\th}\th_{\x\x}
-\frac{R}{v^2}v_{\x\x}
+\frac{2R}{v^3}(v_\x)^2
-\frac{2R}{v^2 \th}v_\x \th_\x,\\
\Big(\frac{R}{v}\Big)_{\x\x}
&=-\frac{Rv_{\x\x}}{v^2}
+\frac{2R(v_\x)^2}{v^3}.
\end{align*}
By \eqref{vxtx-1}, the first-order derivatives are uniformly small in \(L^\infty\), whereas the second-order derivatives are known to be small only in \(L^2\).
To account for this distinction, we decompose the preceding expressions as follows:
\begin{equation}\label{2th-d}
\begin{aligned}
\Big(\frac{R\th}{v^2}\Big)_{\x\x}
&=\frac{R\chi_{\x\x}}{v^2}
-\frac{2R\th}{v^3}\phi_{\x\x}
+\frac{6R\th}{v^4}(v_\x)^2
-\frac{4R}{v^3}v_\x \th_\x
+\frac{R\thbar_{\x\x}}{v^2}
-\frac{2R\th}{v^3}\vbar_{\x\x},\\
\frac{p_{\x\x}}{\th}
&=\frac{R}{v\th}\chi_{\x\x}
-\frac{R}{v^2}\phi_{\x\x}
+\frac{2R}{v^3}(v_\x)^2
-\frac{2R}{v^2 \th}v_\x \th_\x
+\frac{R}{v\th}\thbar_{\x\x}
-\frac{R}{v^2}\vbar_{\x\x},\\
\Big(\frac{R}{v}\Big)_{\x\x}
&=-\frac{R\phi_{\x\x}}{v^2}
+\frac{2R(v_\x)^2}{v^3}
-\frac{R\vbar_{\x\x}}{v^2}.
\end{aligned}
\end{equation}
The resulting decomposition, along with \eqref{assYap} and \eqref{vxtx-1}, gives the bound
\begin{align*}
\abs{L_7}
&\le C\e_1 \norm{\big(\phi_{\x\x},\psi_{\x\x},\chi_{\x\x}\big)}_{L^2}^2
+C(\e_1+\d_0)^2 \norm{\big(\phi_\x,\psi_\x,\chi_\x\big)}_{H^1}^2
+C\d_0 \norm{\big(\phi_\x,\psi_\x,\chi_\x\big)}_{H^1}^2\\
&\le C(\e_1+\d_0)\norm{\big(\phi_\x,\psi_\x,\chi_\x\big)}_{H^1}^2.
\end{align*}
Here, Lemmas \ref{lem:uR}--\ref{lem:VS} were also used.

\vspace{2mm}
\bci{L_8} To handle \(L_8\), we label its five constituent terms as \(L_{81},\ldots,L_{85}\), in order of appearance.
Except for \(L_{82}\), all these terms can be readily estimated using \eqref{vxtx-1} as follows:
\begin{align*}
\abs{L_{81}}+\abs{L_{83}}+\abs{L_{84}}+\abs{L_{85}}
&\le C \norm{(v_\x,u_\x,\th_\x)}_{L^\infty}
\Big(
\norm{r_{\x\x}}_{L^2} \norm{r_{\x\x\x}}_{L^2}
+\norm{r}_{L^2} \norm{r_{\x\x\x}}_{L^2}\\
&\hspace{40mm}
+\norm{\chi_\x}_{L^2} \norm{r_{\x\x\x}}_{L^2}
+\norm{r_\x}_{L^2} \norm{r_{\x\x}}_{L^2}
\Big)\\
&\le C(\e_1+\d_0) \Big(\norm{r}_{H^3}^2
+\norm{\chi_\x}_{L^2}^2\Big).
\end{align*}
The estimate of \(L_{82}\) requires the last decomposition in \eqref{2th-d}, as well as \eqref{vxtx-1} and \eqref{rh3}.
With the aid of Lemmas \ref{lem:uR}--\ref{lem:VS}, we then obtain
\begin{align*}
\abs{L_{82}}
&\le C \norm{(r_\x,v_\x,\vbar_{\x\x})}_{L^\infty}
\Big(
\norm{\phi_{\x\x}}_{L^2} \norm{r_{\x\x\x}}_{L^2}
+\norm{r_\x}_{L^2} \norm{r_{\x\x\x}}_{L^2}
\Big)\\
&\le C(\e_1+\d_0) \Big(\norm{\phi_{\x\x}}_{L^2}^2 + \norm{r}_{H^3}^2 \Big).
\end{align*}
These estimates lead to the following bound for \(L_8\): 
\[
\abs{L_8}
\le C(\e_1+\d_0) \Big(\norm{\phi_{\x\x}}_{L^2}^2
+\norm{\chi_\x}_{L^2}^2 +\norm{r}_{H^3}^2\Big).
\]

\bci{L_9} An application of Young's inequality to \eqref{sec-o-2} gives
\begin{multline*}
\abs{L_9}
\le \n \norm{r_{\x\x\x}}_{L^2}^2
+C_\n \norm{Q_3}_{H^1}^2\\
+C_\n \d_0^2 \norm{\big(\phi_\x,\psi_\x,\chi_\x\big)}_{H^1}^2
+C_\n \d_0 \big(G^R+G^S \big)
+C_\n \d_0\d_C \int_\RR W^2 (\phi^2+\chi^2) d\x.
\end{multline*}
We now use the bounds established above in \eqref{h2begin} to deduce that
\begin{align*}
&\frac{d}{dt}\int_\RR \left[
\frac{R\th}{2v^2}(\phi_{\x\x})^2
+ \frac{1}{2}(\psi_{\x\x})^2
+ \frac{R}{2(\g-1)\th}(\chi_{\x\x})^2 \right] d\x
+\int_\RR \frac{(r_{\x\x\x})^2}{4\b v \th^4} d\x
+\int_\RR \frac{\a v (r_{\x\x})^2}{4\b \th^4} d\x\\
&\le \big(\n + C_\n \d_0^2 + C(\e_1+\d_0)\big)
\norm{\big(\phi_\x,\psi_\x,\chi_\x\big)}_{H^1}^2\\
&\qquad
+C(\e_1 + \d_0) \left(\norm{r}_{H^2}^2 + \norm{r_\x}_{H^2}^2\right)
+ \big(\n+C(\e_1+\d_0)\big) \norm{r_{\x\x\x}}_{L^2}^2\\
&\qquad
+C_\n \d_S^7 |\dot{X}|^2
+C_\n \left(\norm{\big(Q_1, Q_2\big)}_{H^2}^2 +\norm{Q_3}_{H^1}^2 \right)\\
&\qquad
+C_\n\d_0 \big(G^R+G^S \big)
+C_\n \d_0\d_C \int_\RR W^2 (\phi^2+\chi^2) d\x.
\end{align*}
Since the coefficients \(\frac{1}{4\b v\th^4}\) and \(\frac{\a v}{4\b \th^4}\) are uniformly bounded below by positive constants, we choose \(\n\), \(\d_0\), and \(\e_1\) suitably small so that the following holds:
\begin{equation}\begin{aligned}\label{h2-r}
&\frac{d}{dt}\int_\RR \left[
\frac{R\th}{2v^2}(\phi_{\x\x})^2
+ \frac{1}{2}(\psi_{\x\x})^2
+ \frac{R}{2(\g-1)\th}(\chi_{\x\x})^2 \right] d\x
+ c\norm{r_{\x\x}}_{H^1}^2\\
&\le \big(\n + C_\n (\e_1+\d_0)\big) \norm{\big(\phi_\x,\psi_\x,\chi_\x\big)}_{H^1}^2
+C(\e_1 + \d_0) \norm{r}_{H^2}^2
+C_\n \d_S^7 |\dot{X}|^2\\
&\hspace{2mm}
+C_\n \left(\norm{\big(Q_1, Q_2\big)}_{H^2}^2 +\norm{Q_3}_{H^1}^2 \right)
+C_\n\d_0 \big(G^R+G^S \big)
+C_\n \d_0\d_C \int_\RR W^2 (\phi^2+\chi^2) d\x.
\end{aligned}\end{equation}

We proceed to recover the higher-order diffusion for \(v\), \(u\) and \(\th\) by exploiting the same structure of the system.
As in the first-order case, the diffusion for \(\th\) is obtained from the radiative equation, whereas that for \(v\) and \(u\) will be derived using the auxiliary quantity \eqref{key-qty} at one higher derivative level.

Differentiating the radiative equation in \eqref{qeq-x} once with respect to \(\x\), we have
\begin{multline*}
4\b\th^3 \chi_{\x\x}
=-12\b\th^2\th_\x \chi_\x
+\frac{r_{\x\x\x}}{v}
-\frac{2v_\x}{v^2}r_{\x\x}
+\frac{2(v_\x)^2}{v^3}r_\x
-\frac{v_{\x\x}}{v^2}r_\x
-\a v_\x r
-\a v r_\x\\
-\Big(Q_3
+\a \qbar \phi
+4\b(\th^3-\thbar^3)\thbar_\x\Big)_\x
-\left(\frac{\phi\qbar_\x}{v\vbar}\right)_{\x\x}.
\end{multline*}
The terms in the second line of the right-hand side can be controlled collectively using \eqref{sec-o-2}, while those in the first line are estimated using \eqref{vt-1} and \eqref{vxtx-1} as follows:
\begin{align*}
\norm{-12\b\th^2 \th_\x \chi_\x}_{L^2}^2
&\le C(\e_1+\d_0)^2 \norm{\chi_\x}_{L^2}^2,
&\norm{\frac{r_{\x\x\x}}{v}}_{L^2}^2
&\le C\norm{r_{\x\x\x}}_{L^2}^2,\\
\norm{-\frac{2v_\x}{v^2}r_{\x\x}}_{L^2}^2
&\le C(\e_1+\d_0)^2 \norm{r_{\x\x}}_{L^2}^2,
&\norm{\frac{2(v_\x)^2}{v^3}r_\x}_{L^2}^2
&\le C(\e_1+\d_0)^2 \norm{r_\x}_{L^2}^2,\\
\norm{-\a v_\x r}_{L^2}^2
&\le C(\e_1+\d_0)^2 \norm{r}_{L^2}^2,
&\norm{-\a v r_\x}_{L^2}^2
&\le C \norm{r_\x}_{L^2}^2,
\end{align*}
and 
\[
\norm{-\frac{v_{\x\x}}{v^2}r_\x}_{L^2}^2
=\norm{-\frac{\phi_{\x\x}}{v^2}r_\x -\frac{\vbar_{\x\x}}{v^2}r_\x}_{L^2}^2
\le C(\e_1+\d_0)^2 \norm{\phi_{\x\x}}_{L^2}^2
+C\d_0^2 \norm{r_\x}_{L^2}^2.
\]
Indeed, the last inequality follows from \eqref{rh3} and Lemmas \ref{lem:uR}--\ref{lem:VS}.\\
The uniform positivity of \(4\b\th^3\) then allows us to establish
\begin{multline}\label{thxx-df}
\norm{\chi_{\x\x}}_{L^2}^2 
\le C \norm{r}_{H^3}^2
+ C\norm{Q_3}_{H^1}^2
+ C(\e_1+\d_0)^2 \norm{\big(\phi_\x,\psi_\x,\chi_\x\big)}_{H^1}^2\\
+C\d_0 \big(G^R+G^S \big)
+C\d_0 \d_C\int_\RR W^2 (\phi^2+\chi^2) d\x.
\end{multline}

To recover the corresponding second-order diffusion terms for \(v\) and \(u\), we consider the higher-order counterpart of the auxiliary quantity in \eqref{key-qty}: 
\[
\frac{d}{dt} \int_\RR \left[\frac{R}{\g-1}\chi_\x \psi_{\x\x} - \frac{p^*}{2}\psi_\x \phi_{\x\x}\right]d\x.
\]
Differentiating the perturbation equations \eqref{pert-eq} once with respect to \(\x\), we observe that
\begin{align*}
&\frac{d}{dt} \int_\RR \frac{R}{\g-1}\chi_\x \psi_{\x\x} d\x
= \int_\RR \frac{R}{\g-1} \big((\chi_t)_\x \psi_{\x\x} - \chi_{\x\x} (\psi_t)_\x\big) d\x\\
&= \frac{R}{\g-1} \dot{X} \int_\RR \left[(\th^S)_{\x\x} \psi_{\x\x} - (u^S)_{\x\x} \chi_{\x\x}\right] d\x
+\int_\RR \left[\frac{R}{\g-1}\chi_{\x\x} (Q_1)_\x -\psi_{\x\x} (Q_2)_\x\right]d\x\\
&\qquad\qquad
-\int_\RR r_{\x\x} \psi_{\x\x} d\x
-\int_\RR (pu_\x -\pbar\ubar_\x)_\x \psi_{\x\x} d\x
+\frac{R}{\g-1}\int_\RR (p-\pbar)_{\x\x} \chi_{\x\x} d\x
\end{align*}
and that
\begin{align*}
&-\frac{p^*}{2} \frac{d}{dt} \int_\RR \psi_\x \phi_{\x\x} d\x
=\frac{p^*}{2} \int_\RR \big((\phi_t)_\x \psi_{\x\x} - \phi_{\x\x} (\psi_t)_\x\big) d\x\\
&=\frac{p^*}{2} \dot{X} \int_\RR \left[(v^S)_{\x\x} \psi_{\x\x} - (u^S)_{\x\x} \phi_{\x\x}\right] d\x
+\frac{p^*}{2} \int_\RR \phi_{\x\x} (Q_1)_\x d\x\\
&\qquad\qquad
+\frac{p^*}{2} \int_\RR (p-\pbar)_{\x\x} \phi_{\x\x} d\x
+\frac{p^*}{2} \int_\RR (\psi_{\x\x})^2 d\x.
\end{align*}
Taking a \(\x\)-derivative of \eqref{p-px}, we find that
\begin{multline*}
(p-\pbar)_{\x\x} 
=-\frac{p^*}{v^*}\phi_{\x\x} + \frac{R}{v^*}\chi_{\x\x}\\
+\left[R\bigg(\frac{1}{v}-\frac{1}{v^*}\bigg)\chi_\x 
-\bigg(\frac{p}{v}-\frac{p^*}{v^*}\bigg)\phi_\x
+R\bigg(\frac{1}{v}-\frac{1}{\vbar}\bigg)\thbar_\x 
-\bigg(\frac{p}{v}-\frac{\pbar}{\vbar}\bigg)\vbar_\x\right]_\x.
\end{multline*}
We further note that 
\[
-\int_\RR (pu_\x -\pbar\ubar_\x)_\x \psi_{\x\x} d\x
=-p^*\int_\RR (\psi_{\x\x})^2 d\x
-\int_\RR \left[(p-p^*)\psi_\x + (p-\pbar)\ubar_\x\right]_\x \psi_{\x\x} d\x.
\]
Taken together, these four relations imply that
\begin{align*}
&\frac{d}{dt} \int_\RR \left[\frac{R}{\g-1}\chi_\x \psi_{\x\x} - \frac{p^*}{2}\psi_\x \phi_{\x\x}\right]d\x
+\frac{(p^*)^2}{2v^*}\int_\RR (\phi_{\x\x})^2 d\x
+\frac{p^*}{2}\int_\RR (\psi_{\x\x})^2 d\x\\
&=\dot{X}\int_\RR \left[ \frac{R}{\g-1}\left((\th^S)_{\x\x} \psi_{\x\x} - (u^S)_{\x\x} \chi_{\x\x}\right)
+\frac{p^*}{2} \left((v^S)_{\x\x} \psi_{\x\x} - (u^S)_{\x\x} \phi_{\x\x}\right) \right] d\x\\
&\quad
+\int_\RR \left[ \frac{p^*}{2}\phi_{\x\x} (Q_1)_\x + \frac{R}{\g-1}\chi_{\x\x} (Q_1)_\x - \psi_{\x\x} (Q_2)_\x \right] d\x
-\int_\RR r_{\x\x} \psi_{\x\x} d\x\\
&\quad
+\int_\RR \left[ \frac{R^2}{(\g-1)v^*} (\chi_{\x\x})^2
+\left(\frac{Rp^*}{2v^*}-\frac{Rp^*}{(\g-1)v^*}\right) \phi_{\x\x} \chi_{\x\x} \right] d\x \\
&\quad
-\int_\RR \left(\frac{R}{\g-1}\chi_{\x\x} + \frac{p^*}{2}\phi_{\x\x}\right)
\left[\bigg(\frac{p}{v}-\frac{p^*}{v^*}\bigg)\phi_\x
-R\bigg(\frac{1}{v}-\frac{1}{v^*}\bigg)\chi_\x\right]_\x d\x
-\int_\RR \Big[(p-p^*)\psi_\x\Big]_\x \psi_{\x\x} d\x\\
&\quad
-\int_\RR \left(\frac{R}{\g-1}\chi_{\x\x} + \frac{p^*}{2}\phi_{\x\x}\right)
\left[\bigg(\frac{p}{v}-\frac{\pbar}{\vbar}\bigg)\vbar_\x
-R\bigg(\frac{1}{v}-\frac{1}{\vbar}\bigg)\thbar_\x\right]_\x d\x
-\int_\RR \Big[(p-\pbar)\ubar_\x\Big]_\x \psi_{\x\x} d\x\\
&\eqqcolon \sum_{j=1}^8 L_j.
\end{align*}
We now estimate the right-hand side on a term-by-term basis.

\bci{L_1} Using Lemma \ref{lem:VS} and Young's inequality, we derive the bound
\begin{align*}
\abs{L_1}
&\le
\n \norm{\big(\phi_{\x\x},\psi_{\x\x},\chi_{\x\x}\big)}_{L^2}^2
+C_\n |\dot{X}|^2 \norm{\big((v^S)_{\x\x},(u^S)_{\x\x},(\th^S)_{\x\x}\big)}_{L^2}^2 \\
&\le
\n \norm{\big(\phi_{\x\x},\psi_{\x\x},\chi_{\x\x}\big)}_{L^2}^2
+C_\n \d_S^5 |\dot{X}|^2.
\end{align*}

\bci{L_2} Young's inequality also provides
\[
\abs{L_2}
\le \n \norm{\big(\phi_{\x\x},\psi_{\x\x},\chi_{\x\x}\big)}_{L^2}^2
+C_\n\norm{\big(Q_1, Q_2\big)}_{H^1}^2.
\]

\bci{L_3} A further application of Young's inequality leads to
\[
\abs{L_3}
\le \n \norm{\psi_{\x\x}}_{L^2}^2
+C_\n \norm{r_{\x\x}}_{L^2}^2.
\]
The first term on the right-hand side can be absorbed into the \(u\)-diffusion term by choosing \(\n>0\) sufficiently small.

\bci{L_4} The corresponding bound for \(L_4\) reads
\[
\abs{L_4}
\le \n \norm{\phi_{\x\x}}_{L^2}^2 
+C_\n \norm{\chi_{\x\x}}_{L^2}^2.
\]
Similarly, an appropriate choice of \(\n>0\) allows us to absorb the first term on the right-hand side into the \(v\)-diffusion term.

\bci{L_5+L_6} To estimate \(L_5+L_6\), we use the following decomposition:
\begin{multline*}
L_5+L_6
=-\int_\RR \left(\frac{R}{\g-1}\chi_{\x\x} + \frac{p^*}{2}\phi_{\x\x}\right)
\left[\bigg(\frac{p}{v}\bigg)_\x\phi_\x
-R\bigg(\frac{1}{v}\bigg)_\x\chi_\x\right] d\x
-\int_\RR p_\x \psi_\x \psi_{\x\x} d\x\\
-\int_\RR \left(\frac{R}{\g-1}\chi_{\x\x} + \frac{p^*}{2}\phi_{\x\x}\right)
\left[\bigg(\frac{p}{v}-\frac{p^*}{v^*}\bigg)\phi_{\x\x}
-R\bigg(\frac{1}{v}-\frac{1}{v^*}\bigg)\chi_{\x\x}\right] d\x
-\int_\RR (p-p^*) (\psi_{\x\x})^2 d\x.
\end{multline*}
We denote the terms on the right-hand side by \(L_{51}\), \(L_{61}\), \(L_{52}\), and \(L_{62}\), respectively, in order of appearance.
The terms in the first line on the right-hand side, i.e., \(L_{51}\) and \(L_{61}\), can be estimated using Young's inequality and \eqref{vxtx-1}:
\[
\abs{L_{51}} + \abs{L_{61}}
\le \n \norm{\big(\phi_{\x\x},\psi_{\x\x},\chi_{\x\x}\big)}_{L^2}^2
+C_\n(\e_1+\d_0)^2 \norm{\big(\phi_\x,\psi_\x,\chi_\x\big)}_{L^2}^2.
\]
For those in the second line, the assumption \eqref{assYap} and the smallness of the wave strengths give
\[
\abs{L_{52}} + \abs{L_{62}}
\le C(\e_1+\d_0) \norm{\big(\phi_{\x\x},\psi_{\x\x},\chi_{\x\x}\big)}_{L^2}^2.
\]
These estimates then imply that 
\[
\abs{L_5} + \abs{L_6}
\le \big(\n+C(\e_1+\d_0)\big) \norm{\big(\phi_{\x\x},\psi_{\x\x},\chi_{\x\x}\big)}_{L^2}^2
+C_\n(\e_1+\d_0)^2 \norm{\big(\phi_\x,\psi_\x,\chi_\x\big)}_{L^2}^2.
\]

\bci{L_7+L_8} An application of Young's inequality in conjunction with \eqref{first-o-2} gives
\begin{multline*}
\abs{L_7}+\abs{L_8}
\le \n \norm{\big(\phi_{\x\x},\psi_{\x\x},\chi_{\x\x}\big)}_{L^2}^2\\
+C_\n \d_0^2 \norm{\big(\phi_\x,\psi_\x,\chi_\x\big)}_{L^2}^2
+C_\n \d_0(G^R+G^S)
+C_\n \d_0\d_C\int_\RR W^2 (\phi^2+\chi^2) d\x.
\end{multline*}
The bounds established above imply that
\begin{align*}
&\frac{d}{dt} \int_\RR \left[\frac{R}{\g-1}\chi_\x \psi_{\x\x} - \frac{p^*}{2}\psi_\x \phi_{\x\x}\right]d\x
+\frac{(p^*)^2}{2v^*}\int_\RR (\phi_{\x\x})^2 d\x
+\frac{p^*}{2}\int_\RR (\psi_{\x\x})^2 d\x\\
&\le
\big(\n+C(\e_1+\d_0)\big)\norm{\big(\phi_{\x\x},\psi_{\x\x},\chi_{\x\x}\big)}_{L^2}^2
+C_\n \norm{\chi_{\x\x}}_{L^2}^2
+C_\n \norm{r_{\x\x}}_{L^2}^2\\
&\qquad
+C_\n \d_S^5 |\dot{X}|^2
+C_\n\norm{\big(Q_1, Q_2\big)}_{H^1}^2
+C_\n \d_0(G^R+G^S)
+C_\n \d_0\d_C \int_\RR W^2 (\phi^2+\chi^2) d\x\\
&\qquad
+C_\n(\e_1+\d_0)^2 \norm{\big(\phi_\x,\psi_\x,\chi_\x\big)}_{L^2}^2.
\end{align*}
Fixing \(\n=\min\left\{\frac{(p^*)^2}{8v^*},\frac{p^*}{8}\right\}>0\) and taking \(\d_0\) and \(\e_1\) sufficiently small, we deduce that
\begin{align*}
&\frac{d}{dt} \int_\RR \left[\frac{R}{\g-1}\chi_\x \psi_{\x\x} - \frac{p^*}{2}\psi_\x \phi_{\x\x}\right]d\x
+\frac{(p^*)^2}{4v^*}\int_\RR (\phi_{\x\x})^2 d\x
+\frac{p^*}{4}\int_\RR (\psi_{\x\x})^2 d\x\\
&\le C \norm{\chi_{\x\x}}_{L^2}^2
+C \norm{r_{\x\x}}_{L^2}^2
+C \d_S^5 |\dot{X}|^2
+C\norm{\big(Q_1, Q_2\big)}_{H^1}^2\\
&\qquad
+C\d_0(G^R+G^S)
+C\d_0\d_C \int_\RR W^2 (\phi^2+\chi^2) d\x
+C(\e_1+\d_0)^2 \norm{\big(\phi_\x,\psi_\x,\chi_\x\big)}_{L^2}^2.
\end{align*}
Using the upper bound of \(\norm{\chi_{\x\x}}_{L^2}^2\) in \eqref{thxx-df}, we recover the corresponding second-order diffusion terms for \(v\), \(u\) and \(\th\):
\begin{equation}\begin{aligned}\label{h2-vut}
&\frac{d}{dt} \int_\RR \left[\frac{R}{\g-1}\chi_\x \psi_{\x\x} - \frac{p^*}{2}\psi_\x \phi_{\x\x}\right]d\x
+c \norm{\big(\phi_{\x\x},\psi_{\x\x},\chi_{\x\x}\big)}_{L^2}^2\\
&\le C \norm{r}_{H^3}^2
+C \d_S^5 |\dot{X}|^2
+C\norm{\big(Q_1, Q_2, Q_3\big)}_{H^1}^2\\
&\qquad
+C\d_0(G^R+G^S)
+C\d_0\d_C \int_\RR W^2 (\phi^2+\chi^2) d\x
+C(\e_1+\d_0)^2 \norm{\big(\phi_\x,\psi_\x,\chi_\x\big)}_{L^2}^2.
\end{aligned}\end{equation}

We now combine the two estimates obtained above.
Since 
\[
\int_\RR \left[\frac{R}{\g-1}\chi_\x \psi_{\x\x} - \frac{p^*}{2}\psi_\x \phi_{\x\x}\right] d\x 
\le C\norm{\big(\phi_\x,\psi_\x,\chi_\x\big)}_{L^2}
\norm{\big(\phi_{\x\x},\psi_{\x\x},\chi_{\x\x}\big)}_{L^2},
\]
we may select a sufficiently small constant \(\n_*>0\), independent of \(\d_0\) and \(\e_1\) so that
\begin{align*}
&\int_\RR \left[
\frac{R\th}{2v^2}(\phi_{\x\x})^2
+ \frac{1}{2}(\psi_{\x\x})^2
+ \frac{R}{2(\g-1)\th}(\chi_{\x\x})^2 \right]
+ \n_*\left[\frac{R}{\g-1}\chi_\x \psi_{\x\x} - \frac{p^*}{2}\psi_\x \phi_{\x\x}\right] d\x\\
&\hspace{30mm}
\ge c\norm{\big(\phi_{\x\x},\psi_{\x\x},\chi_{\x\x}\big)}_{L^2}^2
-C\norm{\big(\phi_\x,\psi_\x,\chi_\x\big)}_{L^2}^2.
\end{align*}
We use \eqref{h2-r} and \eqref{h2-vut} to bound the time-derivative of the left-hand side.
Choosing the parameter \(\n>0\) in \eqref{h2-r} sufficiently small, we then take \(\d_0\) and \(\e_1\) small enough accordingly.
Consequently, we have
\begin{align*}
&\frac{d}{dt}\int_\RR \left[
\frac{R\th}{2v^2}(\phi_{\x\x})^2
+ \frac{1}{2}(\psi_{\x\x})^2
+ \frac{R}{2(\g-1)\th}(\chi_{\x\x})^2 \right]
+ \n_*\left[\frac{R}{\g-1}\chi_\x \psi_{\x\x} - \frac{p^*}{2}\psi_\x \phi_{\x\x}\right] d\x\\
&\qquad
+c\norm{\big(\phi_{\x\x},\psi_{\x\x},\chi_{\x\x}\big)}_{L^2}^2
+c\norm{r_{\x\x}}_{H^1}^2\\
&\le C\norm{\big(\phi_\x,\psi_\x,\chi_\x\big)}_{L^2}^2
+C\norm{r}_{H^2}^2
+C \d_S^5 |\dot{X}|^2
+C\norm{\big(Q_1, Q_2\big)}_{H^2}^2
+C\norm{Q_3}_{H^1}^2\\
&\qquad
+C\d_0(G^R+G^S)
+C\d_0\d_C \int_\RR W^2 (\phi^2+\chi^2) d\x.
\end{align*}
Integration over time then gives
\begin{equation}\begin{aligned}\label{eq:H2est}
&\sup_{\t\in[0,t]}\norm{\big(\phi_{\x\x},\psi_{\x\x},\chi_{\x\x}\big)(\t)}_{L^2}^2
+\int_0^t \left[ \norm{\big(\phi_{\x\x},\psi_{\x\x},\chi_{\x\x}\big)(\t)}_{L^2}^2
+ \norm{r_{\x\x}(\t)}_{H^1}^2 \right] d\t\\
&\le C \norm{\big(\phi,\psi,\chi\big)(0)}_{H^2}^2
+C \sup_{\t\in[0,t]}\norm{\big(\phi,\psi,\chi\big)(\t)}_{H^1}^2\\
&\qquad
+C\int_0^t \left[\d_S |\dot{X}|^2 + G^R + G^S + \norm{r}_{H^2}^2
+\norm{\big(\phi_\x,\psi_\x,\chi_\x\big)(\t)}_{L^2}^2\right] d\t\\
&\qquad
+C\d_C\int_0^t \int_\RR \big(W(\t,\x)\big)^2 (\phi^2+\chi^2) d\x d\t
+C\int_0^t \left[ \norm{(Q_1,Q_2)}_{H^2}^2 + \norm{Q_3}_{H^1}^2\right] d\t.
\end{aligned}\end{equation}
Finally, combining \Cref{lem:l2est}, \eqref{eq:h1est} and \eqref{eq:H2est} with the aid of \Cref{rmk:Q} and \Cref{lem:W}, we complete the proof of \Cref{prop:main}. \qed

\begin{appendix}
\section{Proof of Lemmas \ref{lem:EU}--\ref{lem:VS}} \label{appendix:VS}
\setcounter{equation}{0}
\subsection{Preliminaries and Proof Strategy}
The proof relies on the geometric singular perturbation theory (GSPT) \cite{Jones}.
The application of the GSPT approach in the context of fluid dynamics can be found in \cite{EE-CMS26,EEKO-JDE25} for the Brenner--Navier--Stokes--Fourier system.
Our argument follows the strategy developed in \cite{EE-CMS26,EEKO-JDE25}; however, the traveling wave equations in the present setting exhibit an additional degeneracy.
To be precise, after reducing these equations to an equation for the specific volume \(v\), we find that the coefficient of the higher-order derivative vanishes at an interior state.
To address this degeneracy, we introduce an auxiliary independent variable and derive a smooth slow--fast system to which the GSPT applies.
The details are given in the remainder of the paper.

\vspace{2mm}
\(\bullet\) \textbf{Notation:} The notation used in this appendix is independent of that introduced in Sections \ref{sec:pfofmain}--\ref{sec:h2}.
In particular, symbols appearing in both places should not be confused; for instance, \(\z\) in \eqref{eq:zetadef} is unrelated to \(\z\) used below.

\subsubsection{Fenichel's First Theorem}
We present the main concepts underlying the geometric approach.
Consider the following ODE system with a small parameter \(\d\):
\begin{equation}\label{slow}
\begin{cases}
\d x' = f(x,y,\d)\\
y' = g(x,y,\d),
\end{cases}
\end{equation}
where \(x\in\RR^n\) and \(y\in\RR^l\).
We assume that \(f\) and \(g\) are smooth on \(U\times I\), where \(U\subset \RR^{n+l}\) is open and \(I\) is an open interval containing \(0\).

Following \cite{Jones}, we call an \(l\)-dimensional manifold \(M_0 \subset U\) a \textit{critical manifold} of the system \eqref{slow} if \(f(x,y,0)=0\) for every \((x,y)\in M_0\).

We introduce the definitions needed to state Fenichel's first theorem.
\begin{definition}
We say that the manifold \(M_0\) is normally hyperbolic relative to \eqref{slow} if the \(n\times n\) matrix  
\[
D_x f(x,y, \d) \big|_{\d=0}
\]
has \(n\) eigenvalues (counting multiplicity) with nonzero real part for each \((x,y)\in M_0\).
\end{definition}

Fenichel's first theorem establishes the existence of a manifold that is locally invariant under the flow of \eqref{slow}.
To make this precise, we introduce the notion of a locally invariant manifold.

\begin{definition}
Fix \(\d \in I\setminus \{0\}\).
A set \(M \subset U\) is locally invariant under the flow of \eqref{slow} if there is an open set \(V \subset U\) containing \(M\) such that, for any \((x_0,y_0)\in M\) and any \(T\in\RR\), the trajectory \((x,y)\) starting at \((x_0,y_0)\), whenever defined on \([0,T]\), satisfies 
\[
(x,y)([0,T]) \subset V \quad \Longrightarrow \quad
(x,y)([0,T]) \subset M.
\]
When \(T<0\), replace \([0,T]\) with \([T,0]\).
\end{definition}

Note that if \(M_0\) is normally hyperbolic, the implicit function theorem allows us to express it locally as a graph of \(x\) in terms of \(y\).
In the following discussion, however, we only consider the case in which \(M_0\) can be globally expressed as a graph of \(x\) in terms of \(y\) over a compact domain \(K \subset \RR^l\).
In this setting, Fenichel's first theorem can be stated as follows.

\begin{proposition}[\cite{Jones}, Fenichel's First Theorem]\label{prop_FFT}
Consider the system \eqref{slow} with normally hyperbolic critical manifold \(M_0\).
Suppose that there is a smooth function \(h \colon K \to \RR^n\) whose graph is contained in \(M_0\), i.e., 
\[
\{(x,y) \mid x=h(y), y\in K\} \subset M_0,
\]
where \(K \subset \RR^l\) is a compact, simply connected smooth domain.
Then, for each \(r\in \mathbb{N}\), there exist a constant \(\d_0=\d_0(r)>0\), with \((-\d_0,\d_0)\subset I\), and a family of functions \(h^\d \colon K \to \RR^n\), defined for \(\abs{\d}<\d_0\), such that for each \(0<|\d|<\d_0\), the graph 
\[
\{(x,y)\mid x=h^\d(y), \, y\in K\}
\]
is contained in an \(l\)-dimensional manifold \(M_\d \subset U\) that is locally invariant under the flow of \eqref{slow}.
Moreover, the map \((y,\d)\mapsto h^\d(y)\) is \(\Ccal^r\) on \(K\times(-\d_0,\d_0)\), and \(h^0=h\) on \(K\).
\end{proposition}

\begin{remark}\label{rmk-FFT}
Fix an integer \(r\ge 2\).
Using Proposition \ref{prop_FFT}, we choose \(\d_0=\d_0(r)\) and the family \(h^\d\) so that \(h^\d(y)\) is of class \(C^r\) jointly in \(y\) and \(\d\).
Since \(h^0=h\), there exists a function 
\[
s \colon K\times(-\d_0,\d_0)\to \RR^n
\]
of class \(\Ccal^{r-1}\) jointly in \(y\) and \(\d\) such that 
\[
h^\d(y) = h(y) + \d s(y,\d), \qquad \forall (y,\d)\in K\times (-\d_0,\d_0).
\]
\end{remark}

\subsubsection{Reduction to Slow--Fast System}
The key step in the proof is to reduce the system of ODEs in \eqref{VS} to a slow--fast system of the form \eqref{slow}, to which Proposition \ref{prop_FFT} can be applied, following the strategy of \cite{EE-CMS26}.
To this end, we first carry out a formal reduction of the radiative shock equations.
More precisely, by eliminating \(u^S\), \(\th^S\) and \(q^S\), we derive a scalar ODE for \(v^S\), which can subsequently be reformulated as a slow--fast system.

At this stage, the computations in this subsubsection are only formal, since the existence and regularity of the radiative shock profile have not yet been established.
This apparent lack of rigor will be resolved in the next subsection.
There, Proposition \ref{prop_FFT} yields a solution of the reduced slow--fast system and then we use this solution to construct a solution of the original ODE system.
In this way, we obtain the desired radiative shock profile.

We integrate the traveling wave equation \eqref{VS} to find that 
\begin{equation}\label{VS-I}
\begin{cases}
-\s (v^S-v^*) - (u^S-u^*) = 0, \\
-\s (u^S-u^*) + p^S-p^* = 0, \\
-\s \left(\frac{R}{\g-1}(\th^S-\th^*)+\frac{(u^S)^2-(u^*)^2}{2}\right)
+ (p^S u^S - p^* u^*) + (q^S-q^*) =0.
\end{cases}
\end{equation}
Then, we eliminate the variables \(u^S\), \(\th^S\), and \(q^S\) to derive an equation for \(v^S\).\\
The first two equations in \eqref{VS-I} yield that 
\begin{equation}\label{uinv}
u^S = u^* -\s (v^S-v^*),
\end{equation}
and that
\[
-\s^2 (v^S -v^*) = p^S - p^* = \frac{R \th^S}{v^S} - p^*.
\]
This implies 
\begin{equation}\label{thinv}
\th^S = \frac{v^S}{R} \big(p^* - \s^2 (v^S-v^*)\big).
\end{equation}
Then, it follows that 
\begin{align*}
\th^S - \th^*
&= \frac{v^S}{R} \left(\frac{R\th^*}{v^*}-\frac{R\th^*}{v^S}-\s^2(v^S-v^*)\right)\\
&= \frac{(p^*-\s^2 v^S)}{R}(v^S-v^*)
= \frac{(p^*-\s^2 v^*)}{R}(v^S-v^*) - \frac{\s^2}{R}(v^S-v^*)^2,
\end{align*}
from which, we obtain 
\[
\frac{R}{\g-1}(\th^S - \th^*)
=\frac{(p^*-\s^2 v^*)}{\g-1}(v^S-v^*)
-\frac{\s^2}{\g-1}(v^S-v^*)^2.
\]
On the other hand, using the first two equations in \eqref{VS-I}, we see that 
\[
\frac{1}{2}(u^S)^2-\frac{1}{2}(u^*)^2
=\frac{1}{2}(u^S-u^*)^2 + u^*(u^S-u^*)
=\frac{\s^2}{2}(v^S-v^*)^2 - \s u^*(v^S-v^*)
\]
and that
\begin{align*}
p^S u^S - p^* u^*
&=(p^S-p^*)(u^S-u^*)+(p^S-p^*)u^*+p^*(u^S-u^*)\\
&=\s^3(v^S-v^*)^2 -\s^2 u^*(v^S-v^*) -\s p^* (v^S-v^*).
\end{align*}
Thus, summing up all, the third equation in \eqref{VS-I} yields that
\begin{align*}
q^S
&=\s \left(\frac{R}{\g-1}(\th^S-\th^*)+\frac{(u^S)^2-(u^*)^2}{2}\right)
-(p^S u^S - p^* u^*)\\
&= \left(\frac{\s(p^*-\s^2 v^*)}{\g-1}-\s^2 u^*
+\s^2 u^*
+\s p^*\right)(v^S-v^*)
-\s^3 \left(\frac{1}{\g-1}-\frac{1}{2}+1\right)(v^S-v^*)^2\\
&= \s \frac{(\g p^*-\s^2 v^*)}{\g-1}(v^S-v^*)
-\s^3 \frac{(\g+1)}{2(\g-1)}(v^S-v^*)^2.
\end{align*}
Note by \cite[(3.4)]{EEKO-JDE25} that 
\begin{equation}\label{gpsv}
\g p^* -\s^2 v^* = \frac{(\g+1)}{2}\s^2 \d_S = \frac{(\g+1)}{2}\s^2 (v_+ - v^*).
\end{equation}
Then, we have 
\[
q^S
= \frac{(\g+1)\s^3}{2(\g-1)}
\Big((v_+ - v^*)(v^S-v^*) - (v^S-v^*)^2\Big)
= \frac{(\g+1)\s^3}{2(\g-1)} (v_+ - v^S)(v^S-v^*).
\]
Let \(Q\) be the constant coefficient above, i.e., 
\[
Q\coloneqq \frac{(\g+1)\s^3}{2(\g-1)}.
\]
It follows that
\begin{equation}\label{qinv}
q^S
= Q (v_+ - v^S)(v^S-v^*).
\end{equation}
Note that once \(v^S\) is constructed, the remaining components are recovered from \eqref{uinv}, \eqref{thinv} and \eqref{qinv}, and the first three equations in \eqref{VS} are then automatically satisfied.
To proceed further, we also observe from \eqref{uinv}, \eqref{thinv} and \eqref{qinv} that 
\begin{equation}\label{tpqp}
(\th^S)' = \frac{(v^S)'}{R} \big[p^* - \s^2(2v^S-v^*)\big], \qquad
(q^S)' = Q (v^S)' \big[2(v^*-v^S)+(v_+-v^*)\big].
\end{equation}

\vspace{2mm}
We introduce the normalized profile \(w_0\) (and \(w_1\)) and the rescaled spatial variable \(z\) by
\begin{equation}\label{var}
z \coloneqq \d_S \x, \qquad
v^S(\x) = v^* + \d_S w_0(z), \qquad
w_1(z) = \frac{dw_0}{dz}.
\end{equation}
Then, it simply follows that 
\[
(v^S)'(\x) = \d_S^2 w_1(z), \qquad
(q^S)'(\x) = Q \d_S^3 w_1 (1-2w_0).
\]
Let \(\O\) be an open neighborhood of \([0,1]\). We define the following functions on \(\O\):
\begin{equation}\label{def-TDBF}
\begin{aligned}
\Theta_{\d_S}(w) &\coloneqq \frac{(v^*+\d_S w)}{R} (p^* -\s^2 \d_S w),
&D_{\d_S}(w) &\coloneqq \frac{Q(1-2w)}{v^*+\d_S w},\\
B_{\d_S}(w) &\coloneqq \frac{4\b}{R} \big(\Theta_{\d_S}(w)\big)^3 \big[p^* - \s^2(v^*+2\d_S w)\big],
&F_{\d_S}(w) &\coloneqq \a Q (v^*+\d_S w) w(1-w).
\end{aligned}
\end{equation}
We remark that, in addition to the explicit \(\d_S\)-dependence displayed in the functions, the shock speed \(\s\) and the constant \(Q\) also depend on \(\d_S\).
With the functions above, we obtain 
\[
\th^S = \Theta_{\d_S}(w_0), \qquad
q^S = Q \d_S^2 w_0 (1-w_0), \qquad
\b((\th^S)^4)' = \d_S^2 B_{\d_S}(w_0) w_1,
\]
and 
\[
\frac{(q^S)'}{v^S} = \d_S^3 D_{\d_S}(w_0) w_1.
\]
The fourth equation in \eqref{VS} 
\[
- \left(\frac{(q^S)'}{v^S}\right)'
+ \a v^S q^S 
+ \b ((\th^S)^4)' = 0
\]
reduces to 
\begin{equation}\label{sinsys}
\begin{cases}
\dfrac{dw_0}{dz} = w_1,\\
\d_S^2 D_{\d_S}(w_0) \dfrac{dw_1}{dz}
= F_{\d_S}(w_0) + B_{\d_S}(w_0)w_1
-\d_S^2 D_{\d_S}'(w_0) w_1^2.
\end{cases}
\end{equation}
It is therefore natural to take 
\[
\begin{cases}
\dfrac{dw_0}{dz} = w_1,\\
\d_S^2 \dfrac{dw_1}{dz}
= \big(D_{\d_S}(w_0)\big)^{-1} \Big(F_{\d_S}(w_0) + B_{\d_S}(w_0)w_1
-\d_S^2 D_{\d_S}'(w_0) w_1^2 \Big).
\end{cases}
\]
as the slow--fast system.
However, since \(w_0=\frac{1}{2}\) corresponds to \(v^S = \frac{v^*+v_+}{2}\) and
\[
D_{\d_S}\Big(\frac{1}{2}\Big) = 0, \qquad
D_{\d_S}'(w) = -\frac{Q(2v^*+\d_S)}{(v^*+\d_S w)^2} < 0,
\]
that is, the denominator vanishes at \(w_0=\frac{1}{2}\), the right-hand side fails to define a smooth function.
For this purpose, we consider an auxiliary variable \(\z\) defined by \(\frac{dz}{d\z}=\d_S D_{\d_S}(w_0)\).
Then, the system \eqref{sinsys} can be rewritten as 
\begin{equation}\label{sys}
\begin{cases}
\dfrac{d\wtil_0}{d\z} = \d_S D_{\d_S}(\wtil_0) \wtil_1,\\
\d_S \dfrac{d\wtil_1}{d\z}
= F_{\d_S}(\wtil_0) + B_{\d_S}(\wtil_0)\wtil_1 -\d_S^2 D_{\d_S}'(\wtil_0) \wtil_1^2,
\end{cases}
\end{equation}
where we define
\[
\wtil_0(\z) \coloneqq w_0(z(\z)), \qquad 
\wtil_1(\z) \coloneqq w_1(z(\z)).
\]
The right-hand side is now smooth, including at \(w_0=\frac{1}{2}\) and \(\d_S = 0\).

\vspace{2mm}
In what follows, we will apply Proposition \ref{prop_FFT} to the slow--fast system \eqref{sys} and thereby obtain a solution of this reduced system.
This solution will then be used to construct a solution of the original system \eqref{VS}, yielding the desired radiative shock.
In other words, although the computations carried out in the present subsection were merely formal and cannot yet be rigorously justified at this stage, their sole purpose is to identify the appropriate slow-fast system.
A rigorous construction of the radiative shock will be given in the next subsection.

\subsection{Proof of Lemma \ref{lem:EU}: Existence and Uniqueness}
First of all, we establish the existence of the radiative shock, and then we prove its uniqueness.
As mentioned earlier, the existence result relies on Fenichel's first theorem, stated as Proposition \ref{prop_FFT}, whereas uniqueness follows from standard ODE theory.

\subsubsection{Proof of Lemma \ref{lem:EU}: Existence}
We now prove the existence of a solution to the ODE system \eqref{VS}.
To this end, we first apply Proposition \ref{prop_FFT} to find an invariant manifold for the slow--fast system \eqref{sys}, from which we obtain a solution to this reduced system.
We then use this solution to construct a solution to the original system \eqref{VS}.

\vspace{2mm}
First of all, we observe that the right-hand sides of the system \eqref{sys} are smooth for \(w\) in a fixed small neighborhood \(\O\) of \([0,1]\) and for \(\d_S\) in a sufficiently small two-sided neighborhood of zero.
In fact, the shock strength \(\d_S\) is always positive.
However, the geometric theory requires the functions to be defined smoothly for \(\d\) in a small neighborhood of zero, including negative values of \(\d_S\).
For this reason, we define \(\d_S\)-dependence of the traveling speed \(\s\) by 
\[
\s = \s_{\d_S} = \sqrt{\frac{\g p^*}{v^* + \frac{\g+1}{2}\d_S}}
\]
This definition coincides with the actual shock speed for \(\d_S>0\), while providing a well-defined smooth extension to negative values of \(\d_S\).
We refer to \cite[(3.3)]{EEKO-JDE25} for a justification of this choice of \(\s\).

The limits of the coefficients and functions are as follows:
\begin{align*}
Q^* &= \frac{(\g+1)(\s^*)^3}{2(\g-1)},
&\Theta_0(w) &= \th^* >0,\\
D_0(w) &= \frac{Q^*}{v^*}(1-2w),
&F_0(w) &= \a Q^* v^* w(1-w).
\end{align*}
Using \eqref{gpsv}, we also find that 
\[
B_0(w) = \frac{4\b}{R}(\th^*)^3 \big(p^* -(\s^*)^2 v^*\big)
= - \frac{4\b(\g-1)p^*(\th^*)^3}{R} < 0.
\]
Thus, in the small neighborhood \(\O\) of \([0,1]\), the functions 
\[
-B_{\d_S}(w), \qquad
v^*+\d_S w, \qquad
\Theta_{\d_S}(w)
\]
are positive and bounded away from zero.
The critical manifold \(M_0\) of \eqref{sys} satisfies that
\[
F_0(\wtil_0) + B_0(\wtil_0)\wtil_1 =0, \qquad \forall (\wtil_0,\wtil_1)\in M_0,
\]
and so it can be characterized as follows: 
\[
M_0 = \{(\wtil_0,\wtil_1)\mid \wtil_1 = A^* \wtil_0(1-\wtil_0)\},
\]
where 
\[
A^* =\frac{\a R \g (\g+1)\s^*}{8\b(\g-1)^2(\th^*)^3}>0.
\]
This coincides with \eqref{eq:A^*def}.
Moreover, \(M_0\) is normally hyperbolic relative to \eqref{sys}, since 
\[
\rd_{\wtil_1} \big[F_{\d_S}(\wtil_0) + B_{\d_S}(\wtil_0)\wtil_1 -\d_S^2 D_{\d_S}'(\wtil_0) \wtil_1^2\big] \Big|_{\d_S=0}
= B_0 < 0.
\]

We are now ready to apply Proposition \ref{prop_FFT}.
We first fix an integer \(m\ge 2\) and choose a compact interval \(K\subset\O\) such that \([0,1]\subset \operatorname{int} K\).
Using Proposition \ref{prop_FFT} (and Remark \ref{rmk-FFT}) and taking \(r=m+1\), we find that for small enough \(\d_S\), there exists a one-dimensional locally invariant manifold \(M_{\d_S}\) for \eqref{sys} together with the following graph:
\begin{equation}\label{graph}
h^{\d_S}(w)
= A^* w(1-w)
+ \d_S \cdot s(w,\d_S), \qquad w\in K,
\end{equation}
where \(s\) is a \(\Ccal^m\) function jointly in \(w\) and \(\d_S\).
In particular, \(h^{\d_S} \to h^0\) in \(\Ccal^1(K)\) when \(\d_S\to0\).

\vspace{2mm}
For any \(a\in[0,1]\), the Cauchy--Lipschitz theorem gives a local solution \((\wtil_0,\wtil_1)\) of \eqref{sys} with initial datum \((\wtil_0,\wtil_1)(0)=(a,h^{\d_S}(a))\).
By the local invariance of \(M_{\d_S}\), this solution satisfies \(\wtil_1=h^{\d_S}(\wtil_0)\) near \(\z=0\).
Then, differentiating this identity and using \eqref{sys} at \(\z=0\), we find that 
\begin{equation}\label{h-eqn}
\d_S^2 \frac{d h^{\d_S}}{d\wtil_0}(\wtil_0) D_{\d_S}(\wtil_0) h^{\d_S}(\wtil_0)
=F_{\d_S}(\wtil_0) + B_{\d_S}(\wtil_0) h^{\d_S}(\wtil_0)
-\d_S^2 \frac{d D_{\d_S}}{d\wtil_0}(\wtil_0) \big(h^{\d_S}(\wtil_0)\big)^2.
\end{equation}
This holds for all \(\wtil_0 \in[0,1]\), since \(a\in[0,1]\) was arbitrary.
Then, we obtain that 
\[
h^{\d_S}(\wtil_0)
=\frac{F_{\d_S}(\wtil_0)}{-B_{\d_S}(\wtil_0)
+ \d_S^2 \frac{d}{d\wtil_0}\Big(D_{\d_S}(\wtil_0) h^{\d_S}(\wtil_0)\Big)}.
\]
Thanks to the \(\Ccal^1\) bounds on \(D_{\d_S}\) and \(h^{\d_S}\) together with \(B_0<0\), the denominator is positive and bounded away from zero for small enough \(\d_S\).
We further observe from the definition of \(F_{\d_S}\) in \eqref{def-TDBF} that 
\[
h^{\d_S}(0) = h^{\d_S}(1) = 0
\]
and that \(h^{\d_S}\) is positive on the interval \((0,1)\).

\vspace{2mm}
We now consider the following initial value problem:
\begin{equation}\label{hivp}
\frac{dw_0}{dz}=h^{\d_S}(w_0), \qquad
w_0(0)= \frac{1}{2},
\end{equation}
and set \(w_1(z)\coloneqq h^{\d_S}(w_0(z))\).
Note that since \(0\) and \(1\) are equilibria of \eqref{hivp}, neither of them can be reached in finite time.
By standard ODE theory (see, for instance, \cite[Lemma 3.1]{EEKO-JDE25}), the local solution by the Cauchy--Lipschitz theorem extends to all \(z\in\RR\) and stays within \((0,1)\).
Moreover, it is strictly increasing, and the absence of any other zero of \(h^{\d_S}\) in \([0,1]\) yields that 
\[
\lim_{z\to-\infty}(w_0,w_1)(z) = (0,0), \qquad
\lim_{z\to+\infty}(w_0,w_1)(z) = (1,0).
\]
The regularity of \(h^{\d_S}\) also implies that \((w_0,w_1)\) is of class \(\Ccal^m\).

\vspace{2mm}
We next verify that \((w_0,w_1)\) solves \eqref{sinsys}.
Fix \(z_0 \neq 0\), so that \(w_0(z_0)\neq \frac{1}{2}\), and take the local solution of \eqref{sys} subject to the initial datum \((\wtil_0,\wtil_1)(0)=(w_0(z_0),w_1(z_0))\).
Define \(z=z(\z)\) locally by 
\[
\frac{dz}{d\z}=\d_S D_{\d_S}(\wtil_0(\z)), \qquad
z(0) = z_0.
\]
Since \(\frac{dz}{d\z}\neq 0\) at \(\z=0\), this change of variable has a local inverse \(\z=\z(z)\).
Thanks to the locally invariant manifold \(M_{\d_S}\), we have \(\wtil_1 = h^{\d_S}(\wtil_0)\), and thus in the variable \(z\), the function \(\wtil_0\) satisfies the same ODE as \(w_0\), that is, \eqref{hivp}, with the same value at \(z_0\).
Then, uniqueness implies that \(\wtil_0(\z(z))=w_0\) and \(\wtil_1(\z(z))=w_1\) hold locally.
Hence, \((w_0,w_1)\) satisfies \eqref{sinsys} at every \(z_0 \neq 0\).
On the other hand, at \(z=0\), we use the graph relation instead of the change of variable.
Indeed, \eqref{h-eqn} yields that whenever \(w\neq\frac{1}{2}\),
\[
\frac{F_{\d_S}(w)+B_{\d_S}(w)h^{\d_S}(w)
-\d_S^2 D_{\d_S}'(w)\big(h^{\d_S}(w)\big)^2}{\d_S^2 D_{\d_S}(w)}
=(h^{\d_S})'(w) h^{\d_S}(w),
\]
whose right-hand side is continuous.
Thus, the apparent singularity is removable along the graph, together with the uniquely determined extension \(w_1'(0)=(h^{\d_S})'\big(\frac{1}{2}\big) h^{\d_S}\big(\frac{1}{2}\big)\).
Since \((w_0,w_1)\) is \(\Ccal^m\), \eqref{sinsys} also holds at \(z=0\).

Note that \((\frac{1}{2}, h^{\d_S}(\frac{1}{2}))\) is itself a critical point of \eqref{sys}.
In contrast with \cite{EE-CMS26,EEKO-JDE25}, a single trajectory of the auxiliary system on \(M_{\d_S}\) cannot connect the two endpoints across this additional critical point.
The construction above uses the graph of the locally invariant manifold, rather than a global reparametrization of one trajectory.

\vspace{2mm}
Finally, in accordance with \eqref{var}, we define 
\[
v^S(\x) \coloneqq v^* + \d_S w_0(\d_S \x).
\]
Then, based on \eqref{uinv}, \eqref{thinv} and \eqref{qinv}, we also define
\begin{align*}
u^S = u^* -\s (v^S-v^*), \qquad
\th^S = \frac{v^S}{R} \big(p^* - \s^2 (v^S-v^*)\big), \qquad
q^S = Q (v_+ - v^S)(v^S-v^*).
\end{align*}
The first three equations of \eqref{VS} follow immediately and these functions are \(\Ccal^m\) together with \(v^S,\th^S,q^S>0\).
Since \eqref{sinsys} is now justified for every \(z\in\RR\), the fourth equation also holds.
Moreover, \eqref{tpqp} yields that for small enough \(\d_S>0\),
\[
(v^S)' = \d_S^2 w_1>0, \qquad
(u^S)' = -\s (v^S)' <0, \qquad
(\th^S)' = \frac{p^*-\s^2(2v^S-v^*)}{R}(v^S)' < 0.
\]
This completes the construction of the radiative shock in the \(3\)-shock case. \qed

\subsubsection{Proof of Lemma \ref{lem:EU}: Uniqueness}
We now prove uniqueness, up to a translation, among monotone \(\Ccal^2\) profiles.
As in the existence proof, we consider the \(3\)-shock case, with \(\d_S = v_+ -v^*>0\), and work in the original variable \(\x\).

\vspace{2mm}
Let \((v^S,u^S,\th^S,q^S)\) be any such profile.
We first note that the integrated system \eqref{VS-I} gives \eqref{thinv} and \eqref{qinv} as follows: 
\[
\th^S = \frac{v^S}{R} \big(p^* - \s^2 (v^S-v^*)\big), \qquad
q^S = Q (v_+ - v^S)(v^S-v^*).
\]
Then we set
\[
y^S \coloneqq \frac{(q^S)'}{v^S} - \b (\th^S)^4, \qquad
v_c \coloneqq \frac{v^*+v_+}{2}.
\]
The fourth equation of \eqref{VS} is now written as \((y^S)' = \a v^S q^S\). 
Thus, away from \(v^S = v_c\), the profile satisfies the following system: 
\begin{equation}\label{vy-sys}
\begin{cases}
(v^S)' = \dfrac{v^S\big(y^S+\b(\th^S)^4\big)}{Q(v^*+v_+ -2v^S)},\\
(y^S)' = \a Q v^S (v_+-v^S)(v^S-v^*),
\end{cases}
\qquad \th^S = \frac{v^S}{R} \big(p^* - \s^2 (v^S-v^*)\big),
\end{equation}
whose right-hand side is smooth near each of the two critical points 
\[
P^* \coloneqq \big(v^*, -\b(\th^*)^4\big), \qquad
P_+ \coloneqq \big(v_+, -\b(\th_+)^4\big).
\]
Since \((y^S)'= \a v^S q^S \ge 0\), the identity \((q^S)'=v^S(y^S+\b(\th^S)^4)\) and the far-field conditions in \eqref{VS} imply 
\[
\lim_{\x\to-\infty}(v^S,y^S)(\x) = P^*, \qquad
\lim_{\x\to+\infty}(v^S,y^S)(\x) = P_+.
\]

We then consider the linearizations of \eqref{vy-sys}.
Using 
\[
\frac{d\th^S}{dv^S} = \frac{p^* - \s^2(2v^S-v^*)}{R}, \qquad
p_+ = p^* -\s^2 \d_S,
\]
the Jacobian matrices at \(P^*\) and \(P_+\) are respectively
\begin{equation}\label{jacobian}
J^* = \begin{pmatrix}
\frac{4\b v^* (\th^*)^3}{R Q \d_S} (p^*-\s^2 v^*) & \frac{v^*}{Q\d_S}\\
\a Q v^* \d_S & 0
\end{pmatrix}, \qquad
J_+ = \begin{pmatrix}
-\frac{4\b v_+ (\th_+)^3}{R Q \d_S} (p_+ - \s^2 v_+) & -\frac{v_+}{Q \d_S}\\
-\a Q v_+ \d_S & 0
\end{pmatrix}.
\end{equation}
The determinants of these matrices are given by 
\[
\det J^* = -\a (v^*)^2 <0, \qquad
\det J_+ = -\a (v_+)^2 <0,
\]
and thus both critical points are saddles.
In particular, the unstable manifold of \(P^*\) and the stable manifold of \(P_+\) are one-dimensional.
Note that monotonicity selects the unstable branch at \(P^*\) with \(v^S > v^*\) and the stable branch at \(P_+\) with \(v^S<v_+\).

\vspace{2mm}
We now compare two monotone \(\Ccal^2\) profiles \((v^S,u^S,\th^S,q^S)\) and \(\vhat,\uhat,\thhat,\qhat\), and we define \(\yhat\) in the same way as \(y^S\).
Since \((v^S)'=(q^S)'=\vhat'=\qhat'=0\) at each end state, neither profile can attain the end state at a finite point.
In addition, each profile attains \(v_c\) exactly once due to the monotonicity.
Thus, we may suppose that \(v^S(0)=\vhat(0)=v_c\).

Near \(P^*\), both trajectories lie on the same one-dimensional unstable branch with \(v>v^*\).
We choose \(\x_1,\x_2<0\) such that \((v^S,y^S)(\x_1)=(\vhat,\yhat)(\x_2)\).
Since the system is autonomous, the translated trajectories \((v^S,y^S)(\x+\x_1)\) and \((\vhat,\yhat)(\x+\x_2)\) solve the same initial value problem at \(\x=0\).
By the Cauchy--Lipschitz theorem, they coincide on their common interval, where both \(v\)-components remain below \(v_c\).
By continuity, their first arrival times at \(v_c\) must coincide.
Since each profile attains \(v_c\) only at \(\x=0\), these times must be \(-\x_1\) and \(-\x_2\), respectively.
Thus, it follows that \(\x_1=\x_2\) and 
\[
(v^S,y^S)(\x) = (\vhat, \yhat)(\x), \qquad \x<0.
\]
Similarly, near \(P_+\), both trajectories lie on the same one-dimensional stable branch with \(v<v_+\).
The same argument backwards in \(\x\) yields 
\[
(v^S,y^S)(\x) = (\vhat, \yhat)(\x), \qquad \x>0.
\]
This establishes the desired uniqueness of the radiative shocks. \qed

\subsection{Proof of Lemma \ref{lem:VS}: Quantitative Estimates}
Throughout this subsection, \(C\), \(c\) and \(C_k\) denote positive constants independent of \(\d_S\), whose values may change from line to line.
We again consider a \(3\)-shock and follow the order of estimates in \cite{EEKO-JDE25}: derivative bounds, comparison of components, and then exponential decay.
The estimates for the \(q\)-variable are proved last.

By uniqueness, the normalized profile coincides with the one constructed in the existence proof.
We recall the construction of the radiative shock solution, in particular, \eqref{var}: 
\[
z =\d_S \x, \qquad
v^S(\x) = v^* + \d_S w_0(z), \qquad
\rd_z w_0 = w_1 = h^{\d_S}(w_0),
\]
where \(0<w_0<1\), \(w_0(0)=\frac{1}{2}\), and \(w_1>0\).
Moreover, the function \(h^{\d_S}\) can be written as 
\[
h^{\d_S}(w_0)
=A^* w_0(1-w_0) + \d_S s(w_0,\d_S),
\]
for some \(\Ccal^m\) function \(s\), and it satisfies \(h^{\d_S}(0)=h^{\d_S}(1)=0\).
For the constant \(\d_4\) in Lemma \ref{lem:EU}, taken smaller if necessary, the following holds: 
\[
\sup_{0<\d_S<\d_4} \|h^{\d_S}\|_{\Ccal^3([0,1])} \le C.
\]
In particular, \(0<w_1\le C\).
Differentiating the graph relation \(w_1 = h^{\d_S}(w_0)\) yields 
\begin{align*}
\rd_z w_1
&= (\rd_{w_0} h^{\d_S}) w_1,\\
\rd_z^2 w_1
&= (\rd_{w_0}^2 h^{\d_S}) w_1^2
+ (\rd_{w_0} h^{\d_S}) \rd_z w_1,\\
\rd_z^3 w_1
&= (\rd_{w_0}^3 h^{\d_S}) w_1^3
+ 3(\rd_{w_0}^2 h^{\d_S}) w_1 \rd_z w_1
+ (\rd_{w_0} h^{\d_S}) \rd_z^2 w_1,
\end{align*}
where all derivatives of \(h^{\d_S}\) are evaluated at \(w_0(z)\).
It follows successively that 
\[
|\rd_z^k w_1| \le C_k w_1, \qquad \text{for } k=1,2,3.
\]
Then, since \(\rd_\x^k v^S = \d_S^{k+1} \rd_z^{k-1}w_1\), we obtain 
\begin{equation}\label{v-der-bd}
0 < (v^S)' \le C\d_S^2, \qquad
\big|\rd_\x^k v^S \big|
\le C_k \d_S^{k-1} (v^S)', \qquad \text{for } k=2,3,4.
\end{equation}
Moreover, from \eqref{uinv} and \eqref{thinv}, we observe
\[
\rd_\x^k u^S = -\s \rd_\x^k v^S, \qquad
\rd_\x^k \th^S
= \frac{p^*-\s^2(2v^S-v^*)}{R} \rd_\x^k v^S
-\frac{\s^2}{R}\sum_{l=1}^{k-1} \binom{k}{l} (\rd_\x^l v^S) (\rd_\x^{k-l} v^S).
\]
All of the coefficients are uniformly bounded.
Thus, taking \(\d_4 <1\), we obtain 
\begin{equation}\label{high-der}
\big|\rd_\x^k (v^S,u^S,\th^S)\big|
\le C_k \d_S^{k-1} \big|(v^S)'\big|, \qquad k=2,3,4.
\end{equation}

Next, we show \eqref{ratio-vu} and \eqref{ratio-vth}.
It is well established that \eqref{sm1} holds (see also \eqref{gpsv}).
Then, \eqref{uinv} and \eqref{tpqp} simply yield \eqref{ratio-vu} and \eqref{ratio-vth}.
Note that this part is considerably simpler than the corresponding analysis in the previous work \cite{EE-CMS26,EEKO-JDE25} for the Brenner--Navier--Stokes--Fourier system, since the first two equations contain no dissipative terms.
Moreover, for small enough \(\d_S\), they yield \eqref{derivsign} and 
\[
|(v^S)'| \sim |(u^S)'| \sim |(\th^S)'|.
\]

We now prove \eqref{tail-0}.
Using \eqref{qinv}, we find that 
\[
\left(\frac{(q^S)'}{v^S}\right)'
= Q \left[
\frac{v^* + v_+ - 2v^S}{v^S}(v^S)''
-\frac{v^*+v_+}{(v^S)^2}\big((v^S)'\big)^2
\right].
\]
Note that \(|v^*+v_+ -2v^S|\le \d_S\) always holds.
Thus, \eqref{v-der-bd} implies that
\begin{equation}\label{rad-minor}
\left|\left(\frac{(q^S)'}{v^S}\right)'\right|
\le C\big(\d_S |(v^S)''|
+ |(v^S)'|^2 \big)
\le C \d_S^2 (v^S)'.
\end{equation}
On the other hand, together with \eqref{qinv}--\eqref{tpqp}, the last equation of \eqref{VS} can be written as
\begin{equation}\label{trans-4eq}
(v_+-v^S)(v^S-v^*)
= \frac{4\b(\th^S)^3 [\s^2(2v^S-v^*)-p^*]}{\a R Q v^S} (v^S)'
+ \frac{1}{\a Q v^S}\left(\frac{(q^S)'}{v^S}\right)'.
\end{equation}
Since 
\[
|v^S - v^*|
+|\th^S - \th^*|
+|\s - \s^*|
+|Q - Q^*|
\le C\d_S,
\]
the coefficient of \((v^S)'\) in \eqref{trans-4eq} differs from \(\frac{1}{A^*}\) by at most \(C\d_S\).
Combining this with \eqref{rad-minor}, we obtain 
\[
\abs{(v_+-v^S)(v^S-v^*)-\frac{1}{A^*}(v^S)'}
\le C\d_S (v^S)'.
\]
Since \(A^*>0\), this implies that 
\begin{equation}\label{vcomp}
c(v_+-v^S)(v^S-v^*)
\le (v^S)'
\le C(v_+-v^S)(v^S-v^*).
\end{equation}
These two estimates yield \eqref{tail-0}.

To deduce the exponential decay, we first note that \(v_+-v^S \ge \frac{\d_S}{2}\) on \(\x\le0\), whereas we have \(v^S-v^*\ge\frac{\d_S}{2}\) on \(\x\ge0\).
Thus, \eqref{vcomp} implies that 
\[
(v^S-v^*)' \ge c\d_S(v^S-v^*),\quad \forall \x\le0, \qquad
(v_+-v^S)' \le -c\d_S(v_+-v^S),\quad \forall \x\ge0.
\]
The standard Gr\"onwall argument (see for instance, \cite[Lemma 2.1]{KV-JEMS21}) together with \eqref{uinv} and \eqref{thinv} yields the desired exponential decay \eqref{decay}.
Then, substitution of \eqref{decay} into \eqref{vcomp} gives 
\[
0< (v^S)'(\x)
\le C\d_S^2 e^{-c\d_S|\x|}.
\]
The ratio between the first derivatives yields \eqref{derivdecay}, and \eqref{high-der} completes \eqref{2bound1}.

\vspace{2mm}
It remains to show the estimates for \(q^S\).
The exact identity \eqref{qS} is already given by \eqref{qinv}.
Since \(Q\) is positive and bounded away from zero, \eqref{vcomp} directly yields \eqref{qder}.
Then, for \(j=1\), \(2\), \(3\), differentiating \eqref{qinv} implies 
\[
\rd_\x^j q^S
=Q\left[
\big(v^*+v_+-2v^S\big)\rd_\x^j v^S
-\sum_{l=1}^{j-1}\binom{j}{l} (\rd_\x^l v^S)
(\rd_\x^{j-l} v^S)
\right],
\]
where the sum is empty when \(j=1\).
Using \eqref{v-der-bd} and \(|v^*+v_+ -2v^S|\le \d_S\), we find that 
\[
|\rd_\x^j q^S(\x)|
\le C_j \d_S^j (v^S)'(\x)
\le C_j \d_S^{j+2} e^{-c\d_S|\x|}, \qquad j=1,2,3.
\]
This completes the proof of Lemma \ref{lem:VS}. \qed
\end{appendix}

\vspace{2mm}
\noindent \textbf{Conflict of interest.} The authors declare no conflict of interest.

\vspace{2mm}
\noindent \textbf{Data availability statement.} The authors do not analyze or generate any datasets, as this work is purely theoretical and mathematical.

\vspace{2mm}
\noindent \textbf{Declaration of the use of AI.}
During the preparation of this manuscript, the authors used OpenAI's ChatGPT to assist with language editing as well as the identification of typographical errors and potential inconsistencies in notation and formatting. The final draft was given to GPT-6 to review for mathematical accuracy. This resulted in the detection of one minor mathematical issue, which was simple to repair without the use of AI. No essential mathematical ideas and concepts were provided by AI. All text (both mathematical and non-mathematical) was first written by the authors. All AI-assisted suggestions were carefully reviewed by the authors, who take full responsibility for the content and accuracy of the manuscript.

\bibliography{reference}

\end{document}

%% file: newmacros.tex
\usepackage{amssymb,mathrsfs,mathtools,graphicx,enumerate,color,colortbl}
\usepackage[hidelinks]{hyperref}
\usepackage{cleveref}
\usepackage{comment}

\newcommand{\Ccal}{\mathcal{C}}

\newcommand{\Hcal}{\mathcal{H}}

\newcommand{\Jcal}{\mathcal{J}}

\newcommand{\Ocal}{\mathcal{O}}

\newcommand{\Abar}{\bar{A}}

\newcommand{\Ubar}{\bar{U}}

\renewcommand{\hbar}{\bar{h}}

\newcommand{\pbar}{\bar{p}}
\newcommand{\qbar}{\bar{q}}

\newcommand{\ubar}{\bar{u}}
\newcommand{\vbar}{\bar{v}}

\newcommand{\thbar}{\bar{\th}}

\newcommand{\qtil}{\tilde{q}}

\newcommand{\util}{\tilde{u}}
\newcommand{\vtil}{\tilde{v}}
\newcommand{\wtil}{\tilde{w}}

\newcommand{\thtil}{\tilde{\th}}

\renewcommand{\a}{\alpha}
\renewcommand{\b}{\beta}
\newcommand{\g}{\gamma}
\renewcommand{\d}{\delta}

\newcommand{\z}{\zeta}
\newcommand{\et}{\eta}
\renewcommand{\th}{\theta}

\renewcommand{\l}{\lambda}

\newcommand{\n}{\nu}
\newcommand{\x}{\xi}

\newcommand{\s}{\sigma}
\renewcommand{\t}{\tau}

\renewcommand{\O}{\Omega}

\newcommand{\e}{\varepsilon}

\newcommand{\RR}{{\mathbb R}}

\newcommand{\abs}[1]{\left|#1\right|}

\newcommand{\norm}[1]{\left\|#1\right\|}

\newcommand{\Phif}[2]{\Phi\Big(\frac{#1}{#2}\Big)}

\newcommand{\bpf}[1]{\noindent\textbf{Proof of \eqref{#1}:}} 
\newcommand{\bci}[1]{\noindent\textbf{Control of \(#1\):}}

\newcommand{\vhat}{\widehat{v}}
\newcommand{\uhat}{\widehat{u}}
\newcommand{\thhat}{\widehat{\theta}}
\newcommand{\qhat}{\widehat{q}}
\newcommand{\yhat}{\widehat{y}}

\definecolor{black}{rgb}{0.0, 0.0, 0.0}
\definecolor{red}{rgb}{1.0, 0.5, 0.5}

\newcommand{\rd}{\partial}